\documentclass[12pt]{article}
\usepackage{amsmath,amssymb,amsthm}
\usepackage[frenchb,english]{babel}
\usepackage[latin1]{inputenc}
\usepackage{graphics}
\usepackage{graphicx}
\usepackage{epsfig}
\usepackage{amsfonts}
\usepackage{amssymb}
\usepackage{latexsym}
\usepackage{amscd}
\usepackage{color}
\usepackage{shadow}
\usepackage{a4wide}
\usepackage[all]{xy}

\usepackage{url}

\numberwithin{equation}{section}

\normalsize \makeatletter

\newtheorem{thm}{Theorem}[section]
\newtheorem{prop}[thm]{Proposition}

\newtheorem{lemma}[thm]{Lemma}

\theoremstyle{remark}
\newtheorem{assum}{Assumption}
\newtheorem{rmq}[thm]{Remark}

\newtheorem{dfn}[thm]{Definition}

\title{Counterexamples to the Strichartz inequalities for the wave equation in domains II}

\author{Oana Ivanovici  
 \\Johns Hopkins University\\ Department of Mathematics, Baltimore MD 21218\\
ivanovici@math.jhu.edu }
\date{ }

\begin{document}
\bigskip
\bigskip
\bigskip
\bigskip
\maketitle

\begin{abstract}
In this paper we consider a smooth and bounded domain $\Omega\subset\mathbb{R}^d$ of dimension $d\geq 2$ with smooth boundary $\partial\Omega$ and we   construct sequences of solutions to the wave equation with Dirichlet boundary conditions which contradict the Strichartz estimates of the free space, providing losses of derivatives at least for a subset of the usual range of indices. This is due to micro-local phenomena such as caustics generated in arbitrarily small time near the boundary. 

The result we obtain here is a generalization of the paper \cite{doi} where we provided a counterexample to the optimal Strichartz estimates for the wave equation in the particular case when the manifold is the Friedlander's model domain of dimension $2$.  The key tool which allows to generalize the result of \cite{doi} is Melrose's equivalence of glancing hypersurfaces theorem, together with a subtle reduction to the two-dimensional case. 
\end{abstract}

2000 Mathematics Subject Classification: 35L20, 58J30, 58J32.


\section{Introduction}
Let $\Omega$ be a smooth manifold of dimension $d\geq 2$ with $C^{\infty}$ boundary $\partial\Omega$, equipped with a Riemannian metric $g$. Let $\Delta_{g}$ be the Laplace-Beltrami operator associated to $g$ on $\Omega$, acting on $L^{2}(\Omega)$ with Dirichlet boundary condition. Let $0<T<\infty$ and consider the wave equation with Dirichlet boundary conditions:
\begin{equation}\label{unde}
\left\{
      \begin{array}{ll}
      (\partial^{2}_{t}-\Delta_{g})u=0\ \mathrm{on} \ \Omega\times [0,T],\\
      u|_{t=0}=u_{0},\quad \partial_{t}u|_{t=0}=u_{1},\\
      u|_{\partial\Omega}=0.
      \end{array}
      \right.
\end{equation}
Strichartz estimates are a family of dispersive estimates on solutions $u:\Omega\times [0,T]\rightarrow\mathbb{C}$ to the wave equation \eqref{unde}. In their most general form, local Strichartz estimates state that
\begin{equation}\label{strichartz}
\|u\|_{L^{q}([0,T],L^{r}(\Omega))}\leq C(\|u_{0}\|_{\dot{H}^{\gamma}(\Omega)}+\|u_{1}\|_{\dot{H}^{\gamma-1}}),
\end{equation}
where $\dot{H}^{\gamma}(\Omega)$ denotes the homogeneous Sobolev space over $\Omega$ and where the pair $(q,r)$ is wave admissible in dimension $d$, i.e. it satisfies $2\leq q\leq\infty$, $2\leq r<\infty$ and moreover
\begin{equation}\label{adpair}
\frac{1}{q}+\frac{d}{r}=\frac{d}{2}-\gamma, \quad \frac{2}{q}+\frac{d-1}{r}\leq\frac{d-1}{2}.
\end{equation}
When equality holds in \eqref{adpair} the pair $(q,r)$ is called sharp wave admissible in dimension $d$. Estimates involving $r=\infty$ hold when $(q,r,d)\neq (2,\infty,3)$, but typically require the use of Besov spaces.

Our main result is work in the opposite direction. Roughly speaking, we show that if $\Omega\subset\mathbb{R}^{d}$ is a smooth and bounded domain of $\mathbb{R}^d$ and $(q,r)$ is a sharp wave-admissible pair in dimension $d\geq 2$ with $r>4$, then there exists $T=T(\Omega)$ such that the quotient
\[
\frac{\|u\|_{L^{q}([0,T],L^{r}(\Omega))}}{\|u_{0}\|_{\dot{H}^{\gamma+\frac{1}{6}(\frac{1}{4}-\frac{1}{r})}(\Omega)}+\|u_{1}\|_{\dot{H}^{\gamma+\frac{1}{6}(\frac{1}{4}-\frac{1}{r})-1}(\Omega)}}
\]
takes arbitrarily large values for suitable initial data $(u_{0},u_{1})$, i.e. a loss of at least $\frac{1}{6}(\frac{1}{4}-\frac{1}{r})$ derivatives is unavoidable in the Strichartz estimates for the wave flow.
\vskip5mm

The main motivation for the above types of Strichartz estimates comes from
applications to harmonic analysis and the study of nonlinear dispersive equations.
Estimates like \eqref{strichartz}  can be used to prove existence
theorems for nonlinear wave equations. 

In $\mathbb{R}^{d}$ and for $g_{ij}=\delta_{ij}$, Strichartz estimates in the context of the wave and Schr\"{o}dinger equations have a long history, beginning with Strichartz pioneering work \cite{stri77}, where he proved the particular case $q=r$ for the wave and (classical) Schr\"{o}dinger equation. This was later generalized to mixed $L^{q}((-T,T),L^{r}(\Omega))$ norms by Ginibre and Velo \cite{give85} for Schr\"{o}dinger equation, where $(q,r)$ is sharp admissible and $q>2$; the wave estimates were obtained independently by Ginibre-Velo \cite{give95} and Lindblad-Sogge \cite{ls95}, following earlier work by Kapitanski \cite{lev90}. The remaining endpoints for both equations  were finally settled by Keel and Tao \cite{keta98}. In that case $\gamma=\frac{(d+1)}{2}(\frac{1}{2}-\frac{1}{r})$ and one can obtain a global estimate with $T=\infty$; (see also Kato \cite{ka87}, Cazenave-Weissler \cite{cawe90}). 

However, for general manifolds phenomena such as trapped geodesics or finiteness of volume can preclude the development of global estimates, leading us to consider local in time estimates.

In the variable coefficients case, even without boundary, the situation is much more complicated: we simply recall here the pioneering work of Staffilani and Tataru \cite{stta02}, dealing with compact, non trapping perturbations of the flat metric and recent work of Bouclet and Tzvetkov \cite{botz} in the context of Schrodinger equation, which considerably weakens the decay of the perturbation (retaining the non trapping character at spatial infinity). On compact manifolds  without boundary, due to the finite speed of propagation, it is enough to work in coordinate charts and to establish local Strichartz estimates for variable coefficients wave operators in $\mathbb{R}^{d}$: we recall here the works by Kapitanski \cite{kap91} and Mockenhaupt, Seeger and Sogge \cite{moseso} in the case of smooth coefficients when one can use the Lax parametrix construction to obtain the appropriate dispersive estimates. In the case of $C^{1,1}$ coefficients, Strichartz estimates were shown in the works by Smith \cite{sm98} and by Tataru \cite{tat02}, the latter work establishing the full range of local estimates; here the lack of smoothness prevents the use of Fourier integral operators and instead wave packets and coherent state methods are used to construct parametrices for the wave operator.
In these situations, if the metric is sufficiently smooth (if it has at least two derivatives bounded), the Strichartz estimates hold as in the Euclidian case.

Let us recall the result for the flat space: if we denote by $\Delta$ the Euclidian Laplace operator, then
the Strichartz estimates for the wave equation posed on $\mathbb{R}^{d}$ read as follows (see \cite{keta98}):
\begin{prop}\label{proprd}
Let $(q,r)$ be a wave admissible pair in dimension $d\geq2$. If $u$ satisfies 
\begin{equation}
(\partial^{2}_{t}-\Delta)u=0,\quad 
[0,T]\times\mathbb{R}^{d},\quad u|_{t=0}=u_{0},\quad
\partial_{t}u|_{t=0}=u_{1}
\end{equation}
for some $0<T<\infty$, $u_{0}, u_{1}\in C^{\infty}(\mathbb{R}^{d})$, then there is a constant $C=C_{T}$ such that
\begin{equation}\label{striw}
\|u\|_{L^{q}([0,T],L^{r}(\mathbb{R}^{d}))}\leq C(\|u_{0}\|_{\dot{H}^{\frac{(d+1)}{2}(\frac{1}{2}-\frac{1}{r})}(\mathbb{R}^{d})}+\|u_{1}\|_{\dot{H}^{\frac{(d+1)}{2}(\frac{1}{2}-\frac{1}{r})-1}(\mathbb{R}^{d})}).
\end{equation}
\end{prop}

Even though the boundaryless case has been well understood for some time,
obtaining results for the case of manifolds with boundary has been surprisingly elusive.

For a manifold with smooth, strictly geodesically concave boundary (i.e. for which the second fundamental form is strictly negative definite), the Melrose and Taylor parametrix yields the Strichartz estimates for the wave equation with Dirichlet boundary condition for the range of exponents in \eqref{adpair} (not including the endpoints) as shown in the paper of Smith and Sogge \cite{smso95}. If the concavity assumption is removed, however, the presence of multiply reflecting geodesic and their limits, the gliding rays, prevent the construction of a similar parametrix!

Note that on an exterior domain a source point does not generate caustics and that the presence of caustics generated in small time near a source point is the one which makes things difficult inside a strictly convex set. 

Recently, Burq, Lebeau and Planchon \cite{bulepl07}, \cite{bupl07} established Strichartz type inequalities on a manifold with boundary using the $L^{r}(\Omega)$ estimates for the spectral projectors obtained by Smith and Sogge \cite{smso06}. The range of triples $(q,r,\gamma)$ that can be obtained in this manner, however, is restricted by the allowed range of $r$ in the square function estimate for the wave equation, which controls the norm of $u$ in the space $L^{r}(\Omega,L^{2}(-T,T))$ (see \cite{smso06}). In dimension $3$, for example, this restricts the indices to $q,r \geq 5$. The work of Blair, Smith and Sogge \cite{blsmso08} expands the range of indices $q$ and $r$ obtained in \cite{bulepl07}: specifically, they show that if $\Omega$ is a compact manifold with boundary (or without boundary but with Lipschitz metric $g$) and $(q,r,\gamma)$ is a triple satisfying the first condition in \eqref{adpair} together with the restriction 
\[
\left\{
      \begin{array}{ll}
      \frac{3}{q}+\frac{d-1}{r}\leq\frac{d-1}{2},\quad d\leq 4\\
      \frac{1}{q}+\frac{1}{r}\leq\frac{1}{2},\quad d\geq 4,
      \end{array}
      \right.
\]
then the Strichartz estimates \eqref{strichartz} hold true for solutions $u$ to \eqref{unde} satisfying Dirichlet or Neumann homogeneous boundary conditions, with a constant $C$ depending on $\Omega$ and $T$. 
\vskip5mm

In this paper we prove that Strichartz estimates for the wave equation inside the domain $\Omega$ suffer losses when compared to the usual case $\mathbb{R}^{d}$, at least for a subset of the usual range of indices, under the assumption that there exists a point in $T^{*}\partial\Omega$ where the second fundamental form on the boundary of the manifold has a strictly positive eigenfunction. Precisely, out assumption reads as follows:

\begin{assum}\label{assumomega} Let $\Omega$ be a smooth manifold of dimension $d\geq 2$ with $C^\infty$ boundary $\partial\Omega$. We assume that there exists a bicharacteristic that intersects $\partial\Omega\times\mathbb{R}$ tangentially at some point $(\rho_{0},\vartheta_{0})\in T^{*}(\partial\Omega\times\mathbb{R})$ having exactly second order contact with the boundary at $(\rho_{0},\vartheta_{0})$ and which remains in the complementary of $\bar{\Omega}\times\mathbb{R}$.
We call the point $(\rho_{0},\vartheta_{0})$ a \emph{gliding point}.
\end{assum}
\begin{rmq}\label{rmqas1}
In particular, any smooth and bounded domain of $\mathbb{R}^d$, $d\geq 2$,  satisfies the condition in the Assumption \ref{assumomega}.
\end{rmq}

Our main result reads as follows:
\begin{thm}\label{thmstrichartz}
Let $(\Omega,g)$ satisfy the conditions in Assumption \ref{assumomega} and with $d\in\{2,3,4\}$. 
Then there exists $T=T(\Omega)\in(0,\infty)$ and for every small $\epsilon>0$ there exist sequences
$V_{h,j,\epsilon}\in C^{\infty}(\bar{\Omega})$, $j=\overline{0,1}$ such that the solution $V_{h,\epsilon}$ to the wave equation with
Dirichlet boundary conditions
\begin{equation}\label{wavvv}
\left\{
\begin{array}{ll}
(\partial^{2}_{t}-\Delta_{g})V_{h,\epsilon}=0,\\ 
V_{h,\epsilon}|_{t=0}=V_{h,0,\epsilon},\quad
\partial_{t}V_{h,\epsilon}|_{t=0}=V_{h,1,\epsilon},\\
V_{h,\epsilon}|_{\partial\Omega\times[0,T]}=0,
\end{array}
\right.
\end{equation}
satisfies
\begin{equation}\label{ince}
\sup_{\epsilon>0,h\in(0,1],j} h^{-\frac{(d+1)}{2}(\frac{1}{2}-\frac{1}{r})-\frac{1}{6}(\frac{1}{4}-\frac{1}{r})+2\epsilon+j}\|V_{h,j,\epsilon}\|_{{L}^{2}(\Omega)}\leq 1
\end{equation}
and
\begin{equation}\label{ince2}
\lim_{h\rightarrow 0}\|V_{h,\epsilon}\|_{L^{q}_{t}([0,1],L^{r}(\Omega))}=\infty,
\end{equation}
for every sharp wave admissible pair $(q,r)$ in dimension $d$ with $r>4$.
Moreover $V_{h,\epsilon}$ has compact support for the normal variable in a neighborhood of the boundary of size $h^{\frac{1-\epsilon}{2}}$ and is well localized at spatial frequency $\frac{1}{h}$ in the tangential variable.
\end{thm}
\begin{rmq}
Notice that Theorem \ref{thmstrichartz} shows an explicit loss of at least $\frac{1}{6}(\frac{1}{4}-\frac{1}{r})$ derivatives in the Strichartz estimates for domains $\Omega$ satisfying the Assumption \ref{assumomega} compared to the Euclidian case (see Proposition \ref{proprd}). 
\end{rmq}
\begin{rmq}\label{rmqd}
The proof of Theorem \ref{thmstrichartz} will show that the restriction on the dimension comes only from the fact that for $d\geq 5$ all admissible pairs $(q,r)$ satisfy $r\leq 4$. 
\end{rmq}
\begin{rmq}
From the Remarks \ref{rmqas1} and \ref{rmqd} it follows that Theorem \ref{thmstrichartz} holds true if $\Omega$ is any smooth, bounded domain of $\mathbb{R}^d$ with $d\in \{2,3,4\}$.
\end{rmq}
\begin{rmq}
In \cite{doi} we proved Theorem \ref{thmstrichartz} in the particular case of the two-dimensional half-space $\{(x,y)| x>0, y\in\mathbb{R}\}$ with Laplace operator given by $\partial^2_x+(1+x)\partial^2_y$. We notice that the half-space together with the metric inherited from the above Laplace operator becomes a strictly convex domain (called the Friedlander's model domain). 

In this paper we generalize the result of \cite{doi} to any smooth domain satisfying the Assumption \ref{assumomega} using the Melrose's Theorem of glancing surfaces.
\end{rmq}
\begin{rmq}
Notice that Theorem \ref{thmstrichartz} states for instance that the scale-invariant  Strichartz estimates fail for $\frac{3}{q}+\frac{1}{r}>\frac{15}{24}$, whereas the result of Blair, Smith and Sogge states that such estimates hold if $\frac{3}{q}+\frac{1}{r}\leq\frac{1}{2}$. Of course, the counterexample places a lower bound on the loss for such indices $(q,r)$, and the work \cite{blsmso08} would place some upper bounds, but this concise statement shows one explicit gap in our knowledge that remains to be filled.

A very interesting and natural question would be to determine the sharp range of exponents for \eqref{strichartz} in any dimension $d\geq 2$! 
\end{rmq}

A classical way to prove Strichartz inequalities is to use dispersive estimates: the fact that weakened dispersive estimates can \emph{still} imply optimal (and scale invariant) Strichartz estimates for the solution of the wave equation was first noticed by Lebeau: in \cite{gle06} he proved that a loss of  derivatives is unavoidable for the wave equation inside a strictly convex domain, and this appears because of swallowtail type caustics in the wave front set of $u$:
\[
|\chi(hD_{t})u(t,x)|\lesssim h^{-d}\min(1,(h/t)^{\frac{d-2}{2}+\frac{1}{4}}).
\]
However, these estimates, although optimal for the dispersion, imply Strichartz type inequalities without losses, but with indices $(q,r,d)$ satisfying 
\[
\frac{1}{q}\leq(\frac{d-2}{2}+\frac 14)(\frac{1}{2}-\frac{1}{r}).
\]
A natural strategy for proving Theorem \ref{thmstrichartz} would be to use the Rayleigh whispering gallery modes which accumulate their energy near the boundary contributing to large $L^{r}(\Omega)$ norms. Applying the semi-classical Schr\"{o}dinger evolution shows that a loss of derivatives is necessary for the Strichartz estimates. However, when dealing with the wave operator this strategy fails as the gallery modes satisfy the Strichartz estimates of the free space, as it is shown in \cite{doi}. 

In the proof of Theorem \ref{thmstrichartz} we shall proceed in a different manner, using co-normal waves with multiply reflected cusps at the boundary, together with Melrose's Theorem of glancing rays to reduce the study of the iterated boundary operators to the Friedlander case, in which case all the computations are explicit. We only recall here the main ingredients of the proof given in \cite{doi} and show how this can be used to construct a counterexample under the much more general assumptions of Theorem \ref{thmstrichartz}. The reduction to the model case relies essentially on Melrose's Theorem \cite{meta87} of glancing surfaces.

The organization of the paper is as follows: in Section \ref{secred} we show that in order to prove Theorem \ref{thmstrichartz} it is enough to consider the two-dimensional case. In Section \ref{secdoide} we recall the construction in the model case of the strictly convex domain of dimension two we dealt with in \cite{doi} and use it to determine an approximate solution of \eqref{wavvv} which satisfies Theorem \ref{thmstrichartz}. In the Appendix we compute the $L^{r}$ norms of a wave with a cusp type singularity.

\section*{Acknowledgements}
The author would like to thank Gilles Lebeau who suggested this problem to her and Nicolas Burq for many helpful discussions on the subject. She would also like to thank Michael Taylor for having sent her the manuscript "Boundary problems  for the wave equations with grazing and gliding rays".

\section{Reduction to the two dimensional case}\label{secred}
Let $\Omega$ satisfy the assumptions of Theorem \ref{thmstrichartz}. Write local coordinates on $\Omega$ as $(x,y_{1},..,y_{d-1})$ with $x> 0$ on $\Omega$, $\partial\Omega=\{(0,y)|y=(y_{1},..,y_{d-1})\in\mathbb{R}^{d-1}\}$ and local coordinates induced by the product $X=\Omega\times\mathbb{R}_{t}$, as $(x,y,t)$. 

Local coordinates on the base induce local coordinates on the cotangent bundle, namely $(\rho,\vartheta)=(x,y,t,\xi,\eta,\tau)$ on $T^{*}X$ near $\pi^{-1}(q)$, $q\in T^{*}\partial X$, where $\pi:T^{*}X\rightarrow ^{b}T^{*}X$ is the canonical inclusion from the cotangent bundle into the $b$-cotangent bundle defined by $^{b}T^{*}X=T^{*}\mathring{X}\cup T^{*}\partial X$. The corresponding local coordinates on the boundary are denoted $(y,t,\eta,\tau)$ (on a neighborhood of a point $q$ in $T^{*}\partial X$). The metric function in $T^{*}\Omega$ has the form
\[
g(x,y,\xi,\eta)=A(x,y)\xi^{2}+2\sum_{j=1}^{d-1}C_{j}(x,y)\xi\eta_{j}+\sum_{j,k=1}^{d-1}B_{j,k}(x,y)\eta_{j}\eta_{k},
\]
with $A$, $B_{j,k}$, $C_{j}$ smooth. Moreover, these coordinates can be chosen so that $A(x,y)=1$ and $C_{j}(x,y)=0$ (see \cite[Appendix C]{hormand}). Thus, in this coordinates chart the metric on the boundary writes
\[
g(0,y,\xi,\eta)=\xi^{2}+\sum_{j,k=1}^{d-1}B_{j,k}(0,y)\eta_{j}\eta_{k}.
\]
On $T^{*}\partial\Omega$ the metric $g$ takes even a simpler form, since introducing geodesic coordinates we can assume moreover that, locally,
\[
B_{1,1}(0,y)=1,\quad B_{1,j}(0,y)=0\quad \forall j\in\{2,..,d-1\}.
\]
Hence, if we write $R(x,y,\eta):=\sum_{j,k=1}^{d-1}B_{j,k}(x,y)\eta_{j}\eta_{k}$, then for small $x$ we have
\begin{multline}\label{formetric}
R(x,y,\eta)=
(1+x\partial_{x}B_{1,1}(0,y_{1},y'))\eta^{2}_{1}\\+\sum_{j=1}^{d-1}(x\partial_{x}B_{1,j}(0,y)+O(x^{2}))\eta_{1}\eta_{j}+\sum_{j,k=2}^{d-1}B_{j,k}(x,y)\eta_{j}\eta_{k}.
\end{multline}
The Assumption \ref{assumomega} on the domain $\Omega$ is equivalent to saying that there exists a point $(0,y_{0},\xi_{0},\eta_{0})$ on $T^{*}\Omega$ where the boundary is microlocally strictly convex, i.e. that there exists a bicharacteristic passing through this point that intersects $\partial\Omega$ tangentially having exactly second order contact with the boundary and remaining in the complement of $\partial\bar{\Omega}$. If $p\in C^{\infty}(T^{*}X\setminus o)$ (where we write $o$ for the "zero" section) denotes the principal symbol of the wave operator $\partial^{2}_{t}-\Delta_{g}$, this last condition translates into
\begin{equation}\label{glid11}
\tau^{2}=R(0,y_{0},\eta_{0}),\quad \{p,x\}= \frac{\partial p}{\partial\xi}=2\xi_{0}=0,
\end{equation}
\begin{equation}\label{glid12}
\{\{p,x\},p\}=\{\frac{\partial p}{\partial\xi},p\}=2\partial_{x}R(0,y_{0},\eta_{0})>0,
\end{equation}
where $\{f_{1},f_{2}\}$ denotes the Poisson bracket  
\[
\{f_{1},f_{2}\}=\frac{\partial f_{1}}{\partial\vartheta}\frac{\partial f_{2}}{\partial\rho}-\frac{\partial f_{1}}{\partial\rho}\frac{\partial f_{2}}{\partial\vartheta}.
\]
Denote the gliding point (in $T^{*}\Omega\times\mathbb{R}$) provided by the Assumption \ref{assumomega} by
\[
(\rho_{0},\vartheta_{0})=(0,y_{0},0,0,\eta_{0},\tau_{0}=-\sqrt{R(0,y_{0},\eta_{0})}).
\] 
We start the proof of Theorem \ref{thmstrichartz} by reducing the problem to the study of the two dimensional case. Consider the following assumptions:
\begin{assum}\label{assumthms2}
Let $(\tilde{\Omega},\tilde{g})$ be a smooth manifold of dimension $2$ with $C^{\infty}$ boundary and Riemannian metric $\tilde{g}$. Assume that in a chart of local coordinates 
$\tilde{\Omega}=\{(x,\tilde{y})|x>0,\tilde{y}\in\mathbb{R}\}$
and that the Laplace-Beltrami operator associated to $\tilde{g}$ is given by
\[
\partial^{2}_{x}+(1+xb(\tilde{y}))\partial^{2}_{\tilde{y}},
\]
where $b(\tilde{y})$ is a smooth function. Suppose in addition that there exists a bicaracteristic which intersects the boundary tangentially at $(0,\tilde{y}_{0},\tilde{\xi}_{0},\tilde{\eta}_{0})\in T^{*}\tilde{\Omega}$ having exactly second order contact with the boundary and which remains in the complementary of $\overline{\tilde{\Omega}}$. This is equivalent to saying that at $(0,\tilde{y}_{0},\tilde{\xi}_{0},\tilde{\eta}_{0})$ the following holds
\[
\tilde{\xi}_{0}=0,\quad 2b(\tilde{y}_{0})>0.
\]
We suppose in addition (without loss of generality, since we can always rescale the normal variable $x$) that $b(\tilde{y}_{0})=1$.
\end{assum}
\begin{thm}\label{thms2}
Let $(\tilde{\Omega},\tilde{g})$ satisfy the conditions in Assumption \ref{assumthms2}. Then there exists $T=T(\tilde{\Omega})\in (0,\infty)$ and for every $\epsilon>0$ small enough there exist sequences $\tilde{V}_{h,j,\epsilon}$, $j\in\{0,1\}$ and approximate solutions $\tilde{V}_{h,\epsilon}$ to the wave equation on $\tilde{\Omega}$ with Dirichlet boundary condition 
\begin{equation}\label{undered}
\left\{
      \begin{array}{ll}
\partial^{2}_{t}V-\partial^{2}_{x}V-(1+xb(\tilde{y}))\partial^{2}_{\tilde{y}}V=0,\quad\text{on}\quad \tilde{\Omega}\times\mathbb{R}\\
V|_{t=0}=\tilde{V}_{h,0,\epsilon},\quad
\partial_{t}V|_{t=0}=\tilde{V}_{h,1,\epsilon},\\
V|_{\partial\tilde{\Omega}\times[0,T]}=0,
     \end{array}
     \right.
\end{equation}
which satisfy the following conditions:
\begin{itemize}
\item First, $\tilde{V}_{h,\epsilon}$ is an approximate solution to \eqref{undered} in the sense that
\begin{equation}
\partial^{2}_{t}\tilde{V}_{h,\epsilon}-\partial^{2}_{x}\tilde{V}_{h,\epsilon}-(1+xb(\tilde{y}))\partial^{2}_{\tilde{y}}\tilde{V}_{h,\epsilon}=O_{L^{2}(\tilde{\Omega})}(1/h),\quad \|\tilde{V}_{h,\epsilon}\|_{L^{2}(\tilde{\Omega})}\leq 1.
\end{equation}

\item Secondly, $\tilde{V}_{h,\epsilon}$ writes as a sum
\begin{equation}\label{hred}
 \tilde{V}_{h,\epsilon}(x,\tilde{y},t)=\sum_{n=0}^{N}v^{n}_{h,\epsilon}(x,\tilde{y},t),
\end{equation}
where $1\leq N\simeq h^{-\frac{(1-\epsilon)}{4}}$ and 
where the functions $v^{n}_{h,\epsilon}(x,\tilde{y},t)$ satisfy the following conditions:
\begin{itemize}
\item  for $4<r<\infty$:
\begin{equation}\label{hest}     
\left\{
\begin{array}{ll}
\|v^{n}_{h,\epsilon}(.,t)\|_{L^{r}(\tilde{\Omega})}\geq C h^{-\frac{3}{2}(\frac{1}{2}-\frac{1}{r})-\frac{1}{6}(\frac{1}{4}-\frac{1}{r})+2\epsilon},\\
\sup_{\epsilon>0}\Big(\|v^{n}_{h,\epsilon}(.,t)\|_{L^{2}(\tilde{\Omega})}+h\|\partial_{t}v^{n}_{h,\epsilon}(.,t)\|_{L^{2}(\tilde{\Omega})}\Big)\leq 1,
\end{array}
\right.
\end{equation}
where the constant $C=C(T)>0$ is independent of $h$, $\epsilon$ and $n$;

\item $v^{n}_{h,\epsilon}(x,\tilde{y},t)$ are essentially supported for the time variable $t$ in almost disjoint intervals of time $I_n$ satisfying $|I_0|\simeq |I_n|\simeq  h^{\frac{(1-\epsilon)}{4}}$ for all $n\in\{0,..,N\}$, and also supported for the tangential variable $\tilde{y}$ in almost disjoint intervals.

\item $\tilde{V}_{h,\epsilon}$ are supported for the normal variable $0\leq x \lesssim h^{(1-\epsilon)/2}$ (where the respective constants in the inequality depend only on $\tilde{\Omega}$) and localized at spatial frequency $\frac{1}{h}$ in the tangential variable $\tilde{y}$. Moreover we have, uniformly in $\epsilon>0$,
\begin{equation}\label{hnormes}
\sup_{\epsilon>0} \|\tilde{V}_{h,\epsilon}\|_{L^{2}(\tilde{\Omega})}\lesssim 1,\quad \sup_{\epsilon>0}\|\partial_{\tilde{y}}\tilde{V}_{h,\epsilon}\|_{L^{2}(\tilde{\Omega})}\lesssim\frac{1}{h},\quad  \sup_{\epsilon>0}\|\partial^{2}_{\tilde{y}}\tilde{V}_{h,\epsilon}\|_{L^{2}(\tilde{\Omega})}\lesssim\frac{1}{h^{2}}.
\end{equation}
\end{itemize}
\end{itemize}
\end{thm}
In the rest of this section we show how Theorem \ref{thms2} implies Theorem \ref{thmstrichartz}. Assume we have proved Theorem \ref{thms2}. 
Let $(\Omega,g)$ be a Riemannian manifold of dimension $d>2$ satisfying the assumptions of Theorem \ref{thmstrichartz} and let $(0,y_{0},\xi_{0},\eta_{0})\in T^{*}\Omega$ be a point satisfying \eqref{glid11}, \eqref{glid12}. From \eqref{formetric} it follows that local coordinates can be chosen such that $y_{0}=0\in\mathbb{R}^{d-1}$, $\eta_{0}=(1,0,..,0)\in\mathbb{R}^{d-1}$ and such that the Laplace-Beltrami operator $\Delta_{g}$ be given by
\begin{equation}\label{deflaplags}
\Delta_{g}=\partial^{2}_{x}+\sum_{j,k=1}^{d-1}B_{j,k}(x,y)\partial_{j}\partial_{k},
\end{equation}
where for $x$ small enough
\[
B_{1,1}(x,y)=1+x\partial_{x}B_{1,1}(0,y)+O(x^{2}),\quad \partial_{x}B_{1,1}(0,y)>0
\]
and for $j\in\{2,..,d-1\}$ we have $B_{1,j}(0,y)=0$. By performing a rescaling of the normal variable $x$, we can assume without loosing generality that $\partial_{x}B_{1,1}(0,0)=1$.

We can now define $\tilde{\Omega}$, locally in a neighborhood of $(x=0,y_1=0,\xi=0,\eta_1=0)$, to be the two dimensional half-space $\tilde{\Omega}:=\{(x,y_{1})|x>0,y_{1}\in\mathbb{R}\}$  equipped with the metric 
\[
\tilde{g}(x,y_{1},\xi,\eta_{1}):=\xi^{2}+(1+xb(y_{1}))\eta^{2}_{1},\quad b(y_{1}):=\partial_{x}B_{1,1}(0,y_{1},0).
\]
Recall that we assumed $b(0)=1$.
Applying Theorem \ref{thms2} near $(0,y_{1}=0,0,\eta_{1}=1)\in T^{*}\tilde{\Omega}$ we obtain, for $\epsilon>0$ small enough, sequences $\tilde{V}_{h,\epsilon,j}$, $j\in\{0,1\}$ such that the solution $\tilde{V}_{h,\epsilon}$ to \eqref{undered} satisfies \eqref{hred}, \eqref{hest} and \eqref{hnormes}. Let $\chi\in C^{\infty}_{0}(\mathbb{R}^{d-2})$ be a cut-off function supported in the coordinate chart such that $\chi=1$ in a neighborhood of $0\in\mathbb{R}^{d-2}$ and
for $j\in\{0,1\}$ set
\begin{equation}\label{forvhej}
V_{h,\epsilon,j}(x,y_{1},y'):=h^{-(d-2)/4}\tilde{V}_{h,\epsilon/3,j}(x,y_{1})e^{-\frac{|y'|^{2}}{2h}}\chi(y').
\end{equation}
\begin{prop}\label{propcontr}
The solution $V_{h,\epsilon}$ to the wave equation \eqref{wavvv} with Dirichlet boundary condition where $\Delta_{g}$ is given by \eqref{deflaplags} and with initial data $(V_{h,\epsilon,0},V_{h,\epsilon,1})$ defined in \eqref{forvhej} satisfies \eqref{ince}, \eqref{ince2}.
\end{prop}
\begin{rmq} Notice that Proposition \ref{propcontr} implies immediately Theorem \ref{thmstrichartz}.
\end{rmq}
\begin{proof}
We proceed by contradiction.
Let $(q,r)$ be a sharp wave admissible pair in dimension $d\in\{2,3,4\}$ with $r>4$ and set
\[
\beta(r,d)=\frac{(d+1)}{2}(\frac{1}{2}-\frac{1}{r})+\frac{1}{6}(\frac{1}{4}-\frac{1}{r}).
\]
We suppose to the contrary that the operator
\[
\sin (t\sqrt{-\Delta_{g}}):L^{2}(\Omega)\rightarrow L^{q}([0,T],L^{r}(\Omega))
\]
is bounded by $h^{-\beta(r,d)+2\epsilon}$, where $T=T(\tilde{\Omega})$ is given by Theorem \ref{thms2}. Let $\tilde{V}_{h,\epsilon/3}$ be the approximate solution to \eqref{undered} with initial data $(\tilde{V}_{h,\epsilon/3,j})_{j=0,1}$ satisfying all the conditions in Theorem \ref{thms2}. For $t\in [0,T]$ we define
\[
W_{h,\epsilon}(x,y,t):=h^{-(d-2)/4}\tilde{V}_{h,\epsilon/3}(x,y_{1},t)e^{-\frac{|y'|^{2}}{2h}}\chi(y').
\]
\begin{lemma}\label{lemestw}
There exists a constant $c(T)>0$ independent of $h$ such that $W_{h,\epsilon}$ satisfies
\begin{equation}\label{estwred}
\|W_{h,\epsilon}\|_{L^{q}([0,T],L^{r}(\Omega))}\geq c(T) h^{-\beta(r,d)+2\epsilon/3},
\end{equation}
\begin{equation}
\|W_{h,\epsilon}|_{t=0}\|_{L^{2}(\Omega)}+h\|\partial_{t}W_{h,\epsilon}|_{t=0}\|_{L^{2}(\Omega)}\lesssim 1.
\end{equation}
\end{lemma}
\begin{proof}
Indeed, using the special form of $\tilde{V}_{h,\epsilon}$ provided by Theorem \ref{thms2} we can estimate
\begin{align}
\nonumber
\|W_{h,\epsilon}\|^{q}_{L^{q}([0,T],L^{r}(\Omega))} & =\int_{0}^{T}\|W_{h,\epsilon}\|^{q}_{L^{r}(\Omega)}dt \\
\nonumber
& =\Big(\int_{0}^{T}\|\sum_{n=0}^{N}v^{n}_{h,\epsilon/3}\|^{q}_{L^{r}(\tilde{\Omega})}dt\Big) \times \|h^{-(d-2)/4}e^{-\frac{|y'|^{2}}{2h}}\chi(y')\|^{q}_{L^{r}(\mathbb{R}^{d-2})}\\
\nonumber
& \geq ch^{-\frac{q(d-2)}{2}(\frac{1}{2}-\frac{1}{r})}\sum_{k\leq N}\int_{t\in I_{k}}\|\sum_{n=0}^{N}v^{n}_{h,\epsilon/3}\|^{q}_{L^{r}(\tilde{\Omega})}dt +O(h^{\infty})\\
\nonumber
& \simeq ch^{-\frac{q(d-2)}{2}(\frac{1}{2}-\frac{1}{r})}\sum_{k\leq N}|I_{k}|\|v^{0}_{h,\epsilon/3}\|^{q}_{L^{r}(\tilde{\Omega})} +O(h^{\infty})\\
\nonumber
& \simeq cT h^{-\frac{q(d-2)}{2}(\frac{1}{2}-\frac{1}{r})}\|v^{0}_{h,\epsilon/3}\|^{q}_{L^{r}(\tilde{\Omega})}+O(h^{\infty})\\ 
& \geq cT h^{(-\beta(r,d)+2\epsilon/3)q}.
\end{align}
We used here the fact that each $v^n_{h,\epsilon/3}$ provided by Theorem \ref{thms2} is essentially supported in time in  an interval $I_n$ of size $1/N$ and that $(I_n)_{n\in\{0,..,N\}}$ are almost disjoint. Take $c(T)=(c T)^{1/q}$, where $c$ is the bound from below of the integral in the $d-2$ tangential variables. 

To estimate the $L^{2}(\Omega)$ norm we use again the fact that $v^{n}_{h,\epsilon}$ and its time derivative have disjoint essential supports in the tangential variable $y_{1}$. For $W_{h,\epsilon}(.,0)$ we have, for instance
\[
\|W_{h,\epsilon}|_{t=0}\|_{L^{2}(\Omega)}=\|\tilde{V}_{h,\epsilon/3,0}\|_{L^{2}(x,y_{1})}\|h^{-(d-2)/4}e^{-\frac{|y'|^{2}}{2h}}\chi(y')\|_{L^{2}(\mathbb{R}^{d-2})}\lesssim 1.
\]
\end{proof}
Let $V_{h,\epsilon}$ be the solution to the wave equation \eqref{wavvv} with initial data $(V_{h,\epsilon,j})_{j=0,1}$ and write 
\[
V_{h,\epsilon}=W_{h,\epsilon}+w_{h,\epsilon,err}.
\] 
If we denote $\Delta_{\tilde{g}}=\partial^{2}_{x}+(1+xb(y_{1}))\partial^{2}_{y_{1}}$, $\square_{\tilde{g}}=\partial^{2}_{t}-\Delta_{\tilde{g}}$, then $W_{h,\epsilon}$ solves 
\begin{align}
\left\{
\begin{array}{ll}
\square_{\tilde{g}} W_{h,\epsilon} & =\square_{\tilde{g}}\tilde{V}_{h,\epsilon/3} h^{-(d-2)/4}e^{-\frac{|y'|^{2}}{2h}}\chi(y'),\\
W_{h,\epsilon}|_{t=0} & =V_{h,\epsilon,0},\quad \partial_{t}W_{h,\epsilon}|_{t=0}=V_{h,\epsilon,1},\\
W_{h,\epsilon}|_{\partial\Omega\times [0,T]} & =0.
\end{array}
\right.
\end{align}
Since $V_{h,\epsilon}$ is a solution to \eqref{wavvv},  $w_{h,\epsilon,err}$ must satisfy the following equation
\begin{align}\label{eqwerr}
\left\{
\begin{array}{ll}
\square_{g} w_{h,\epsilon,err} & =-\square_{\tilde{g}}\tilde{V}_{h,\epsilon/3} h^{-(d-2)/4}e^{-\frac{|y'|^{2}}{2h}}\chi(y')
-(1+xb(y_{1}))\partial^{2}_{y_{1}}W_{h,\epsilon}+\\ & +\sum_{j,k=1}^{d-1}B_{j,k}(x,y)
\partial^{2}_{y_{j},y_{k}}W_{h,\epsilon},\\
w_{h,\epsilon,err}|_{t=0} & = 0,\quad \partial_{t}w_{h,\epsilon,err}|_{t=0}=0,\\
w_{h,\epsilon,err}|_{\partial\Omega\times[0,T]} & =0,
\end{array}
\right.
\end{align}
where we set $\square_{g}:=\partial^{2}_{t}-\Delta_{g}$ and we used that 
\[
\Delta_{g}-\Delta_{\tilde{g}}=-(1+xb(y_{1}))\partial^{2}_{y_{1}}+\sum_{j,k=1}^{d-1}B_{j,k}(x,y)
\partial^{2}_{y_{j},y_{k}}.
\]
\begin{lemma}\label{lemestimerr}
For $t\in [0,T]$ the solution $w_{h,\epsilon,err}$ to the wave equation \eqref{eqwerr} satisfies 
\begin{equation}\label{estimerr}
\|(\partial^{2}_{t}-\Delta_{g})w_{h,\epsilon,err}(.,t)\|_{L^{2}(\Omega)}  \lesssim h^{-2(1-(1-\epsilon/3)/2)}\|w_{h,\epsilon,err}\|_{L^{2}(\Omega)}\simeq h^{-1-\epsilon/3},
\end{equation}
\begin{equation}\label{estimerr1}
\|(\partial^{2}_{t}-\Delta_{g})w_{h,\epsilon,err}(.t)\|_{\dot{H}^{-1}(\Omega)}  \lesssim  h^{-\epsilon/3}\|w_{h,\epsilon,err}\|_{L^{2}(\Omega)}\simeq h^{-\epsilon/3},
\end{equation}
where the estimates hold uniformly in $t\in [0,T]$ with constants independent of $\epsilon$ of $h$.

Moreover,
\begin{equation}\label{estimwqr}
\|w_{h,\epsilon,err}\|_{L^{q}([0,T],L^{r}(\Omega))}\leq C h^{-\beta(r,d)+2\epsilon-\epsilon/3},
\end{equation}
where $C=C(T)>0$ is independent of $\epsilon$.
\end{lemma}
\begin{proof}
We start with \eqref{estimwqr}.
Assume we have already proved \eqref{estimerr1}. The Duhamel formula for $w_{h,\epsilon,err}$ writes
\begin{equation}
w_{h,\epsilon,err}(x,y,t)=\int_{0}^{t}\frac{\sin((t-s)\sqrt{-\Delta_{g}})}{\sqrt{-\Delta_{g}}}\Big((\partial^{2}_{t}-\Delta_{g})w_{h,\epsilon,err}(x,y,s)\Big)ds.
\end{equation}
Using the Minkowski inequality together with \eqref{estimerr} we find
\begin{align}\label{esterr}
\|w_{h,\epsilon,err}(.,t)\|_{L^{r}(\Omega)} & =\|\int_{0}^{t}\frac{\sin((t-s)\sqrt{-\Delta_{g}})}{\sqrt{-\Delta_{g}}}\Big((\partial^{2}_{t}-\Delta_{g})w_{h,\epsilon,err}(.,s)\Big)ds\|_{L^{r}(\Omega)} \\
\nonumber
& \lesssim \int_{0}^{t} \|\frac{\sin((t-s)\sqrt{-\Delta_{g}})}{\sqrt{-\Delta_{g}}}\Big((\partial^{2}_{t}-\Delta_{g}) w_{h,\epsilon,err}(.,s)\Big)\|_{L^{r}(\Omega)}ds \\
\nonumber
& \lesssim h^{-\beta(r,d)+2\epsilon}\|(\sqrt{-\Delta_{g}})^{-1}(\partial^{2}_{t}-\Delta_{g}) w_{h,\epsilon,err}\|_{L^{1}([0,T],L^{2}(\Omega))}\\
\nonumber
& \simeq h^{-\beta(r,d)+2\epsilon}\|(\partial^{2}_{t}-\Delta_{g}) w_{h,\epsilon,err}\|_{L^{1}([0,T],\dot{H}^{-1}(\Omega))}\\
\nonumber
& \lesssim h^{-\beta(r,d)+2\epsilon-\epsilon/3},
\end{align}
where in the third line we used that the wave operator $\sin (t\sqrt{-\Delta_{g}})$ was supposed to be 
bounded by $h^{-\beta(r,d)+2\epsilon}$ and where in the last line we used \eqref{estimerr}. The respective constant in the last inequality depend only on $T$ and the estimates hold uniformly with respect to $t\in [0,T]$.

We now proceed with \eqref{estimerr} and \eqref{estimerr1}. In order to do this we use the special form of $\Delta_{g}$ and the fact that $\tilde{V}_{h,\epsilon/3}(x,y_{1},t)$ (and thereforee $V_{h,\epsilon}$) is supported for $0\leq x\lesssim h^{(1-\epsilon/3)/2}$. The inhomogeneous part of the equation \eqref{eqwerr} writes
\begin{equation}\label{estimnonlinterms}
\square\tilde{V}_{h,\epsilon/3} h^{-(d-2)/4}e^{-\frac{|y'|^{2}}{2h}}\chi(y')+(1+xb(y_{1}))\partial^{2}_{y_{1}}W_{h,\epsilon}-\sum_{j,k=1}^{d-1}B_{j,k}(x,y)\partial^{2}_{y_{j},y_{k}}W_{h,\epsilon}.
\end{equation}
The $L^{2}(\Omega)$ norm of $\square\tilde{V}_{h,\epsilon/3} h^{-(d-2)/4}e^{-\frac{|y'|^{2}}{2h}}\chi(y')$ is estimated using the last condition in Theorem \ref{thms2} and its contribution in the norm of the non-linear term of \eqref{eqwerr} is $O_{L^{2}(\Omega)}(1/h)$. We estimate the second term in \eqref{estimnonlinterms} as follows:
\begin{multline}
-(1+xb(y_{1}))\partial^{2}_{y_{1}}W_{h,\epsilon}+B_{1,1}(x,y)\partial^{2}_{y_{1},y_{1}}W_{h,\epsilon}=\\ \\h^{-(d-2)/4}e^{-\frac{|y'|^{2}}{2h}}\chi(y')\partial^{2}_{y_{1}}\tilde{V}_{h,\epsilon/3} \Big(B_{1,1}(x,y)-1-b(y_{1})\Big).
\end{multline}
The last term in \eqref{estimnonlinterms} splits into two sums, corresponding to $k=1$,
\begin{multline}\label{eqwerrlastterm}
\sum_{j=1}^{d-1}B_{1,j}(x,y)\partial^{2}_{y_{1},y_{j}}W_{h,\epsilon}=\\ 
-h^{-(d-2)/4}e^{-\frac{|y'|^{2}}{2h}} \frac{1}{h}\partial_{y_{1}}\tilde{V}_{h,\epsilon/3}\sum_{j=2}^{d-1}B_{1,j}(x,y)\Big(y_{j}\chi(y')+
h\partial_{y_{j}}\chi(y')\Big),
\end{multline}
and $k\in\{2,..,d-1\}$, respectively,
\begin{multline}\label{eqwerrlastterm}
\sum_{j,k=2}^{d-1}B_{j,k}(x,y)\partial^{2}_{y_{j},y_{k}}W_{h,\epsilon}=
h^{-(d-2)/4}e^{-\frac{|y'|^{2}}{2h}} 
\frac{1}{h^{2}}B_{j,k}(x,y)\tilde{V}_{h,\epsilon/3}\\
\times\sum_{j,k=2}^{d-1}\Big(y_{j}y_{k}\chi(y')-h(y_{j}\partial_{y_{k}}\chi(y')+y_{k}\partial_{y_{j}}\chi(y')+\delta_{j=k})
+h^{2}\partial^{2}_{y_{j},y_{k}}\chi(y')\Big).
\end{multline}
If $|y'|\geq h^{(1-\epsilon')/2}$ for some $\epsilon'>0$, then $e^{-\frac{|y'|^{2}}{2h}}\leq C_{M}h^{M}$, for all $M\geq 0$, thus taking $\epsilon'=\epsilon/3$ we can estimate the $L^{2}(\Omega)$ norm of \eqref{eqwerrlastterm} as follows
\begin{align}
\nonumber
\|-(1+xb(y_{1}))\partial^{2}_{y_{1}}W_{h,\epsilon}+\sum_{j,k=1}^{d-1}B_{j,k}(x,y)\partial_{y_{j}}\partial_{y_{k}}W_{h,\epsilon}\|_{L^{2}(\Omega)}&  \lesssim h^{-2+(1-\epsilon/3)}\|\tilde{V}_{h,\epsilon/3}\|_{L^{2}(\tilde{\Omega})}\\  & \lesssim h^{-1-\epsilon/3},
\end{align}
where we used that 
\[
\sup_{\epsilon>0}\|\tilde{V}_{h,\epsilon/3}\|_{L^{2}(\tilde{\Omega})}\lesssim 1, \quad \sup_{\epsilon>0}\|\partial_{y_{1}}\tilde{V}_{h,\epsilon/3}\|_{L^{2}(\tilde{\Omega})}\lesssim\frac{1}{h}, \quad \sup_{\epsilon>0}\|\partial^{2}_{y_{1}}\tilde{V}_{h,\epsilon/3}\|_{L^{2}(\tilde{\Omega})}\lesssim\frac{1}{h^{2}}.
\]
In the same way we obtain the following bounds
\begin{align}
\nonumber
\|-(1+xb(y_{1}))\partial^{2}_{y_{1}}W_{h,\epsilon}+\sum_{j,k=1}^{d-1}B_{j,k}(x,y)\partial^{2}_{y_{j},y_{k}}W_{h,\epsilon}\|_{\dot{H}^{-1}(\Omega)} \\
\nonumber
 \lesssim  
h\|-(1+xb(y_{1}))\partial^{2}_{y_{1}}W_{h,\epsilon}+\sum_{j,k=1}^{d-1}B_{j,k}(x,y)\partial^{2}_{y_{j},y_{k}}W_{h,\epsilon}\|_{L^{2}(\Omega)}\\  \lesssim  h^{-\epsilon/3}.
\end{align}
For the last inequality we used the following lemma
\begin{lemma}\label{lemh}
Let $f(x,y):\Omega\rightarrow\mathbb{R}$ be localized at frequency $1/h$ in the $y\in\mathbb{R}^{d-1}$ variable, i.e. such that there exists $\psi\in C^{\infty}_{0}(\mathbb{R}^{d-1}\setminus \{0\})$ with $\psi(hD_{y})f=f$. Then there exists a constant $C>0$ independent of $h$ such that one has
\[
\|f\|_{\dot{H}^{-1}(\Omega)}\leq Ch\|f\|_{L^{2}(\Omega)}.
\]
\end{lemma}
\begin{proof}(\emph{of Lemma \ref{lemh}}:)
Since $\psi(hD_{y})f=f$ we have
\begin{align}
\|f\|_{\dot{H}^{-1}(\Omega)}&=\sup_{\|g\|_{\dot{H}^{1}(\Omega)}\leq 1}\int \psi f\bar{g}\leq\|f\|_{L^{2}(\Omega)}\times \sup_{\|g\|_{\dot{H}^{1}(\Omega)}\leq 1}\|\psi(hD_{y})g\|_{L^{2}(\Omega)}\\
\nonumber
&\leq h\|f\|_{L^{2}(\Omega)}\|\tilde{\psi}(hD_{y})\nabla_{y}g\|_{L^{2}(\Omega)}\\
\nonumber
&\leq C h\|f\|_{L^{2}(\Omega)},
\end{align}
where we set $\tilde{\psi}(\eta)=|\eta|^{-1}\psi(\eta)$. Hence Lemma \ref{lemh} is proved.
\end{proof}
\end{proof}
\emph{End of the proof of Proposition \ref{propcontr}:}

Recall that we have assumed that the operator
\[
\sin t\sqrt{-\Delta_{g}}:L^{2}(\Omega)\rightarrow L^{q}([0,T],L^{r}(\Omega))
\]
is bounded by $h^{-\beta(r,d)+2\epsilon}$. This assumption implies
\begin{align}\label{estimv}
\|V_{h,\epsilon}\|_{L^{q}([0,T],L^{r}(\Omega))} & \leq C h^{-\beta(r,d)+2\epsilon}(\|V_{h,\epsilon,0}\|_{L^{2}(\Omega)}+\|V_{h,\epsilon,1}\|_{\dot{H}^{-1}(\Omega)})\\ 
\nonumber
& \leq \tilde{C} h^{-\beta(r,d)+2\epsilon},
\end{align}
where $C,\tilde{C}>0$ are independent of $h$. If \eqref{estimv} were true, together with \eqref{estwred} it would yield
\begin{align}
h^{-\beta(r,d)+2\epsilon/3} & \lesssim\|W_{h,\epsilon}\|_{L^{q}([0,T],L^{r}(\Omega))}\\ 
\nonumber
& \lesssim (\|V_{h,\epsilon}\|_{L^{q}([0,T],L^{r}(\Omega))}+\|w_{h,\epsilon,err}\|_{L^{q}([0,T],L^{r}(\Omega))}).
\end{align}
The last estimate together with \eqref{estimwqr} and \eqref{estimv} gives a contradiction, since it would imply 
\[
h^{-\beta(r,d)+2\epsilon/3}\lesssim h^{-\beta(r,d)+2\epsilon}+h^{-\beta(r,d)+2\epsilon-\epsilon/3} 
\]
which obviously can't be true. The proof is complete.
\end{proof}

\section{Proof of Theorem \ref{thmstrichartz}}\label{secdoide}
In order to finish the proof of Theorem \ref{thmstrichartz} it remains to prove Theorem \ref{thms2}. We recall its statement using the notations of Theorem \ref{thmstrichartz} that we will keep in the rest of this paper. We are therefore reduced to prove the following:
\begin{thm}\label{thms3}
Let $(\Omega,g)$ be a Riemannian manifold of dimension $d=2$ with $C^\infty$ boundary $\partial\Omega$. Suppose that local coordinates can be chosen near $(x=0,y=0)$ such that in a neighborhood of this point $\Omega$ be given by
\begin{equation}\label{defomeg}
\Omega=\{(x,y)|x>0,y\in\mathbb{R}\},
\end{equation}
and the Laplace-Beltrami operator $\Delta_{g}$ associated to the metric $g$ be given by
\begin{equation}\label{deflaplag}
\Delta_{g}=\partial^{2}_{x}+(1+xb(y))\partial^{2}_{y},
\end{equation}
where $b$ is a smooth function with $b(0)=0$. Let $Y>0$ be such that for $y\in [0,Y]$ we have $|b^{1/3}(y)-1|\leq \frac{1}{10}$. Then there exists $T=T(\Omega)$ and for every $\epsilon>0$ small enough there exist sequences $V_{h,j,\epsilon}$, $j\in\{0,1\}$ and approximate solutions $V_{h,\epsilon}$ to the wave equation on $\Omega$ with Dirichlet boundary condition 
\begin{equation}\label{undered3}
\left\{
      \begin{array}{ll}
\partial^{2}_{t}V-\partial^{2}_{x}V-(1+xb(y))\partial^{2}_{y}V=0,\quad\text{on}\quad \Omega\times\mathbb{R}\\
V|_{t=0}=V_{h,0,\epsilon},\quad
\partial_{t}V|_{t=0}=V_{h,1,\epsilon},\\
V|_{\partial\Omega\times[0,T]}=0,
     \end{array}
     \right.
\end{equation}
satisfying the following conditions:
\begin{itemize}
\item First, $V_{h,\epsilon}$ is an approximate solution to \eqref{undered} in the sense that
\begin{equation}
\partial^{2}_{t}V_{h,\epsilon}-\partial^{2}_{x}V_{h,\epsilon}-(1+xb(y))\partial^{2}_{y}V_{h,\epsilon}=O_{L^{2}(\Omega)}(1/h),\quad \|V_{h,\epsilon}\|_{L^{2}(\Omega)}\leq 1.
\end{equation}

\item Secondly, $V_{h,\epsilon}$ writes as a sum
\begin{equation}\label{hred}
V_{h,\epsilon}(x,y,t)=\sum_{n=0}^{N}v^{n}_{h,\epsilon}(x,y,t),\quad N\simeq h^{-\frac{(1-\epsilon)}{4}},
\end{equation}
where the functions $v^{n}_{h,\epsilon}(x,y,t)$ satisfy the following conditions:
\begin{itemize}
\item  for $4<r<\infty$:
\begin{equation}\label{hest}     
\left\{
\begin{array}{ll}
\|v^{n}_{h,\epsilon}(.,t)\|_{L^{r}(\Omega)}\geq C h^{-\frac{3}{2}(\frac{1}{2}-\frac{1}{r})-\frac{1}{6}(\frac{1}{4}-\frac{1}{r})+2\epsilon},\\
\sup_{\epsilon>0}\Big(\|v^{n}_{h,\epsilon}(.,t)\|_{L^{2}(\Omega)}+h\|\partial_{t}v^{n}_{h,\epsilon}(.,t)\|_{L^{2}(\Omega)}\Big)\leq 1,
\end{array}
\right.
\end{equation}
where the constant $C=C(Y)\simeq Y^{1/q}>0$ is independent of $h$, $\epsilon$ or $n$;

\item $v^{n}_{h,\epsilon}(x,y,t)$ are essentially supported for the time variable $t$ in almost disjoint intervals of time each of size $\simeq h^{\frac{(1-\epsilon)}{4}}$ and for the tangential variable $y$ in almost disjoint intervals.

\item $V_{h,\epsilon}$ are supported for the normal variable $0\leq x\lesssim h^{(1-\epsilon)/2}$ (with constants dependent only on $\Omega$ and not on $\epsilon$,$h$) and localized at spatial frequency $\frac{1}{h}$ in the tangential variable $y$. Moreover, we have,uniformly in $\epsilon>0$,
\begin{equation}\label{hnormes}
\sup_{\epsilon>0} \|V_{h,\epsilon}\|_{L^{2}(\Omega)}\lesssim 1,\quad \sup_{\epsilon>0}\|\partial_{y}V_{h,\epsilon}\|_{L^{2}(\Omega)}\lesssim\frac{1}{h},\quad  \sup_{\epsilon>0}\|\partial^{2}_{y}V_{h,\epsilon}\|_{L^{2}(\Omega)}\lesssim\frac{1}{h^{2}}.
\end{equation}
\end{itemize}
\end{itemize}
\end{thm}
\begin{rmq}\label{defepsilon}
In what follows we fix $\epsilon>0$ small enough and we do not mention anymore the dependence on $\epsilon$ of the solution to the wave equation \eqref{undered3} we are going to construct. 
\end{rmq}
Before starting the proof of Theorem \ref{thms3} we briefly recall some definitions we shall use in the rest of the paper  (for details see \cite{hormand} or \cite{vas06}, for example).

Set $X=\Omega\times\mathbb{R}_{t}$, let $\square_{g}=\partial^{2}_{t}-\Delta_{g}$ denote the wave operator on $X$ and let $p\in C^{\infty}(T^{*}X\setminus o)$  be the principal symbol of $\square_{g}$, which is homogeneous of degree $2$ in $T^{*}X\setminus o$ (where we write $o$ for the "zero section" of $T^{*}X$), 
\begin{equation}\label{symbwave}
p(x,y,t,\xi,\eta,\tau)=\xi^{2}+(1+xb(y))\eta^{2}-\tau^{2}.
\end{equation}
The characteristic set $P:=\text{Char}(p)\subset T^{*}X\setminus o$ of $\square_{g}$ is defined by $p^{-1}(\{0\})$. If we denote $N^{*}\partial\Omega$ the conormal bundle of $\partial X$ we notice that $\text{Char}(p)\cap N^{*}\partial\Omega=\emptyset$, meaning that the boundary is non-characteristic for $\square_{g}$. 

Let us consider the Dirichlet problem for $\square_{g}$:
\begin{equation}\label{bdrypb}
\square_{g} u=0, \quad u|_{\partial X}=0.
\end{equation}
The statement of the propagation of singularities of solutions to \eqref{bdrypb} has two main ingredients: locating singularities of a distribution, as captured by the wave front set, and describing the curves along which they propagate, namely the bicharacteristics. Both of these are closely related to an appropriate notion of "phase space", in which both the wave front set and the bicharateristics are located. On manifolds without boundary, this phase space is the standard cotangent bundle $T^{*}X$. In presence of boundaries the phase space is the $b$-cotangent bundle, $^{b}T^{*}X$. Let $o$ denote the zero section of $^{b}T^{*}X$. Then $^{b}T^{*}X\setminus o$ is equipped with an $\mathbb{R}^{+}$-action (fiberwise multiplication) which has no fixed points.  
There is a natural non-injective "inclusion" $\pi:T^{*}X\rightarrow ^{b}T^{*}X$. We define the elliptic, glancing and hyperbolic sets in $T^{*}\partial X$ as follows: 
\[
\mathcal{E}=\{q\in \pi(T^{*}X)\setminus o| \pi^{-1}(q)\cap \text{Char}(p)=\emptyset\},
\]
\[
\mathcal{G}=\{q\in \pi(T^{*}X)\setminus o| \text{Card}(\pi^{-1}(q)\cap \text{Char}(p))=1\},
\]
\[
\mathcal{H}=\{q\in \pi(T^{*}X)\setminus o| \text{Card}(\pi^{-1}(q)\cap \text{Char}(p))\geq 2\},
\]
with $\text{Card}$ denoting the cardinality of a set; each of these is a conic subset of $\pi(T^{*}X)\setminus o$. Note that in $T^{*}\mathring{X}$, $\pi$ is the identity map, so every point $q\in T^{*}\mathring{X}$ is either elliptic or glancing, depending on weather $q\notin \text{Char}(p)$ or $q\in \text{Char}(p)$.

The canonical local coordinates on $T^{*}X$ will be denoted $(x,y,t,\xi,\eta,\tau)$, so one forms are $\alpha=\xi dx+\eta dy+\tau dt$.
Let $(\rho,\vartheta)=(x,y,t,\xi,\eta,\tau)$ on $T^{*}X$ near $\pi^{-1}(q)$, $q\in T^{*}\partial X$, and corresponding coordinates $(y,t,\eta,\tau)$ on a neighborhood $\mathcal{U}$ of $q$ in $T^{*}\partial X$. Consequently,
\[
\mathcal{E}\cap\mathcal{U}=\{(y,t,\eta,\tau)| \tau^{2}<\eta^{2}\},
\]
\[
\mathcal{G}\cap\mathcal{U}=\{(y,t,\eta,\tau)| \tau^{2}=\eta^{2}\},
\]
\[
\mathcal{H}\cap\mathcal{U}=\{(y,t,\eta,\tau)| \tau^{2}>\eta^{2}\}.
\]
Let $\rho=\rho(s)=(x,y,t)(s)$, $\vartheta=\vartheta(s)=(\xi,\eta,\tau)(s)$ be a bicharacteristic of $p(\rho,\vartheta)$, i.e. such that $(\rho,\vartheta)$ satisfies
\begin{equation}
\frac{d\rho}{ds}=\frac{\partial p}{\partial\vartheta},\quad\frac{d\vartheta}{ds}=-\frac{\partial p}{\partial\rho},\quad
p(\rho(0),\vartheta(0))=0.
\end{equation}
We say that $(\rho(s),\vartheta(s))|_{s=0}$ on the boundary $\partial X$ is a gliding point if it satisfies
\begin{equation}\label{hypconv}
x(\rho(0))=0,\quad \frac{d}{ds}x(\rho(0))=0,\quad \frac{d^{2}}{ds^{2}}x(\rho(0))<0.
\end{equation}
This is equivalent to saying that $(\rho,\vartheta)\in T^{*}X\setminus o$ is a gliding point if
\begin{equation}\label{hyconv}
p(\rho,\vartheta)=0,\quad \{p,x\}_{|_{(\rho,\vartheta)}}=0,\quad \{\{p,x\},p\}_{|_{(\rho,\vartheta)}}>0.
\end{equation}
The assumption on the domain $\Omega$ in Theorem \ref{thms3} near the boundary point $(x=0,y=0)$ is equivalent to saying that there exists a bicaracteristic intersecting tangentially the boundary at $(0,0)$ and having second order contact with the boundary at this point. From \eqref{hyconv} this last condition translates into the following: $\exists (\xi_{0},\eta_{0},\tau_0)$, such that

\begin{equation}\label{glid1}
\tau^{2}_0=(1+xb(y))\eta^{2}_{|_{(0,0,\xi_0,\eta_0)}},\quad \{p,x\}_{|_{(0,0,\xi_0,\eta_0)}}= \frac{\partial p}{\partial\xi}_{|_{(0,0,\xi_0,\eta_0)}}=2\xi_{0}=0,
\end{equation}
\begin{equation}\label{glid2}
\{\{p,x\},p\}_{|_{(0,0,\xi_0,\eta_0)}}=\{\frac{\partial p}{\partial\xi},p\}_{|_{(0,0,\xi_0,\eta_0)}}=2b(0)\eta^{2}_{0}>0.
\end{equation}
Since $b(0)=1$ we must have $\xi_0=0$ and $\eta_0\neq 0$. Suppose without loss of generality that $\eta_{0}=1$. Let $(\rho_{0},\vartheta_{0})=(x=0,y=0,t=0,\xi_0=0,\eta_0=1,\tau_0=-1)$ and denote the gliding point (in $T^{*}\partial X$) by
\begin{equation}\label{dfnglidpt}
\pi (\rho_{0},\vartheta_{0})=(y=0,t=0,\eta_{0},\tau_{0}=-\eta_{0})=(0,0,1,-1)\in\mathcal{G}.
\end{equation} 
 
We define the semi-classical wave front set $WF_{h}(u)$ of a distribution $u$ on $\mathbb{R}^{3}$ to be the complement of the set of points $(\rho=(x,y,t),\zeta=(\xi,\eta,\tau))\in\mathbb{R}^{3}\times(\mathbb{R}^{3}\setminus 0)$  for which there exists a symbol $a(\rho,\zeta)\in\mathcal{S}(\mathbb{R}^{6})$ such that $a(\rho,\zeta)\neq 0$ and for all integers $m\geq 0$ the following holds
\[
\|a(\rho,hD_{\rho})u\|_{L^{2}}\leq c_{m}h^{m}.
\]

\subsection{Choice of an approximate solution}
We look for an approximate solution to the equation \eqref{undered3} of the form
\begin{equation}\label{solgen}
u_{h}(x,y,t)=\int_{\xi,\eta,\tau}e^{\frac{i}{h}(\theta+\zeta\xi+\frac{\xi^{3}}{3})}g_{h}d\xi d\eta d\tau,
\end{equation}
where the phase functions $\theta(x,y,t,\eta,\tau)$, $\zeta(x,y,\eta,\tau)$ are real valued and homogeneous in $(\eta,\tau)$ of degree $1$ and $2/3$, respectively, and where  $g_{h}$ is a symbol to be determined in the next sections. In order $u_h$ to solve \eqref{undered3}, the functions  $\theta$, $\zeta$ must solve an eikonal equation that we derive in what follows. We denote by $<.,.>$ the symmetric bilinear form obtained by polarization of the second order homogeneous principal symbol $p$ of the wave operator $\square_{g}$,
\begin{equation}\label{product}
<da,db>=\partial_{x}a\partial_{x}b+(1+xb(y))\partial_{y}a\partial_{y}b-\partial_{t}a\partial_{t}b.
\end{equation}
Applying the wave operator $h^{2}\square_{g}$ to $u_{h}$, the main contribution becomes
\begin{multline}\label{eqeik}
(\partial_{x}\theta+\xi\partial_{x}\zeta)^{2}+(1+xb(y))(\partial_{y}\theta+\xi\partial_{y}\zeta)^{2}-(\partial_{t}\theta+\xi\partial_{t}\zeta)^{2}=\\
=<d\theta,d\theta>-2\xi<d\theta,d\zeta>+\xi^{2}<d\zeta,d\zeta>.
\end{multline}
In order to eliminate this term after integrations by parts in $\xi$ we ask that the right hand side of \eqref{eqeik} to be a nontrivial multiple of $\partial_{\xi}\Phi$, where we set
\begin{equation}\label{phasephi}
\Phi=\theta+\zeta\xi+\frac{\xi^{3}}{3}.
\end{equation}
This is equivalent to determine $\theta$, $\zeta$ solutions to
\begin{equation}\label{sistem1}
\left\{
	\begin{array}{ll}
<d\theta,d\theta>-\zeta<d\zeta,d\zeta>=0,\\
<d\theta,d\zeta>=0.
\end{array}
	\right.
\end{equation}
The system \eqref{sistem1} is a nonlinear system of partial differential equations, which is elliptic where $\zeta>0$ (shadow region), hyperbolic where $\zeta<0$ (illuminated region) and parabolic where $\zeta=0$ (caustic curve or surface). It is crucial that there is a solution of the form
\begin{equation}\label{phasesistem}
\phi^{\pm}=\theta\mp\frac{2}{3}(-\zeta)^{3/2} 
\end{equation}
with $\theta$, $\zeta$ smooth.  In terms of \eqref{phasesistem}, the eikonal equation takes the form
\begin{equation}\label{psistem}
p(x,y,t,d\phi^{\pm})=0
\end{equation}
by taking the sum and the difference of the equations \eqref{psistem}. 
It is easy, by Hamilton-Jacobi theory, to find many smooth solutions to the eikonal equation \eqref{psistem}. Solutions with the singularity \eqref{phasesistem} arise from solving the initial value problem for    
\eqref{psistem} off an initial surface which does not have the usual transversality condition, corresponding to the fact that there are bicharacteristics tangent to the boundary.

\subsubsection{Geometric reduction}\label{secphases}
Let $X=\Omega\times\mathbb{R}$ as before. Let $p$ and $q$ be functions on $T^*X$ with independent differentials at a point $(\rho,\vartheta)\in T^{*}X\setminus o$. We denote by $P$ and $Q$ the hypersurfaces defined by $p$ and $q$, respectively.
\begin{dfn}\label{dfnglancing}
We say that the hypersurfaces $P$, $Q$ in the symplectic manifold $T^{*}X$ are glancing surfaces at $(\rho,\vartheta)$ if
\begin{enumerate}
\item $\{p,q\}((\rho,\vartheta))=0$,
\item $\{p,\{p,q\}\}((\rho,\vartheta))\neq 0$ and $\{q,\{q,p\}\}((\rho,\vartheta))\neq 0$.
\end{enumerate}
\end{dfn}
In our case we take $q$ to be the defining function of the boundary $\partial\Omega$, therefore $q=x$, and $p$ the symbol of the wave operator $\square_{g}$ defined in \eqref{symbwave}. Precisely,
\begin{equation}\label{hypergen}
Q=\{q(x,y,t,\xi,\eta,\tau)=x=0\},\quad P=\{p=\xi^{2}+(1+xb(y))\eta^{2}-\tau^{2}=0\},
\end{equation}
which are glancing at $(\rho_{0},\vartheta_{0})$ defined in \eqref{dfnglidpt}.
The nondegeneracy conditions in Definition \ref{dfnglancing} hold at a point $(\rho,\vartheta)$ with $\{p,q\}=0$ if and only if $\partial\Omega$ is strictly convex at $(\rho,\vartheta)$.
\begin{rmq}
A model case of a pair of glancing surfaces is given by
\begin{equation}\label{hypermod}
Q_{F}=\{q_{F}(x,y,\xi,\eta,\tau)=x=0\},\quad P_{F}=\{p_{F}=\xi^{2}+(1+x)\eta^{2}-\tau^{2}=0\},
\end{equation}
which have a  second order intersection at the point 
\[
(\bar{\rho}_{0},\bar{\vartheta}_{0}):=(0,y_{0}=0,t_{0}=0,0,\eta_{0}=1,\tau_{0}=-1)\in T^{*}X_{F}\setminus o.
\]
This model case was studied in \cite{doi}. There is a deep geometrical reason underlying the similarity of the general gliding ray parametrice for \eqref{hypergen} and the one for the model example \eqref{hypermod}, which will facilitate solution to the eikonal equation.
\end{rmq}

\begin{thm}\label{thmphase}
Let $P$ and $Q$ be two hypersurfaces in $T^{*}X\setminus o$ satisfying the glancing conditions in Definition \ref{dfnglancing} at $(\rho_{0},\vartheta_{0})\in P\cap Q\subset T^{*}X\setminus o$. Then there exist real functions $\theta$ and $\zeta$ which are $C^{\infty}$ in a conic neighborhood $\mathcal{U}$ of $(\rho_{0},1,-1)\in X\times\mathbb{R}^{2}$, are homogeneous of degrees one and two-thirds, respectively, and have the following properties
\begin{itemize}
\item $\zeta_{0}:=\zeta|_{x=0}=-(\tau^{2}-\eta^{2})\eta^{-4/3}$ and $\partial_x \zeta|_{\partial X}>0$ on $\mathcal{U}\cap \partial X\times\mathbb{R}^{2}$,
\item $d_{y,t}(\partial_{\eta}\theta,\partial_{\tau}\theta)$ are linearly independent on $\mathcal{U}$,
\item the system \eqref{sistem1} holds in $\zeta\leq 0$ and in Taylor series on $\partial X$.
\end{itemize}
Moreover, $\zeta$ is a defining function for the fold set denoted $\Sigma$. By translation invariance in time $\zeta$ is independent of $t$ while the phase function $\theta$ is linear in the time variable.
\end{thm}
\begin{rmq}
Theorem \ref{thmphase} has been proved independently by Melrose in \cite{mel76} and by Eskin in \cite[Thm.1]{esk77} for the following canonical glancing surfaces
\begin{equation}
Q_{can}=\{x=0\},\quad P_{can}=\{\xi^2+x\eta^2-\tau\eta=0\},
\end{equation} 
near the glancing point $(x=0,y,t,\xi=0,\eta=1,\tau=0)$. In the last part of this section we show that phase functions $\theta$ and $\zeta$ can be chosen to verify the conditions stated above and we also give their precise form near the glancing point $(x=0,y=0,\eta=1,\tau=-1)$.
\end{rmq}
\begin{rmq}
The phase functions $\theta$, $\zeta$, solutions to the eikonal equations \eqref{sistem1}, will be perfectly determined to satisfy the conditions in Theorem \ref{thmphase}. In the next sections we shall use the approach in the case of the model glancing surfaces \eqref{hypermod} in order to determine a parametrix for \eqref{undered3} near the general glancing surfaces \eqref{hypergen}. This will be possible using the symplectomorphisme generated by the restriction of the phase function $\theta$ to $\partial X$.
\end{rmq}
\begin{rmq}
Notice in fact that if $P$ and $Q$ are the hypersurfaces in the symplectic space $T^{*}X\setminus o$ defined in \eqref{hypergen}
and glancing at $(\rho_{0},\vartheta_{0})\in T^{*}X\setminus o$, then there exists a canonical transformation  
\begin{equation}\label{canuse}
\chi:\Gamma\subset T^{*}X_{F}\setminus o\rightarrow T^{*}X\setminus o,
\end{equation}
defined in a conic neighborhood $\Gamma$ of $(\bar{\rho}_{0},\bar{\vartheta}_{0})$ and taking $(\bar{\rho}_{0},\bar{\vartheta}_{0})$ to $(\rho_{0},\vartheta_{0})$ and the model pair $P_{F}$ and $Q_{F}$ to $P$ and $Q$. The fact that $\chi$, which is symplectic, maps $Q_{F}$ onto $Q$ means that it defines a local canonical transformation from the quotient space of $Q_{F}$, modulo its Hamilton fibration, to the corresponding quotient space of $Q$, which is naturally identified as the cotangent space of the hypersurface
\[
Q/\mathbb{R}H_{q}\simeq T^{*}\partial X.
\]
Now, as we just said, on $Q$ (and similarly on $Q_{F}$) the symplectic form gives a Hamilton foliation. Let this determine an equivalence relation $\sim$.  Then $Q\cap P/\sim$ has the structure of a symplectic manifold with boundary, and it is naturally isomorphic to the closure of the "hyperbolic" set in $T^{*}\partial X$, the region over which real rays pass, and similarly $Q_{F}\cap P_{F}/\sim$. 
Therefore, the restriction of $\chi$ to $T^{*}\partial X_{F}$, that we denote $\chi_{\partial}$, is also a canonical transformation from a neighborhood $\gamma\subset T^{*}\partial X_{F}\setminus o$ of $\pi(\bar{\rho}_{0},\bar{\vartheta}_{0})$ to a neighborhood of $\pi(\rho_{0},\vartheta_{0})\in T^{*}\partial X\setminus o$,
\[
\chi_{\partial}:\gamma\rightarrow T^{*}\partial X \setminus o, \quad \gamma\subset T^{*}\partial X_{F}\setminus o,
\]
\[
\gamma=\{(y,t,\eta,\tau)\in T^{*}\partial X_{F}|\exists \xi, (0,y,t,\xi,\eta,\tau)\in \Gamma\}
\]
defined in the hyperbolic region by
\begin{equation}\label{chipartmodel}
\chi^{-1}_{\partial}:(y,t,d_{y}\theta_{0},d_{t}\theta_{0})\rightarrow (d_{\eta}\theta_{0},d_{\tau}\theta_{0},\eta,\tau),\quad \chi^{-1}_{\partial}(\pi(\rho_{0},\vartheta_{0}))=\pi(\bar{\rho}_{0},\bar{\vartheta}_{0}),
\end{equation}
where $\theta_{0}:=\theta|_{\partial X}$ is the restriction to $\partial X$ of the phase function $\theta$ introduced in Theorem \ref{thmphase}.

The map $\chi_{\partial}$ has the important property that near $\pi(\bar{\rho}_{0},\bar{\vartheta}_{0})$, it conjugates the billiard ball map $\delta^{\pm}\subset (T^{*}\partial X\setminus o)\times (T^{*}\partial X\setminus o)$ to the normal form $\delta^{\pm}_{F}$ introduced in \eqref{billball} in Section \ref{sectmodel}. Roughly speaking, the billiard ball maps are defined as follows: if $(y,t,\eta,\tau)$ is a hyperbolic point and if $\xi_{+}>0$ denotes the positive solution to $p(0,y,t,\xi,\eta,\tau)=0$ we consider the integral curve $(\rho(s),\vartheta(s))=\exp(s H_{p})(0,y,t,\xi_{+},\eta,\tau)$ of the Hamiltonian vector field of $p$ starting at $(0,y,t,\xi_{+},\eta,\tau)$; if it intersects transversally $T^{*} X|_{\partial X}$ at a time $s_{1}>0$ and lies entirely in $T^{*}\mathring{X}$ for $s\in (0,s_{1})$ we set $(0,y',t',\xi'_{-},\eta',\tau')=\exp(s_{1}H_{p})(0,y,t,\xi_{+},\eta,\tau)$ and define $\delta^{+}(y,t,\eta,\tau):=(y',t',\eta',\tau')$. Its local inverse is denoted $\delta^{-}$. An interpolating Hamiltonian for the billiard ball maps $\delta^{\pm}$ is  $\zeta_{0}$ and we have $\delta^{\pm}(y,t,\eta,\tau)=\exp(\pm\frac{4}{3}H_{(-\zeta_{0})^{3/2}})$. 
\end{rmq}

\subsubsection{Phase functions in the Friedlander's model}
In the Friedlander's model of the half space $\Omega_{F}:=\{(x,y)\in\mathbb{R}_{+}\times\mathbb{R}\}$   with Laplace operator defined by $\Delta_{F}:=\partial^{2}_{x}+(1+x)\partial^{2}_{y}$ that we have dealt with in \cite{doi}, the equation \eqref{psistem} has a solution of the form
\begin{equation}\label{fazacuspmodel}
\phi^{\pm}_{F}=\theta_{F}\mp\frac{2}{3}(-\zeta_{F})^{3/2},
\end{equation}
where
\begin{equation}\label{fazemodel}
\theta_{F}(x,y,t,\eta,\tau)=y\eta+t\tau,\quad \zeta_{F}(x,y,\eta,\tau)=(x-\frac{\tau^{2}-\eta^{2}}{\eta^{2}})\eta^{2/3},
\end{equation}
as can be seen by direct computation. This solution serves very much as a guide to the general construction as we shall see in the next sections.

\subsubsection{Phase functions associated to the special operator $\square_{g}$\\ (Proof of Theorem \ref{thmphase})}\label{sectphafcts}
Notice that in Section \ref{secred} we reduced the problem to the construction of a solution to the wave equation \eqref{undered3} in the two-dimensional case of the half-space $\Omega$ defined in \eqref{defomeg} together with the Laplace operator $\Delta_{g}$ given in \eqref{deflaplag}.

In this part we show that smooth solutions $\theta$ and $\zeta$ of the eikonal equation \eqref{sistem1} can be constructed near the glancing point $(\rho_{0},\vartheta_{0})=(0,0,0,0,1,-1)$ such that they satisfy the conditions in Theorem \ref{thmphase}. 
\begin{rmq}
Notice that $\theta$ can be chosen linear in $t$ and $\zeta$ independent of $t$. This is possible since, because of the translation invariance of both forms $(P_F,P)$ and $(Q_F,Q)$, one can take the canonical transformation $\chi$ to conjugate translation in the time variable.
\end{rmq}
Let $\mathcal{U}$ be a conic neighborhood of $(\varrho_0,1,-1)$ and for $(x,y,t,\eta,\tau)\in \mathcal{U}$ write
\begin{equation}\label{specialtheta}
\theta(x,y,t,\eta,\tau)=\theta_0(y,t,\eta,\tau)+\sum_{j\geq 1}\frac{x^j}{j!}\theta_j(y,\eta,\tau),
\end{equation}
\begin{equation}
\zeta(x,y,\eta,\tau)=\zeta_0(\eta,\tau)+\sum_{j\geq 1}\frac{x^j}{j!}\zeta_j(y,\eta,\tau),
\end{equation}
where $\theta_0$ is linear in $t$, $\partial_t\theta_0=\tau$, $\zeta_0(\eta,\tau)=-\frac{(\tau^2-\eta^2)}{\eta^2}\eta^{2/3}$ is independent of $y$ and where $\zeta_1(y,\eta,\tau)>0$ on $\mathcal{U}\cap \partial X\times \mathbb{R}^2$. We have to prove that smooth functions $\theta_j$ and $\zeta_j$ do exists such that the eikonal equations \eqref{sistem1} to be satisfied. The second eikonal equation in \eqref{sistem1} writes
\begin{equation}\label{estimthetax}
\partial_{x}\zeta\partial_x\theta=-(1+xb(y))\partial_{y}\zeta\partial_{y}\theta,
\end{equation}
from which we find
\begin{align}
\nonumber
\sum_{l\geq 0}\frac{x^l}{l!}\sum_{k=0}^l C_l^k\theta_{k+1}\zeta_{l+1-k} & =-(1+xb(y))\Big(\sum_{l\geq 1}\frac{x^l}{l!}\sum_{k=0}^l C_l^k\partial_y \theta_k\partial_y\zeta_{l-k}\Big)\\
& = -\sum_{l\geq 1}\frac{x^l}{l!}\Big(\sum_{k=0}^l C_l^k\partial_y \theta_k\partial_y\zeta_{l-k}+b(y)(\sum_{k=0}^{l-1} C_{l-1}^k\partial_y \theta_k\partial_y\zeta_{l-1-k})\Big).
\end{align}
Since the coefficient of $x^0$ in the right hand side vanishes and $\zeta_1>0$, we obtain $\theta_1(y,\eta,\tau)=0$. From the last condition we also determine $\theta_{l+1}$ for all $l\geq 1$
\begin{equation}\label{eqthetl}
\theta_{l+1}=-\frac{1}{\zeta_1}\Big(\sum_{k=0}^{l-1}C_l^k\theta_{k+1}\zeta_{l+1-k}+\sum_{k=0}^l C_l^k\partial_y \theta_k\partial_y\zeta_{l-k}+b(y)(\sum_{k=0}^{l-1} C_{l-1}^k\partial_y \theta_k\partial_y\zeta_{l-1-k})\Big).
\end{equation}
\begin{rmq}\label{remtheta}
Notice that $\theta_{l+1}$ is perfectly determined by $(\theta_k)_{k\in\{0,..,l\}}$ and $(\zeta_k)_{k\in\{0,..,l\}}$. Indeed, the term involving $\zeta_{l+1}$ in the preceding sum is a multiple of  $\theta_1$ and hence vanishes. 
\end{rmq}
Introducing \eqref{estimthetax} in \eqref{sistem1} gives
\begin{equation}\label{estimthetay}
\Big((\partial_{y}\theta)^{2}(1+xb(y))-\zeta(\partial_x\zeta)^2\Big)\Big((\partial_x \zeta)^2+(1+xb(y))(\partial _y \zeta)^2\Big)\\
=\tau^{2} (\partial_x \zeta)^2.
\end{equation}
Setting $C_l^k=\frac{l!}{k!(l-k)!}$, we compute
\[
(\partial_y\theta)^2=\sum_{l\geq 0}\frac{x^l}{l!}\sum_{k=0}^{l}C_l^k \partial_y\theta_{k}\partial_y \theta_{l-k},
\]
\[
(\partial_y\zeta)^2=\sum_{l\geq 2}\frac{x^l}{l!}\sum_{k=1}^{l}C_l^k \partial_y\zeta_{k}\partial_y \zeta_{l-k},
\]
\[
(\partial_x\zeta)^2=\sum_{l\geq 0}\frac{x^l}{l!}\sum_{k=0}^{l}C_l^k \zeta_{k+1}\zeta_{l+1-k},
\]
and if we let $d_p=\sum_{k=0}^{p}C_p^k \zeta_{k+1}\zeta_{p+1-k}$ denote the coefficient of $\frac{x^p}{p!}$ in the last expansion, then
\[
\zeta(\partial_x\zeta)^2=\sum_{l\geq 0}\frac{x^l}{l!}\sum_{k=0}^{l}C_l^k \zeta_{k}d_{l-k}.
\]
The first factor in the left hand side of \eqref{estimthetax} is given by
\begin{align}\label{alfal}
\nonumber
& \sum_{l\geq 0}\frac{x^l}{l!}\Big(\sum_{k=0}^{l}C_l^k \partial_y\theta_{k}\partial_y \theta_{l-k}+b(y)(\sum_{k=0}^{l-1}C_{l-1}^k \partial_y\theta_{k}\partial_y \theta_{l-1-k})-\sum_{k=0}^{l}C_l^k \zeta_{k}d_{l-k}\Big) \\
&=(\partial_y\theta_0)^2-\zeta_0\zeta_1^2+ \sum_{l\geq 1}\frac{x^l}{l!}\sum_{k=0}^{l}C_l^k \Big(\partial_y\theta_{k}\partial_y \theta_{l-k}+b(y)\frac{(l-k)}{l}\partial_y\theta_{k}\partial_y \theta_{l-1-k}-\zeta_k d_{l-k}\Big).
\end{align}
The second factor in \eqref{estimthetax} is
\begin{align}\label{betaj}
& \sum_{j\geq 0}\frac{x^j}{j!}\Big(\sum_{k=0}^{j}C_j^k \zeta_{k+1}\zeta_{j+1-k}+\sum_{k=1}^{j}C_j^k \partial_y\zeta_{k}\partial_y \zeta_{j-k}+b(y)(\sum_{k=1}^{j-1}C_{j-1}^k \partial_y\zeta_{k}\partial_y \zeta_{j-1-k})\Big)\\
& = \zeta_1^2+x(2\zeta_1\zeta_2)+\sum_{j\geq 2}\frac{x^j}{j!}\sum_{k=0}^{j}C_j^k \Big(\zeta_{k+1}\zeta_{j+1-k}+\partial_y\zeta_{k}\partial_y \zeta_{j-k}+b(y)\frac{(j-k)}{j}\partial_y\zeta_{k}\partial_y \zeta_{j-1-k}\Big).
\end{align}
The equation \eqref{estimthetax} now yields
\begin{equation}\label{eqtheto}
(\partial_y\theta_0)^2=\tau^2+\zeta_0\zeta_1^2,\quad l=j=0,
\end{equation}
\begin{equation}\label{eqzet1tau}
\zeta_1^3=b\tau^2+\zeta_0(b\zeta_1^2-2\zeta_1\zeta_2),\quad l+j=1.
\end{equation}
Notice from the formula of $\zeta_0$ that for $(\eta,\tau)$ in a small, conic neighborhood of $(1,-1)$ the contribution of $\zeta_0(\eta,\tau)$ will be very small. In particular, $\zeta_1(y,\eta,-\eta)=b^{1/3}(y)\eta^{2/3}$, and hence for $(\eta,\tau)$ close to $(1,-1)$ we can define $\zeta_1$ modulo $O(\zeta_0)$ terms.

If we denote $\alpha_l$ (resp. $\beta_j$) the coefficient of $\frac{x^l}{l!}$ in \eqref{alfal} (resp. $\frac{x^j}{j!}$ in \eqref{betaj}) we have $\alpha_0=\tau^2$, and \eqref{estimthetax} writes
\begin{equation}\label{thetaxp}
\sum_{l=0}^p C_p^l \alpha_l\beta_{p-l}=\tau^2 d_p,\quad \forall p\geq 2.
\end{equation}
For $p\geq 2$, the coefficient of  $\zeta_{p+1}$ in $\alpha_0\beta_p$ is $C_p^0\alpha_0=\tau^2$, and this cancels the one involving $\zeta_{p+1}$ in $\tau^2d_p$. Hence the only remaining terms in \eqref{thetaxp} involving $\zeta_{p+1}$ is the one in $\alpha_p\beta_0$, which is a multiple of $\zeta_0$ and hence very small.

On the other hand, in the sum \eqref{thetaxp} there are only terms involving $(\theta_k)_{k\in\{0,..,l\}}$ and they are defined, according to the Remark \ref{remtheta} (and a recurrent formula involving \eqref{eqtheto} and \eqref{eqthetl}), in terms of $(\zeta_k)_{k\leq l-1}$. 

We want to obtain a formula of $\zeta_p$. It appears three times in \eqref{thetaxp}, in the product $\alpha_p\beta_0$ (with coefficient $-\zeta_1^4$) and in both $\alpha_0\beta_p$ and $\tau^2 d_p$, the last two contribution being equal and hence canceling each other. Therefore we can define $\zeta_p(y,\eta,-\eta)$ from \eqref{thetaxp} for all $p\geq 2$, and also $\zeta_p(y,\eta,\tau)$ for $(\eta,\tau)$ close to $(1,-1)$ (modulo $O(\zeta_0)$ terms). 

By recurrence, we obtain an asymptotic expansion of $\zeta_1(y,\eta,\tau)$ in terms of $\zeta_0$ near the glancing point $(\eta,\tau)=(1,-1)$ (where $\zeta_0=0$).

\subsection{A model operator}\label{sectmodel}
In \cite{doi} we proved Theorem \ref{thmstrichartz} in the case of a two-dimesional, strictly convex domain $\Omega_{F}=\{(x,y)\in\mathbb{R}_{+}\times\mathbb{R}\}$ with Laplace operator given by
\begin{equation}\label{pfried}
\Delta_{F}=\partial^{2}_{x}+(1+x)\partial^{2}_{y}.
\end{equation}
\begin{rmq}
In this paper we want to construct examples for general manifolds with a gliding ray, but the heart of the matter is well illustrated by the particular example studied in \cite{doi} which will generalize using Melrose's equivalence of glancing hypersurfaces theorem.  Therefore we start by recalling the main steps in the construction of \cite{doi} and then use the particular solution of the model case to define an approximate solution for the more general domain described in Section \ref{secdoide}.
\end{rmq}
Let $X_{F}=\Omega_{F}\times\mathbb{R}$ and let $p_{F}\in C^{\infty}(T^{*}X_{F}\setminus o)$ denote the homogeneous symbol of the model wave operator $\square_{F}$, $p_{F}(x,y,t,\xi,\eta,\tau)=\xi^{2}+(1+x)\eta^{2}-\tau^{2}$. Consider the wave equation 
\begin{equation}\label{undef}
\left\{
\begin{array}{ll}
\partial^{2}_{t}v-\partial^{2}_{x}v-(1+x)\partial^{2}_{y}v=0,\\
v|_{\partial\Omega_{F}\times [0,T]}=0.
\end{array}
\right.
\end{equation}
The particular manifold $\Omega_F$ with metric inherited from the Laplace operator $\Delta_F$ studied in \cite{doi} is one for which the eigenmodes are explicitly expressed in terms of Airy's function and can be written as
\[
e^{iy\eta}Ai(|\eta|^{2/3}x-\omega_{k}),
\] 
where the Dirichlet condition dictates that $-\omega_{k}$ have to be the zeroes of the Airy function $Ai(-\omega_{k})=0$. Rewriting the mode in the form 
\[
e^{iy\eta}Ai(|\eta|^{2/3}(x-a)),
\]
for some parameter $a$, the eigenvalue is $-\eta(1+a)^{1/2}$, which means that such a wave moves with velocity $(1+a)^{1/2}$ in the $y$ direction. Writing the Airy function in terms of its integral formula, the corresponding wave writes under the form
\[
\int e^{i(y\eta-t\eta(1+a)^{1/2}+\xi(x-a)+\frac{\xi^3}{3})}d\xi d\eta.
\]
If we impose that $\eta^{2/3}a$ to be equal to some fixed zero of the Airy function $-\omega_k$, this will imply that the corresponding wave belongs to the regime of gallery modes described in \cite[Section 2]{doi}, for which sharp Strichartz and dispersive estimates hold as shown in \cite[Thm 1.8]{doi}. To construct wave-fronts that do not disperse requires superimposing waves with the same value of $a$. If one ignores the boundary condition for the moment, the superposition of such waves over a range $\eta\simeq 1/h$ would give, as can be seen by the asymptotic expansions of the Airy function, a solution living in an $h$-depending neighborhood of the cusp 
\[
y-(1+a)^{1/2}t=\pm|a-x|^{3/2},\quad x\in [0,a].
\]
The goal is to construct a similar solution that satisfies boundary conditions at $x=0$, while taking $a$ as small as possible depending on $h$. Rather than attempt to deal with the zeros of the Airy function, the boundary conditions  are met by taking a superposition of localized cusp solutions, each term in the sum being chosen to cancel off the boundary values of the previous one. Therefore we construct a parametrix for \eqref{undef} as a sum
\[
U_{F,h}(x,y,t)=\sum_{n=0}^{N}u^{n}_{F,h}(x,y,t),
\]
where $u^n_{F,h}$ are cusp type solutions of the form
\[
u^{n}_{F,h}(x,y,t)=\int_{\xi,\eta} e^{\frac{i}{h}(y\eta-t\eta(1+a)^{1/2}+\xi(x-a)+\frac{\xi^3}{3}+\frac{4}{3}\eta na^{3/2})}g^{n}_{F}(t,\xi,\eta,h)\Psi(\eta)d\xi d\eta,
\]
where $g_F^n$ are smooth functions compactly supported for $\eta$ near $1$ and where the relation between the amplitudes in the sum is dictated by the billiard ball maps. In this model case, due to the presence of the translations in $(y,t)$, the billiard ball maps have specific formulas
\begin{equation}\label{billball}
\delta^{\pm}_{F}(y,t,\eta,\tau)=\Big(y\pm4(\frac{\tau^{2}}{\eta^{2}}-1)^{1/2}\pm\frac{8}{3}(\frac{\tau^{2}}{\eta^{2}}-1)^{3/2},t\mp 4(\frac{\tau^{2}}{\eta^{2}}-1)^{1/2}\frac{\tau}{\eta},\eta,\tau\Big).
\end{equation}
After the change of variables $\xi\rightarrow\eta^{1/3}\xi$ the phase functions of $u^n_{F,h}$ become homogeneous of degree one and the associated Lagrangian sets are defined by
\begin{equation}
\Lambda_{F,n}:=\{\xi^{2}+(x-a)=0,y-t(1+a)^{1/2}+\xi(x-a)+\frac{\xi^{3}}{3}+\frac{4}{3}na^{3/2}=0\}\subset T^{*}X_{F}\setminus o.
\end{equation}
\begin{rmq}
Notice that, on the boundary, $\Lambda_{F,n}$ is the graph of the canonical transformation $(\delta^{\pm}_F)^n$ and writes as the composite relation with $n$ factors $\Lambda_{F,0}\circ...\circ\Lambda_{F,0}$.
The iterated billiard ball maps are given by the formula
\begin{equation}\label{deltait}
(\delta^{\pm}_{F})^{n}(y,t,\eta,\tau)=(y\pm4n(\frac{\tau^{2}}{\eta^{2}}-1)^{1/2}\pm\frac{8}{3}n(\frac{\tau^{2}}{\eta^{2}}-1)^{3/2},t\mp 4n(\frac{\tau^{2}}{\eta^{2}}-1)^{1/2}\frac{\tau}{\eta},\eta,\tau).
\end{equation}
\end{rmq}

\begin{figure}
\begin{center}
\includegraphics[width=15cm]{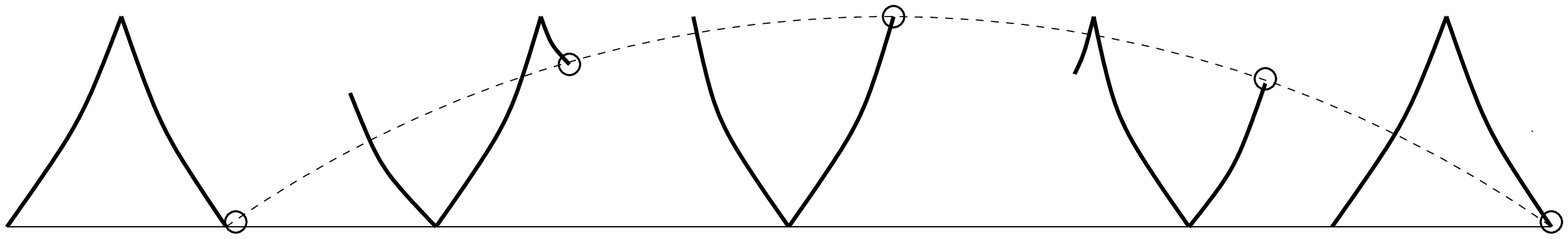}
\caption{Propagation of the cusp}
\label{fig}
\end{center}
\end{figure}

We introduce the fold set $\Sigma_F=\{\xi=0\}$ as the set of singular points of the canonical projection $\Lambda_{F,n}\rightarrow X_F$ and we define the caustic set to be the image of $\Sigma_F$ through this canonical projection, hence the set $\{x-\frac{\tau^{2}-\eta^{2}}{\eta^2}=0\}=\{\zeta_F(x,y,\eta,\tau)=0\}$.
Near the caustic set $\{\zeta_F=0\}$ each solution $u^n_{F,h}$ lives essentially on a cusp defined by
\[
y-t(1+a)^{1/2}+\frac 43 na^{3/2}=\pm(a-x)^{3/2}.
\]
The key observation is that, if the parameter $a$ is small enough, depending on the frequency, each such cusp type solution provides a loss in the Strichartz estimates (loss which increases when $a$ gets smaller).

If the symbols $g_F^n$ could be  chosen such that at time $t=4na^{1/2}$ to localize in a fixed neighborhood of the caustic set, then the respective "pieces of cusps" will propagate until they will reach the boundary but that shortly after that their contribution will become $O_{L^{2}}(h^{\infty})$, since as $t$ increases, one quickly quits a neighborhood of the Lagrangian $\Lambda_{F,n}$ which contains the semi-classical wave front set $WF_{h}(u^{n}_{F,h})$. The choice of the symbols is perfectly determined by the Dirichlet condition.

In fact on the boundary the phases functions have two critical, non-degenerate points, thus each $u^{n}_{F,h}$ writes as a sum of two trace operators, $\text{Tr}_{\pm}(u^{n}_{F,h})$, and in order to obtain a contribution $O_{L^{2}}(h^{\infty})$ on the boundary we define the symbol $g^{n+1}_{F}$ such that the contribution of  $\text{Tr}_{+}(u^{n+1}_{F,h})$ to counteract the contribution $\text{Tr}_{-}(u^{n}_{F,h})$ of the $n$th cusp. This is possible by Egorov theorem, as long as $N\ll a^{3/2}/h$. This last condition, together with the assumption of finite time $T=1$, allows to estimate the number of iterations $N$ and the parameter $a$.

\subsubsection{Construction of an approximate solution in the model case}\label{secmodelconstruction}
In this section we recall the construction from \cite{doi} of the symbols $g_F^n$ of the cusp type parametrices $u_{F,h}^n$ for the model wave operator $\square_F$,
\begin{equation}\label{integral}
u^n_{F,h}(x,y,t)=\int e^{\frac{i}{h}\eta(y-t(1+a)^{1/2}+\xi(x-a)+\frac{\xi^{3}}{3}+\frac 43 na^{3/2})}g^n_{F}(t,\xi,\eta,h) d\xi d\eta.
\end{equation}
Since our main observation was that for such parametrices the loss in the Stricharz estimates increases when $a$ decreases, we will take $a$ to be a strictly positive power of the parameter $h$, $a\simeq h^{\alpha}$  for some $\alpha$ to be chosen later as large as possible in the interval $(0,2/3)$.
\begin{rmq}
Notice that the case $a\simeq h^{2/3}$ corresponds to the whispering gallery modes: in this case it follows from \cite{doi} that losses cannot occur in the Strichartz estimates. The opposite case $a\simeq 1$ describes a wave transverse to the boundary: in this hyperbolic situation the solution to the wave equation can be easily determined as a sum of an incoming and an outgoing wave and one can easily show that it satisfies sharp Strichartz estimates.
\end{rmq}
Applying the wave operator $\square_{F}$ to $u^n_{F,h}$ and integrating by parts with respect to $\xi$ using the eikonal system \eqref{sistem1} gives
\begin{multline}\label{aprosF}
\square_{F} u^n_{F,h}(x,y,t)
=\int e^{\frac{i}{h}\eta(y-t(1+a)^{1/2}+\xi(x-a)+\frac{\xi^{3}}{3}+\frac 43 na^{3/2})}\\\times \Big(\frac{i}{h}\eta(\partial_{\xi}g^n_{F}
-2(1+a)^{1/2}\partial_{t}g^n_{F})+\partial^{2}_{t}g^n_{F}\Big)d\xi d\eta.
\end{multline}
In order to solve the first transport equation which appears in \eqref{aprosF}, we shall choose $g^n_F$ to be smooth functions depending on the integral curve of the vector field $\partial_{\xi}-2(1+a)^{1/2}\partial_t$.

\begin{dfn}\label{dfns}
Let $\lambda\geq 1$. For a given compact $K\subset\mathbb{R}$ we define the space $\mathcal{S}_{K}(\lambda)$, consisting of functions $\varrho(z,\lambda)\in C^{\infty}(\mathbb{R})$ which satisfy
\begin{enumerate}
\item  $\sup_{z\in\mathbb{R},\lambda\geq 1}|\partial^{\alpha}_{z}\varrho(z,\lambda)|\leq C_{\alpha}$, where $C_{\alpha}$ are constants independent of $\lambda$,
\item If $\psi(z)\in C^{\infty}_{0}$ is a smooth function equal to $1$ in a neighborhood of $K$, $0\leq\psi\leq 1$ then $(1-\psi)\varrho\in O_{\mathcal{S}(\mathbb{R})}(\lambda^{-\infty})$.
\end{enumerate}
Here $\mathcal{S}(\mathbb{R})$ denotes the Schwartz space of rapidly decreasing functions.

\textit{Example:} An example of function $\varrho(z,\lambda)\in \mathcal{S}_{K}(\lambda)$, $K\subset\mathbb{R}$ is the following: let $k(z)$ be the smooth function on $\mathbb{R}$ defined by
\[
k(z)=\left\{
                \begin{array}{ll}
               C\exp{(-1/(1-|z|^{2}))},\quad\ \text{if}\ |z|<1,\\
               0,\quad \ \text{if} \ |z| \geq 1,
                \end{array}
                \right.
\]
where $C$ is a constant chosen such that $\int_{\mathbb{R}}k(z)dz=1$. Define a mollifier $k_{\lambda}(z):=\lambda k(\lambda z)$ and let $\tilde{\varrho}\in C^{\infty}_{0}(K)$ be a smooth function with compact support included in $K$. If we set $\varrho(z,\lambda)=(\tilde{\varrho}*k_{\lambda})(z)$, then one can easily check that $\varrho$ belongs to $\mathcal{S}_{K}(\lambda)$.
\end{dfn}
We define a new parameter $\lambda=a^{3/2}/h$ (large since $a\gg h^{2/3}$) and for some small (but fixed) $0<c_{0}\leq \frac 38$ we set $K_{0}=[-c_{0},c_{0}]$.  We take $\varrho(.,\lambda)\in \mathcal{S}_{K_{0}}(\lambda)$ be a smooth function belonging to the space $\mathcal{S}_{K_{0}}(\lambda)$ defined in the Definition \ref{dfns}. We set
\begin{equation}\label{gforme}
g^0_{F}(t,\xi,\eta,h)=\varrho\Big(\frac{t+2(1+a)^{1/2}\xi}{2(1+a)^{1/2}a^{1/2}},\lambda\Big)\Psi(\eta),
\end{equation}
where $\Psi\in C^{\infty}_{0}(\mathbb{R}\setminus \{0\})$ is supported in a small neighborhood of $1$ and $0\leq\Psi(\eta)\leq 1$. In what follows we use the boundary condition to determine the symbols $g^n_F$ for every $0\leq n\leq N$.
\begin{prop}\label{bound}
On the boundary $u^0_{F,h}$ writes (modulo $O_{L^{2}}(\lambda^{-\infty})$) as a sum of two trace operators,
\begin{equation}\label{boun}
u^0_{F,h}(0,y,t)=\sum_{\pm}Tr_{\pm}(u^0_{F,h})(y,t),
\end{equation}
where
\begin{multline}\label{optragen}
Tr_{\pm}(u^0_{F,h})(y,t):=h^{1/3}\int_{\eta} e^{\frac{i}{h}(y\eta-t(1+a)^{1/2}\eta\mp\frac{2}{3} a^{3/2}\eta)}\Psi(\eta)(\eta\lambda)^{-1/6} \times \\ \times I_{\pm}(\varrho(.,\lambda))_{\eta}\Big(\frac{t}{2(1+a)^{1/2}a^{1/2}},\lambda\Big)d\eta,
\end{multline}
with $I_{\pm}(\varrho(.,\lambda))_{\eta}(z,\lambda)$ given by
\begin{equation}\label{alfapm}
I_{\pm}(\varrho(.,\lambda))_{\eta}(z,\lambda)=
\frac{\eta\lambda}{2\pi}\int_{w}e^{i\eta\lambda(w(z-z')\mp\frac{2}{3}((1-w)^{3/2}-1))}\kappa(w)a_{\pm}(w,\eta\lambda)\varrho(z',\lambda)dw.
\end{equation}
Here $\kappa$ is a smooth function supported for $w$ in a small neighborhood of $0$  such that $0\leq \kappa\leq 1$ and $\kappa(w)=1$ for $w$ close to $0$. In the integral \eqref{alfapm}, $a_{\pm}$ are given by the asymptotic expansions of  the Airy function $Ai(-(\lambda\eta)^{2/3}(1-w))$
\[
a_{\pm}(w,\eta,\lambda)\simeq e^{\pm i\pi/2-i\pi/4}(1-w)^{-1/4}\sum_{j\geq 0}a_{\pm,j}(-1)^{-j/2}(1-w)^{-3j/2}(\eta\lambda)^{-j},
\]
Precisely, we used the decomposition 
\[
Ai(-(\lambda\eta)^{2/3}(1-w))=\sum_{\pm}A_{\pm}(-(\lambda\eta)^{2/3}(1-w)),
\]
where $A_{\pm}(z)=Ai(e^{\mp 2\pi i/3}z)$. In particular, using the properties of the Airy functions $A_{\pm}$ it follows that the symbols $k(w)a_{\pm}(w,\eta\lambda)$ are elliptic at $w=0$ (see \cite[Appendix]{doi} for details).
\end{prop}
The proof of Proposition \ref{bound} is given in \cite[Proposition 3.3]{doi} (see also \cite[Lemma 3.2]{doi}).
\begin{prop}\label{lemlocal}
Let $p\in\mathbb{Z}$ and $K_{p}=[-c_{0}+p,c_{0}+p]$. Then for $\eta$ belonging to the support of $\Psi$ we have 
\[
I_{\pm,\eta}:\mathcal{S}_{K_{p}}(\lambda)\rightarrow\mathcal{S}_{K_{p\mp1}}(\lambda).
\]
\end{prop}
The proof of Proposition \ref{lemlocal} is given in (\cite[Lemma 3.4]{doi}).
\begin{prop}
Let $\eta$ belong to the support of $\Psi$ and let $J_{\pm,\eta}$ be the operators defined for some $\tilde{\lambda}\geq1$ and $\breve{\varrho}\in\mathcal{S}_{K\mp 1}(\tilde{\lambda})$ by the formula
\begin{equation}\label{opbbb}
J_{\pm}(\breve{\varrho}(.,\tilde{\lambda}))_{\eta}(z',\lambda):=\frac{\eta\lambda}{2\pi}\int e^{i\eta\lambda((z'-z)w\pm\frac{2}{3}((1-w)^{3/2}-1))}b_{\pm}(w,\eta\lambda)\breve{\varrho}(z,\tilde{\lambda})dzdw,
\end{equation}
where $b_{\pm}(w,\eta\lambda)=\frac{k(w)}{a_{\pm}(w,\eta\lambda)}$ are asymptotic expansions in $(\eta\lambda)^{-1}$. Then $J_{\pm,\eta}$ satisfy 
\[
\breve{\varrho}(.,\tilde{\lambda})=I_{\pm}(J_{\pm}(\breve{\varrho}(.,\tilde{\lambda}))_{\eta}(.,\lambda))_{\eta}(.,\lambda)+O_{\mathcal{S}(\mathbb{R})}(\lambda^{-\infty})+O_{\mathcal{S}(\mathbb{R})}(\tilde{\lambda}^{-\infty}),
\]
\[
\varrho(.,\tilde{\lambda})=J_{\pm}(I_{\pm}(\varrho(.,\tilde{\lambda}))_{\eta}(.,\lambda))_{\eta}(.,\lambda)+O_{\mathcal{S}(\mathbb{R})}(\lambda^{-\infty})+O_{\mathcal{S}(\mathbb{R})}(\tilde{\lambda}^{-\infty}).
\]
\end{prop}
The construction of the operators $J_{\pm,\eta}$ is detailed in \cite[Section 3.3.1]{doi}.
\begin{prop}\label{propimportant1}
Let $N\lesssim \lambda h^{\epsilon}$ for some small $\epsilon>0$ and let $1\leq n\leq N$. Let $T_{k}$ denote the translation operator which to a given function $\varrho(z)$ associates $\varrho(z+k)$. Then for $\eta\in \text{supp}(\Psi)$ we have
\begin{equation}
(T_1\circ J_{+}(.)_{\eta}\circ I_{-}(.)_{\eta}\circ T_1)^{\circ n}:\mathcal{S}_{K_{0}}(\lambda)\rightarrow\mathcal{S}_{K_{0}}(\lambda/n)\quad \text{uniformly in}\quad  n.
\end{equation}
Notice that since $\lambda/n\geq h^{-\epsilon}\gg 1$, then $O_{\mathcal{S}(\mathbb{R})}(\lambda^{-\infty})=O_{\mathcal{S}(\mathbb{R})}((\lambda/n)^{-\infty})=O_{\mathcal{S}(\mathbb{R})}(h^{\infty})$.

Moreover, the operator defined above writes as a convolution 
\[
(T_{1}\circ J_{+}(.)_{\eta}\circ I_{-}(.)_{\eta}\circ T_{1})^{\circ n}(\varrho)=(F_{\eta\lambda})^{*n}*\varrho,
\] 
where
\begin{equation}\label{Fn}
(F_{\eta\lambda})^{*n}(z)=
\frac{\eta\lambda}{2\pi}\int_{w}e^{i\eta\lambda(wz+n(2w+\frac{4}{3}((1-w)^{3/2}-1))} \Big(\kappa(w)a_{+}(w,\eta\lambda)b_{-}(w,\eta\lambda)\Big)^{n}dw.
\end{equation}
\end{prop}
The proof of Proposition \ref{propimportant1} is given in \cite[Proposition 3.6]{doi}.
\begin{dfn}\label{dfncuspn}
Let $\varrho(.,\lambda)\in\mathcal{S}_{K_{0}}(\lambda)$ and $\eta\in \text{supp}(\Psi)$. For $1\leq n\leq N$, $N\lesssim\lambda h^{\epsilon}$ set
\[
\varrho^{n}(z,\eta,\lambda):=(-1)^{n}(T_{1}\circ J_{+}(.)_{\eta}\circ I_{-}(.)_{\eta}\circ T_{1})^{n}(\varrho(.,\lambda))(z),\quad \varrho^{0}(z,\eta,\lambda)=\varrho(z,\lambda).
\]
\begin{rmq}
From Proposition \ref{propimportant1} it follows that $\varrho^{n}(z,\eta,\lambda)\in\mathcal{S}_{K_{0}}(\lambda/n)$. 
\end{rmq}
\end{dfn}
\begin{dfn}\label{dfnsimbadmodn}
For $0\leq n\leq N$ with $N\lesssim\lambda h^{\epsilon}$ define 
\begin{equation}\label{forun1}
g^n_F(t,\xi,\eta,h):=\varrho^{n}\Big(\frac{t+2(1+a)^{1/2}\xi}{2(1+a)^{1/2}a^{1/2}}-2n,\eta,\lambda\Big)\Psi(\eta).
 \end{equation}
\end{dfn}
\begin{prop}\label{proptracemod}
We have for all $0\leq n\leq N-1$
\begin{equation}\label{tracebord}
Tr_{-}(u^{n}_{F,h})(y,t)+Tr_{+}(u^{n+1}_{F,h})(y,t)=O_{L^{2}}(\lambda^{-\infty}).
\end{equation} 
\end{prop}
\begin{proof}
The equality \eqref{tracebord} follows from the relation
\[
I_{-}(T_{1}(\varrho^{n}(.,\eta,\lambda)))_{\eta}+I_{+}(T_{-1}(\varrho^{n+1}(.,\eta,\lambda)))_{\eta}=O_{\mathcal{S}(\mathbb{R})}(\lambda^{-\infty}),
\]
together with the fact that the operators $I_{\pm,\eta}$ are of convolution type and therefore they commute with translations.
\end{proof}
\begin{prop}\label{propsuppunfh}
If $0\leq n\leq N$, $u^{n}_{F,h}(.,y,t)$ is essentially supported for $y$ and $t$ in the interval
\begin{equation}\label{inco}
I_n(c_0):=2a^{1/2}(1+a)^{1/2}\times [2n-(1+c_{0}),2n+(1+c_{0})],
\end{equation}
i.e. for $y$ or $t$ outside any neighborhood of $I_n(c_0)$ the contribution of $u^n_{F,h}$ is $O_{L^{2}}(h^{\infty})$.
\end{prop}

\subsection{Construction of an approximate solution in the general case}\label{subsectfioapplic}

In this section we construct an approximate solution to \eqref{undered3} satisfying the conditions of the Theorem \ref{thms3}. It will be essentially based on the model construction recalled in Section \ref{sectmodel} together with Theorem \ref{thmphase}. 

Inspired from the Friedlander's case, we construct a superposition of localized cusp type solutions $u^n_h$ to \eqref{undered3} of the form \eqref{solgen}, each term in the sum over $n$ being chosen to cancel off the boundary values of the previous one. Precisely, we take $u^{n}_h$ of the form
\begin{equation}\label{solgen1}
u^{n}_{h}(x,y,t)=\int_{\xi,\eta,\tau}e^{\frac{i}{h}\Phi^{n}(x,y,t,\xi,\eta,\tau)}g^{n}_{h}d\xi d\eta d\tau
\end{equation}
for some symbols $g^{n}_{h}$ to be suitably chosen and where the phase functions are given by
\begin{equation}\label{dfnphangen}
\Phi^{n}(x,y,t,\xi,\eta,\tau):=\theta(x,y,t,\eta,\tau)+\eta^{1/3}\xi\zeta(x,y,\eta,\tau)+\eta\frac{\xi^{3}}{3}+\frac{4}{3}n(-\zeta_{0})^{3/2}(\eta,\tau).
\end{equation}
Notice that on the boundary the Lagrangian manifolds $\Lambda_n$ associated to $\Phi^n$ are the graphs of the billiard ball maps $(\delta^{\pm})^n$ described at the end of Section \ref{secphases}.
\begin{rmq}
Away from the caustic set defined by the locus where $\{\xi=0\}$ and $\{\zeta=0\}$, there are two main contributions in $u^0_{h}$ with phase functions $\theta\mp \frac{2}{3}(-\zeta)^{3/2}$. These are the phases corresponding to the Airy functions $A_{\pm}(\zeta)$ and one can think (at least away from the boundary) of the part corresponding to $A_{-}(\zeta)$ as a free wave or the "incoming piece": after hitting the boundary it gives rise to the outgoing one which corresponds to $A_{+}(\zeta)\frac{A_{-}(\zeta_{0})}{A_{+}(\zeta_{0})}$ with phase $-\frac{2}{3}(-\zeta)^{3/2}+\frac{4}{3}(-\zeta_{0})^{3/2}$. The oscillatory part $\frac{4}{3}(-\zeta_{0})^{3/2}$ corresponds to the billiard ball map shift corresponding to reflection. After $n$ reflections the shift is $\frac 43 n(-\zeta_0)^{3/2}$ and the Larangian $\Lambda_n=(\Lambda_0)^{\circ n}$ is parametrized by $\Phi^n$ defined in \eqref{dfnphangen}.
\end{rmq}
 
\subsubsection{The boundary condition}\label{sectbdcond}
We determine the symbols $g^{n}_{h}$ in \eqref{solgen1} such that $u^{n}_{h}$ to be approximate solution to \eqref{undered3} in a sense to be precise. We start by defining their restriction to the boundary by requiring the Dirichlet condition to be fulfilled. We consider an operator $J$ as follows
\begin{equation}\label{fioj}
J(f)(y,t):=\frac{1}{(2\pi h)^{2}}\int_{\eta,\tau}e^{\frac{i}{h}\theta_{0}(y,t,\eta,\tau)}d_{h}(y,\eta,\tau)\widehat{f}(\eta/h,\tau/h)d\eta d\tau,
\end{equation}
where $d_{h}(y,\eta,\tau)=d(y,\eta/h,\tau/h)$ for some elliptic symbol $d(y,\eta,\tau)$ of order $0$ and type $(1,0)$, compactly supported in a conic neighborhood of the glancing point $\pi(\rho_{0},\vartheta_{0})$. Here $\theta_{0}$ denotes the restriction to the boundary of the phase function $\theta$ introduced in Theorem \ref{thmphase}. 

The operator $J$ defines an elliptic Fourier integral operator in a neighborhood of the glancing point $(\pi(\bar{\rho}_{0},\bar{\vartheta}_{0}),\pi(\rho_{0},\vartheta_{0}))$, with canonical relation $\chi_{\partial}$ given by the symplectomorphisme generated by $\theta_{0}$ which satisfies $\chi_{\partial}(\pi(\bar{\rho}_{0},\bar{\vartheta}_{0}))=\pi(\rho_{0},\vartheta_{0})$ (see the remarks following Theorem \ref{thmphase}). 
\begin{rmq}
Using Proposition \ref{proptracemod} it is clear that if we define the two contributions of $u^{n}_h$ to the boundary to be equal to $J\circ Tr_{\pm}(u^{n}_{F,h})$, then the sum over $n$ of $u^n_h$ will verify the Dirichet condition.
\end{rmq}
In what follows we compute $J\circ Tr_{\pm}(u^{n}_{F,h})$, where $0\leq n\leq N$ for some $N$ to be determined later. We keep the notations of  Section \ref{sectmodel}.
\begin{prop}\label{lemjfhn}
On the boundary $J\circ Tr_{\pm}(u^{n}_{F,h})$ writes
\begin{multline}\label{uhbondarprop}
J\circ Tr_{\pm}(u^{n}_{F,h})(y,t)= h^{1/3}\int e^{\frac{i}{h}(\theta_{0}(y,t,\eta,-\eta(1+a)^{1/2})+\frac{2}{3}(2n\mp 1)(-\zeta_{0})^{3/2}(\eta,-\eta(1+a)^{1/2}))}(\lambda\eta)^{-1/6}\\ \times
 I_{\pm}(g^{n}_{h}(.,y,\eta))_{\eta}\Big(\frac{\partial_{\tau}\theta_{0}(y,t,\eta,-\eta(1+a)^{1/2})}{2(1+a)^{1/2}a^{1/2}}-2n,\lambda\Big)d\eta,
\end{multline}
where
\begin{equation}\label{sigmndef}
g^{n}_{h}(z,y,\eta)
\simeq
\Psi(\eta)\Big(\sum_{k\geq 0} h^{k/2}a^{-k/2}\mu_{k}(y,\eta,h)\partial^{k}_z \varrho^{n}(z,\eta,\lambda)\Big).
\end{equation}
Here $\mu_{k}(y,\eta,h)$ are symbols of order $0$ and type $(1,0)$ independent of $n$. Moreover, if $\eta\in\text{supp}(\Psi)$ and $1\leq n\leq N\lesssim \lambda h^{\epsilon}$ for some small $\epsilon>0$ then $g^{n}_{h}(.,y,\eta)\in\mathcal{S}_{K_{0}}(\lambda/n)$. 
\end{prop}
\begin{proof}
We compute explicitely the restriction of each $u^{n}_{F,h}$ to $\{x=0\}$, for $0\leq n\leq N$. 
\begin{multline}\label{fnhint}
u^{n}_{F,h}(0,\bar{y},\bar{t})=
\int_{\xi,\bar{\eta}}e^{\frac{i}{h}\bar{\eta}(\bar{y}-\bar{t}(1+a)^{1/2}+\xi(x-a)+\frac{\xi^{3}}{3}+\frac{4}{3}na^{3/2})}\\ \times \Psi(\bar{\eta})\varrho^{n}\Big(\frac{\bar{t}+2(1+a)^{1/2}\xi}{2(1+a)^{1/2}a^{1/2}}-2n,\bar{\eta},\lambda\Big)d\xi d\bar{\eta}.
\end{multline}
Taking $\xi=a^{1/2}v$, the integral in $\xi$ in \eqref{fnhint} becomes, modulo $O_{\mathcal{S}(\mathbb{R})}(\lambda^{-\infty})$,
\begin{multline}\label{fnhint2}
\Psi(\bar{\eta})\frac{\bar{\eta}\lambda}{2\pi}\int_{z,w}e^{i\bar{\eta}\lambda w\Big(\frac{\bar{t}}{2(1+a)^{1/2}a^{1/2}}-2n-z\Big)}\varrho^{n}(z,\bar{\eta},\lambda)a^{1/2}\int_{v}e^{i\bar{\eta}\lambda(\frac{v^{3}}{3}-v(1-w))}dv dwdz\\=
\Psi(\bar{\eta})\frac{\bar{\eta}\lambda}{2\pi}\int_{z,w}e^{i\bar{\eta}\lambda w\Big(\frac{\bar{t}}{2(1+a)^{1/2}a^{1/2}}-2n-z\Big)}\varrho^{n}(z,\bar{\eta},\lambda) \\\times h^{1/3} Ai(-(\bar{\eta}\lambda)^{2/3}(1-w)) \kappa(w)dwdz,
\end{multline}
where $\kappa(w)\in C^{\infty}_{0}(\mathbb{R})$ is supported in a neighborhood of $0$, $0\leq \kappa\leq 1$ and $\kappa(w)=1$ for $w$ near $0$. Here $Ai$ denotes the Airy function. Indeed, an explicit computation shows that the integral in \eqref{fnhint} equals the term in the right hand side of \eqref{fnhint2} with $\kappa$ replaced by $1$. Using the same arguments as in \cite[Lemma 3.2]{doi} we easily see that for $w$ outside some fixed neighborhood of $0$, as small as we want, the contribution in the integral in \eqref{fnhint2} is $O_{\mathcal{S}(\mathbb{R})}(\lambda^{-\infty})$, therefore we can introduce the cut-off function $\kappa$ with the properties stated above in order to obtain the right hand side term in \eqref{fnhint2} modulo small terms.

We now decompose
$Ai(z)=A^{+}(z)+A^{-}(z)$, where
\begin{equation}\label{aipmdef}
A^{\pm}(-(\bar{\eta}\lambda)^{2/3}(1-w))= e^{\mp\frac{2}{3}i\bar{\eta}\lambda}(\bar{\eta}\lambda)^{-1/6} a_{\pm}(w,\bar{\eta},\lambda),
\end{equation}
where $a_{\pm}(w,\bar{\eta},\lambda)$ are the symbols defined in Proposition \ref{bound}. Hence, \eqref{fnhint} becomes
\begin{multline}\label{traunpmgen}
u^{n}_{F,h}(0,\bar{y},\bar{t})=
h^{1/3}\sum_{\pm}\int_{\bar{\eta}} e^{\frac{i}{h}(\bar{y}\bar{\eta}-\bar{t}(1+a)^{1/2}\bar{\eta}\mp\frac{2}{3}a^{3/2}\bar{\eta}+\frac{4}{3}na^{3/2}\bar{\eta})} \Psi(\bar{\eta})(\bar{\eta}\lambda)^{-1/6}\\\times I_{\pm}(\varrho^{n}(.,\bar{\eta},\lambda))_{\bar{\eta}}\Big(\frac{\bar{t}}{2(1+a)^{1/2}a^{1/2}}-2n,\lambda\Big)d\bar{\eta},
\end{multline}
where $I_{\pm}(\varrho^{n}(.,\bar{\eta},\lambda))_{\bar{\eta}}(z,\lambda)$ are defined in \eqref{alfapm}. 
The contributions corresponding to the $\pm$ signs in the right hand side of \eqref{traunpmgen} are denoted $Tr_{\pm}(u^{n}_{F,h})(\bar{y},\bar{t})$.

We can now proceed to compute $J\circ Tr_{\pm}(u^{n}_{F,h})(\bar{y},\bar{t})$.
\begin{multline}
J\circ Tr_{\pm}(u^{n}_{F,h})(y,t) =\frac{h^{1/3}}{(2\pi h)^{2}}\int e^{\frac{i}{h}(\theta_{0}(y,t,\eta,\tau)-\bar{y}(\eta-\bar{\eta})-\bar{t}(\tau+\bar{\eta}(1+a)^{1/2})\mp\frac{2}{3}a^{3/2}\bar{\eta}+\frac{4}{3}na^{3/2}\bar{\eta})}\Psi(\bar{\eta})\\ 
 \times  I_{\pm}(\varrho^{n}(.,\bar{\eta},\lambda))_{\bar{\eta}}\Big(\frac{\bar{t}}{2(1+a)^{1/2}a^{1/2}}-2n,\lambda\Big) (\bar{\eta}\lambda)^{-1/6}d_{h}(y,\eta,\tau)
d\bar{\eta}d\bar{y}d\bar{t}d\eta d\tau.
\end{multline}
Since the symbol is independent of $\bar{y}$, the integration in $\bar{y}$ gives $\eta=\bar{\eta}$.
Now we are in a situation where the stationary phase theorem can be applied in the variables $(\bar{t},\tau)$. We recall the stationary phase theorem: 
\begin{prop}\label{propstatphas}(\cite[Thm.7.7.7]{hormand})
Let $f(\omega,\tau)$ be a real valued function in $C^{\infty}(\mathbb{R}^{m+1})$, $\tau_{0}\in\mathbb{R}$. If $K$ is a compact subset of $\mathbb{R}^{2+m}$ and $\sigma\in C^{\infty}_{0}(K)$, then
\begin{multline}\label{sumphasstat}
\int_{\tau,\bar{t}}e^{\frac{i}{h}(f(\omega,\tau)-\bar{t}(\tau-\tau_{0}))}\sigma(\tau,\bar{t},\omega)d\tau d\bar{t}\\
\simeq (2\pi i)e^{\frac{i}{h}f(\omega,\tau_{0})}\sum_{\nu\geq 0}(ih)^{\nu}(\partial_{\tau}\partial_{\bar{t}})^{\nu}\Big(e^{\frac{i}{h}r(\omega,\tau)}\sigma(\tau,\bar{t},\omega)\Big)|_{\tau=\tau_{0},\bar{t}=\partial_{\tau}f(\omega,\tau_{0})}.
\end{multline}
Here $r(\omega,\tau)=f(\omega,\tau)-f(\omega,\tau_{0})-(\tau-\tau_{0})\partial_{\tau}f(\omega,\tau_{0})$. 
\end{prop}
\begin{rmq}\label{rmqthmstat}
In the asymptotic sum in \eqref{sumphasstat} the $\bar{t}$ derivative must act on $\sigma$, and $\tau$ derivatives acting on $e^{\frac{i}{h}r(\omega,\tau)}$ bring out with $h^{-1}$ a derivative of $r$ vanishing at $\tau=\tau_{0}$. Another $\tau$ derivative must act on it to give a non-zero contribution. This shows that the terms in the sum are $O(h^{\nu/2})$, for at most $\nu/2$ derivations bring out a factor $h^{-1}$.
\end{rmq}

We apply Proposition \ref{propstatphas} with $\omega=(y,t,\eta)\in\mathbb{R}^{3}$, $f(\omega,\tau)=\theta_{0}(\omega,\tau)$, $\tau_{0}=-\eta(1+a)^{1/2}$ and with symbol 
\[
\sigma(\tau,\bar{t},\omega):=d_{h}(y,\eta,\tau)I_{\pm}(\varrho^{n}(.,\eta,\lambda))_{\eta}\Big(\frac{\bar{t}}{2(1+a)^{1/2}a^{1/2}}-2n,\lambda\Big).
\]
Consequently $J\circ Tr_{\pm}(u^{n}_{F,h})$ admits the asymptotic expansions
\begin{multline}\label{uhbondar}
J\circ Tr_{\pm}(u^{n}_{F,h})(y,t)\simeq 
h^{1/3}\int e^{\frac{i}{h}(\theta_{0}(y,t,\eta,-\eta(1+a)^{1/2})+\frac{2}{3}(2n\mp 1)(-\zeta_{0})^{3/2}(\eta,-\eta(1+a)^{1/2}))} \Psi(\eta) \\ \times \Big[\sum_{k\geq 0} h^{k/2}a^{-k/2}\mu_k (y,\eta,h)  \partial^{k}I_{\pm}(\varrho^{n}(.,\eta,\lambda))_{\eta}\Big(\frac{\partial_{\tau}\theta_{0}(y,t,\eta,-\eta(1+a)^{1/2})}{2(1+a)^{1/2}a^{1/2}}-2n,\lambda\Big)\Big] \\\times (\eta\lambda)^{-1/6}\Psi(\eta)d\eta,
\end{multline}
where we set
\[
\mu_{k}(y,\eta,h)=i^{k}2^{-k}h^{k/2}(1+a)^{-k/2}\partial^{k}\Big(e^{\frac{i}{h}r(y,t,\eta,\tau)}d_{h}(y,\eta,\tau)\Big)|
_{\scriptsize
\left\{
\begin{array}{ll}
\tau=-\eta(1+a)^{1/2}, \\ \bar{t}=\partial_{\tau}\theta_{0}(y,t,\eta,-\eta(1+a)^{1/2}
\end{array}
\right.}.
\]
According to the Remark \ref{rmqthmstat} it follows that the main contribution of $\mu_{2\nu}$ is $(\partial^{2}_{\tau}r)^{\nu}d_{h}e^{\frac{i}{h}r(\omega,\tau)}$ and that of $\mu_{2\nu-1}$ is $h(\partial^{2}_{\tau}r)^{\nu}\partial_{\tau}d_{h}e^{\frac{i}{h}r(\omega,\tau)}$, all the other terms in the sum defining $\mu_{k}$ being positive powers of $h$. Since $d_{h}$ is a symbol of order $0$ and type $(1,0)$, we deduce that $\mu_{k}$ is a also symbol of order $0$ and type $(1,0)$.

Notice, moreover, that  $I_{\pm}(\varrho^{n})_{\eta}$ is a convolution product and consequently $\partial^{k}(I_{\pm}(\varrho^{n})_{\eta})=I_{\pm}(\partial^{k}\varrho^{n})_{\eta}$. Since from the Proposition \ref{propimportant1} and the Definition \ref{dfncuspn} the symbols $\varrho^{n}(.,\eta,\lambda)$ belong to $\mathcal{S}_{K_{0}}(\lambda/n)$, where $K_{0}=[-c_{0},c_{0}]$, it follows that the sum 
\[
\Psi(\eta)\Big(\sum_{k\geq 0} h^{k/2}a^{-k/2}\mu_{k}(y,\eta,h)\partial^{k}\varrho^{n}(z,\eta,\lambda)\Big)
\]
(denoted $g^{n}_{h}(z,y,\eta)$ in the statement of Proposition \ref{lemjfhn}) belongs also to $\mathcal{S}_{K_{0}}(\lambda/n)$.
This achieves the proof of Proposition \ref{lemjfhn}.
\end{proof}

\subsubsection{Transport equations} 
We need to determine the integral curves of the vector field 
$<2d\Phi^{n},d.>-\eta^{-1/3}<d\zeta,d\zeta>\partial_{\xi}$ which appear in the first transport equation associated to the wave operator $\square_g$.
\begin{lemma}\label{lemintcurves}
The functions
\begin{equation}\label{integralcurves}
\eta^{-2/3}\zeta+\xi^{2},\quad \partial_{\tau}\theta+\eta^{1/3} \xi\partial_{\tau}\zeta
\end{equation}
are integral curves of the vector field $<2d\Phi^{n},d.>-\eta^{-1/3}<d\zeta,d\zeta>\partial_{\xi}$,
where we recall that $\Phi^{n}$ is the homogeneous phase function
\[
\Phi^{n}(x,y,t,\xi,\eta,\tau)=\theta(x,y,t,\eta,\tau)+\eta^{1/3}\xi\zeta(x,y,\eta,\tau)+\eta\frac{\xi^{3}}{3}+\frac{4}{3}n(-\zeta_{0})^{3/2}(\eta,\tau).
\]
\end{lemma}
\begin{proof}
The Hamiltonian system writes
\begin{equation}
\left\{
      \begin{array}{ll}
      \dot{x}=2(\partial_{x}\theta+\eta^{1/3}\xi\partial_{x}\zeta),\\
      \dot{y}=2(1+xb(y))(\partial_{y}\theta+\eta^{1/3}\xi\partial_{y}\zeta),\\
      \dot{t}=-2\tau,\\
      \dot{\xi}=-\eta^{-1/3}<d\zeta,d\zeta>
      \end{array}
      \right.
\end{equation}
and we can compute the derivative of the first integral curve in \eqref{integralcurves} 
\begin{align}
\dot{\overbrace{(\xi^{2}+\eta^{-2/3}\zeta)}} & =2\dot{\xi}\xi+\eta^{-2/3}\dot{\zeta}=2\eta^{-1/3}<d\zeta,d\zeta>\xi+\eta^{-2/3}(\dot{x}\partial_{x}\zeta+\dot{y}\partial_{y}\zeta)\\ 
\nonumber
&=
2\eta^{-1/3}<d\zeta,d\zeta>\xi-2\eta^{-4/3}<d\theta,d\zeta>+2\eta^{-1/3}<d\zeta,d\zeta>\xi\\ 
\nonumber
& =0,
\end{align}
where we used the eikonal equations \eqref{sistem1}. For the second one we have
\begin{align}
\dot{\overbrace{(\partial_{\tau}\theta+\eta^{1/3}\partial_{\tau}\zeta)}} & =\dot{x}\partial^{2}_{\tau,x}\theta+\dot{y}\partial^{2}_{\tau,y}\theta+\dot{t}\partial^{2}_{\tau,t}\theta+\eta^{1/3} \xi(\dot{x}\partial^{2}_{\tau,x}\zeta+\dot{y}\partial^{2}_{\tau,y}\zeta)+\eta^{1/3}\dot{\xi}\partial_{\tau}\zeta\\ 
\nonumber
& =
\partial_{\tau}(<d\theta,d\theta>-\zeta<d\zeta,d\zeta>)+2\eta^{-2/3}\xi\partial_{\tau}<d\theta,d\zeta>\\ 
\nonumber
& +
\partial_{\tau}(\zeta<d\zeta,d\zeta>)
+\eta^{-4/3}\xi^{2}\partial_{\tau}<d\zeta,d\zeta>\\ 
\nonumber
& -<d\zeta,d\zeta>\partial_{\tau}\zeta=\eta^{2/3}(\xi^{2}+\eta^{-2/3}\zeta)\partial_{\tau}<d\zeta,d\zeta>\\ 
\nonumber
& =0
\end{align}
on the Lagrangian $\Lambda_n$ which contains the semi-classical wave front set $WF_{h}(u^{n}_{h})$,
\[
\Lambda_n:=\{(x,y,t,\xi,\eta,\tau)|\partial_{\xi}\Phi^{n}=0,\partial_{\eta}\Phi^{n}=0,\tau=-\eta(1+a)^{1/2}\}.
\]
\end{proof}

Now we are in the situation when we can define $u^{n}_{h}$ everywhere as follows:
\begin{dfn}
Let $g^{n}_{h}$ be the symbol defined in \eqref{sigmndef} and for $0\leq n\leq N\lesssim\lambda h^{\epsilon}$ let
\begin{multline}\label{uhndef}
u^{n}_{h}(x,y,t):= \int_{\xi,\eta} e^{\frac{i}{h}\Phi^{n}(x,y,t,\xi,\eta,-\eta(1+a)^{1/2})}\\ \times g^{n}_{h}\Big(\frac{\partial_{\tau}\theta+\eta^{1/3}\xi\partial_{\tau}\zeta}{2(1+a)^{1/2}a^{1/2}}(x,y,t,\eta,-\eta(1+a)^{1/2})-2n,y,\eta\Big)ds d\eta.
\end{multline}
\end{dfn}
\begin{rmq}\label{rmqzetaz}
Notice that since $\partial_{\tau}\theta$, $\eta^{1/3}\partial_{\tau}\zeta$ are homogeneous of degree $0$ in $(\eta,\tau)$, for $\tau=-\eta(1+a)^{1/2}$ they are independent of $\eta$, so the term in the first variable in $g^{n}_{h}$ depends only of $(x,y,t)$ and $a$. 
\end{rmq}
We now prove that the restriction to the boundary of $u^{n}_{h}$ defined in \eqref{uhndef} coincides indeed with the sum of the two terms in \eqref{uhbondar}. \begin{prop}\label{propboundjtracepm}
On the boundary $\partial\Omega$ we have, indeed,
\begin{equation}\label{egaltoprove}
u^{n}_{h}(0,y,t)=\sum_{\pm}J(Tr_{\pm}(u^{n}_{F,h}))(y,t,h).
\end{equation}
Moreover,
\begin{equation}\label{condbondgen}
J(Tr_{-}(u^{n}_{F,h}))(y,t)+J(Tr_{+}(u^{n+1}_{F,h}))(y,t)=O_{L^{2}}(h^{\infty}).
\end{equation}
\end{prop}
\begin{proof}
We first proceed with \eqref{condbondgen}.
Since $J$ is an elliptic Fourier integral operator the proof follows from Proposition \ref{proptracemod}, since 
\[
Tr_{-}(u^{n}_{F,h})(y,t,h)+Tr_{+}(u^{n+1}_{F,h})(y,t,h)=O_{L^{2}}(h^{\infty}).
\]
We now prove \eqref{egaltoprove}. At $x=0$ the integral in $\xi$ in \eqref{uhndef} writes as a sum over $k\geq 0$ of 
\begin{multline}\label{contribouh}
h^{k/2}a^{k/2}\mu_k(y,\eta,h)\times \Psi(\eta)\int e^{\frac{i}{h}(\eta^{1/3}\xi\zeta_{0}(\eta,-\eta(1+a)^{1/2})+\eta\frac{\xi^{3}}{3})}\\ \times \partial^{k}\varrho^{n}\Big(\frac{\partial_{\tau}\theta_{0}+\eta^{1/3}\xi\partial_{\tau}\zeta_{0}}{2(1+a)^{1/2}}(y,t,\eta,-\eta(1+a)^{1/2})-2n,\eta,\lambda\Big)d\xi.
\end{multline}
Using Remark \ref{rmqzetaz} and setting $\xi=a^{1/2}v$ in \eqref{contribouh}, each integral in $\xi$ becomes  
\begin{equation}\label{contribouh2}
\Psi(\eta) \frac{\eta\lambda}{2\pi}\int_{w}e^{i\eta\lambda w\Big(\frac{\partial_{\tau}\theta_{0}(y,t,\eta,-\eta(1+a)^{1/2})}{2(1+a)^{1/2}a^{1/2}}-2n-z\Big)}\partial^k\varrho^n(z,\eta,\lambda)   a^{1/2}\int_v e^{i\eta\lambda(v^3/3-v(1-w))}dvdw.
\end{equation}
The integral in $v$ writes $h^{1/3}a^{-1/2} Ai(-(\eta\lambda)^{2/3}(1-w))$ and consequently \eqref{contribouh2} becomes
\begin{multline}\label{contribouh1}
\Psi(\eta)\frac{\eta\lambda}{2\pi}\int e^{i\eta\lambda w\Big(\frac{\partial_{\tau}\theta_{0}(y,t,\eta,-\eta(1+a)^{1/2})}{2(1+a)^{1/2}a^{1/2}}-2n-z\Big)} \\\times h^{1/3} Ai(-(\eta\lambda)^{2/3}(1-w)) \kappa(w)\partial^{k}\varrho^n(z,\lambda)dzdw+O_{\mathcal{S}(\mathbb{R})}((\eta\lambda)^{-\infty}),
\end{multline}
where $\kappa$ is a smooth function supported for $w$ as close as we want to $0$, $0\leq \kappa\leq 1$, $\kappa(w)=1$ near $0$. Indeed, this follows as in the proof of Proposition \eqref{lemjfhn}: for $\kappa$ equal to $1$ near $0$ we can multiply the symbol of the integral in $w$ by $\kappa(w)+(1-\kappa(w))$ and performing integrations by parts we find that the contribution of the integral over $w$ away from a fixed neighborhood of $0$ is $O_{\mathcal{S}(\mathbb{R})}((\eta\lambda)^{-\infty})$. We now distinguish two contributions in \eqref{contribouh1}, obtained using the decomposition $Ai(z)=A^{+}(z)+A^{-}(z)$, where $A_{\pm}(z)$ are recalled in \eqref{aipmdef}. Consequently, \eqref{contribouh2} becomes, modulo $O_{\mathcal{S}(\mathbb{R})}((\eta\lambda)^{-\infty})$ terms,
\[
h^{1/3}\Psi(\eta)e^{\mp\frac{2}{3}i\eta\lambda}(\eta\lambda)^{-1/6}I_{\pm}(\partial^{k}\varrho^{n}(.,\eta,\lambda))_{\eta}\Big(\frac{\partial_{\tau}\theta_{0}(y,t,\eta,-\eta(1+a)^{1/2})}{2(1+a)^{1/2}a^{1/2}}-2n,\lambda\Big).
\]
Using again the fact that $I_{\pm}(\varrho^{n})_{\eta}$ is a convolution product we obtain $I_{\pm}(\partial^{k}\varrho^{n})_{\eta}=\partial^{k}(I_{\pm}(\varrho^{n})_{\eta})$, which allows to achieve the proof of \eqref{egaltoprove}.
\end{proof}

Let $u^{n}_{h}$ be defined by \eqref{uhndef}. We show that it is an approximate solution to \eqref{undered3} and we obtain bounds for the $L^{2}$ norm of $\square_{g}u^{n}_{h}$. Applying the wave operator $\square_{g}$ to $u^{n}_{h}$ and using the eikonal equations \eqref{sistem1} yields
\begin{multline}
\square_{g} u^{n}_{h}(x,y,t)=
\int e^{\frac{i}{h}\Phi^{n}(x,y,t,\xi,\eta,-\eta(1+a)^{1/2})}\times \\
\times \Big(\frac{i}{h}(<2d\Phi^{n},dg^{n}_{h}>+\eta^{-1/3}<d\zeta,d\zeta>\partial_{\xi}g^{n}_{h}+(\square_{g}\Phi^{n}) g^{n}_{h})+\square g^{n}_{h}\Big)d\xi d\eta .
\end{multline}
Using Lemma \ref{lemintcurves} we obtain the following:
\begin{prop}\label{propnorml2square}
\begin{equation}\label{normesqarel2}
\|\square_{g} u^{n}_{h}(.,t)\|_{L^{2}(\Omega)}=O(h^{-1})\|u^{n}_{h}(.,t)\|_{L^{2}(\Omega)},
\end{equation}
uniformly for $0\leq n\leq N\lesssim \lambda h^{\epsilon}$.
\end{prop}
\begin{rmq}
This result is useful since in order to estimate the error between the approximate solution we are constructing and the exact solution to \eqref{undered3} we are going to use the same approach as in Lemma \ref{lemestimerr} with error $O(h^{-1})\|u^{n}_{h}(.,t)\|_{L^{2}(\Omega)}$. 
\end{rmq}
\begin{proof}
Using the eikonal equations \eqref{sistem1} and integrating by parts with respect to $\xi$ yields
\begin{multline}\label{estl2normsqg}
\square_{g} u^{n}_{h}(x,y,t)=
\int e^{\frac{i}{h}\Phi^{n}(x,y,t,\xi,\eta,-\eta(1+a)^{1/2})}
\Big(\frac{i}{h}((\square_{g}\Phi^{n})g^{n}_{h}+2(1+xb(y))\partial_{y}\Phi_{h}\partial_{y}g^{n}_{h})\\
+\frac{1}{4(1+a)a}(<d\partial_{\tau}\Phi,d\partial_{\tau}\Phi>-\partial_{\tau}\zeta\partial_{\tau}<d\zeta,d\zeta>)\partial^{2}_{z}g^{n}_{h}
+\frac{1}{2(1+a)^{1/2}a^{1/2}}\square_{g}(\partial_{\tau}\Phi)\partial_{z}g^{n}_{h}\\
-(1+xb(y))\partial^{2}_{y}g^{n}_{h}\Big) \Big(\frac{\partial_{\tau}\theta+\eta^{1/3}\xi\partial_{\tau}\zeta}{2(1+a)^{1/2}a^{1/2}}(x,y,t,\eta,-\eta(1+a)^{1/2})-2n,y,\eta,h\Big)d\xi d\eta.
\end{multline}
Here $\partial_{z}g^{n}_{h}$ denotes the derivative of $g^{n}_{h}(z,y,\eta)$ with respect to the first variable.

We estimate the $L^{2}(\Omega)$ norm of the terms in the integral above: since $\square_{g} \Phi^{n}$ is bounded in $(x,y)$, uniformly in $0\leq n\leq N\lesssim \lambda h^{\epsilon}$, independent of $t$ then 
\[
\|\int_{s,\eta}e^{\frac{i}{h}\Phi^{n}(x,y,t,\eta^{1/3}s,\eta,-\eta(1+a)^{1/2})}
(\square_{g}\Phi^{n})g^{n}_{h}\|_{L^{2}(\Omega)}\lesssim \|\square_{g}\Phi^{n}\|_{L^{\infty}(\Omega)}\|u^{n}_{h}\|_{L^{2}(\Omega)}.
\]
The remaining terms have coefficients $\lesssim a^{-1}$ and their symbols are uniformly bounded functions in $(x,y)$ as well. In order to estimate the $L^{2}(\Omega)$ norms of these terms involving derivatives of $g^{n}_{h}$ with respect to its first variable we use the fact that $g^{n}_{h}(.,y,\eta)\in\mathcal{S}_{K_{0}}(\lambda/n)$ so that for every $k\geq 0$ there exist constants $C_{k}>0$ independent of $h$ such that $\sup_{z}|\partial^{k}_{z}g^{n}_{h}(z,\eta,\lambda)|\leq C_{k}$ and the essential support of $\partial^{k}_{z}g^{n}_{h}$ is $K_{0}=[-c_{0},c_{0}]$. To handle the terms involving derivatives of $g^{n}_{h}$ with respect to $y$ we use the explicit formula \eqref{sigmndef} and the fact that $\mu_{k}$ are symbols of order $0$ and type $(1,0)$.
\end{proof}

\subsection{Main properties of the parametrix}
In this section we state the main properties of the parametrix
\begin{equation}\label{uhdef}
U_{h}(x,y,t):=\sum_{n=0}^{N}u^{n}_{h}(x,y,t),
\end{equation}
where $u^{n}_{h}(x,y,t)$ are introduced in \eqref{uhndef} and where $N\lesssim \lambda h^{\epsilon}$, for some $\epsilon>0$. Here $\epsilon$ will be the one fixed at the beginning of Section \ref{secdoide} (see Remark \ref{defepsilon}). We first prove that each $u^{n}_{h}$ is essentially supported for $t$ in an interval of time of size $\simeq a^{1/2}$ and that $(u^n_h)_n$ have almost disjoint supports in time and in the tangential variable $y$. 

In Section \eqref{sectNt} we choose $N$ so that each cusp type solution $u^n_h$ preserves, for $0\leq n\leq N$, the same properties as the first one, $u^0_h$. We prove that this condition requires that
\begin{equation}\label{condaN}
4Na^{1/2}\lesssim Y.
\end{equation}
Taking into account the condition $N\lesssim \lambda h^{\epsilon}$, which must hold so  that the construction of the Fourier integral operators in Section does not degenerate, and since we require the parameter $a$ to be as small as possible (since the loss of derivatives of the norm $\|U_h\|_{L^q L^r}$ increases when $a$ gets smaller, as shown  in the Appendix), we must have 
\begin{equation}
 N\simeq \lambda h^{\epsilon},\ \text{where}\  \lambda=\frac{a^{3/2}}{h}.
\end{equation}
The last conditions yield $a\simeq  \frac 12 Y^{1/2} h^{\frac{1-\epsilon}{2}}$. 
\begin{rmq}
Recall that the assumption that the number $N$ should be "much" smaller than the parameter $\lambda$ was necessary in the construction of the model solution (see Section \ref{sectmodel}). Precisely, the construction of  the operator $(J_{+,\eta}\circ I_{-,\eta})^{\circ n}$  in the proof of Proposition \ref{propimportant1} involves a stationary phase argument with (large) parameter $\lambda/n$; this parameter has to be a strictly positive power of $h^{-1}$,  otherwise the construction fails since the contribution coming from the error terms doesn't belong to $O_{\mathcal{S}(\mathbb{R})}(h^{\infty})$; this is possible only taking $n\leq N\lesssim \lambda h^{\epsilon}$. 
\end{rmq}

\subsubsection{Wave front set}
\begin{prop}\label{lemlagrawfset}
Let $u^{n}_{h}$ be given in \eqref{uhndef}. Then the wave front set $WF_{h}(u^{n}_{h})$ of $u^{n}_{h}$ is contained in the Lagrangian set $\Lambda_{\Phi^{n}}$ defined by\begin{multline}\label{lagranh}
\Lambda_n:=\Big\{(x,y,t,\xi,\eta,-\eta(1+a)^{1/2})| \zeta(x,y,\eta,-\eta(1+a)^{1/2})+\eta^{2/3}\xi^{2}=0,\\(\partial_{\eta}\theta-(1+a)^{1/2}\partial_{\tau}\theta+\xi\zeta)(x,y,t,1,-(1+a)^{1/2})+\frac{\xi^{3}}{3}+\frac{4}{3}na^{3/2}=0\Big\}.
\end{multline}
In other words, outside any neighborhood of $\Lambda_n$ the contribution of $u^{n}_{h}$ is $O_{L^{2}}(h^{\infty})$.
\end{prop}

\begin{proof}
Let $(x,y)$ be such that $|\partial_{\xi}\Phi^{n}|\geq c$ for some $c>0$ and let $L_{1}$ be the operator $L_{1}:=\frac{h}{i}\frac{1}{\xi^{2}+\eta^{-2/3}\zeta}\partial_{\xi}$. After $m$ integrations by parts with respect to $\xi$ we gain a factor $h^{(1-\alpha/2)m}$ and the contribution of $u^{n}_{h}$ in this case is $O_{L^{2}}(h^{\infty})$.

Let now $|\partial_{\eta}\Phi^{n}|\geq c>0$ for some positive constant $c$; we perform (repeated) integrations by parts using this time the operator $L_{2}:=\frac{h}{i}\frac{\partial_{\eta}\Phi^{n}}{|\partial_{\eta}\Phi^{n}|^{2}}\partial_{\eta}$. We need however to estimate the derivatives with respect to $\eta$ of $g^{n}_{h}$, precisely we have to estimate
\[
L^{*m}_{2}\Big(g^{n}_{h}(\frac{\partial_{\tau}\theta+\eta^{1/3}\xi\partial_{\tau}\zeta}{2(1+a)^{1/2}a^{1/2}}(x,y,t,\eta,-\eta(1+a)^{1/2})-2n,y,\eta\Big).
\]
Since $\partial_{\tau}\theta$ and $\eta^{1/3}\partial_{\tau}\zeta$ are homogeneous of degree $0$ with respect to $(\eta,\tau)$, the first variable in the symbol $g^{n}_{h}$ is independent of $\eta$.

The symbol $g^{n}_{h}(z,y,\eta)$ is an asymptotic sum whose general term is of the form 
\[
h^{k/2}a^{-k/2}\Psi(\eta)\eta^{1/3}\mu_{k}(y,\eta,h)\partial^{k}\varrho^{n}(z,\eta,\lambda),
\]
where we recall from Definition \ref{dfncuspn} and Proposition \ref{propimportant1} of Section \ref{sectmodel} that  $\partial^{k}\varrho^{n}$ writes as a convolution
\[
\partial^{k}\varrho^{n}(z,\eta,\lambda)=(F_{\eta\lambda})^{*n}*\partial^{k}\varrho^{0}(.,\lambda)(z),\quad \forall k\geq 0,
\]
where $\varrho^{0}\in\mathcal{S}_{K_{0}}(\lambda)$ is independent of $\eta$ and where $(F_{\eta\lambda})^{*n}$ is defined in \eqref{Fn}. 
Since $\mu_{k}$ are symbols of order $0$ and type $(1,0)$ it will be enough to estimate the contribution of the terms involving the derivates on $(F_{\eta\lambda})^{*n}$. We recall that $(F_{\eta\lambda})^{*n}$ has the explicit form
\[
(F_{\eta\lambda})^{*n}(z)=\frac{\eta\lambda}{2\pi}\int_{w}e^{i\eta\lambda(wz+n(2w+\frac{4}{3}((1-w)^{3/2}-1))} (c(w,\eta\lambda))^{n}dw,
\]
where 
 \begin{align}
 \nonumber
 c(w,\eta\lambda) & =\kappa(w)a_{+}(w,\eta\lambda)b_{-}(w,\eta\lambda)\\
 & \simeq\kappa^{2}(w)\Big(1+\sum_{j\geq 1}c_{j}(1-w)^{-3j/2}(\eta\lambda)^{-j}\Big).
 \end{align}
  Making the change of variables $\tilde{w}=nw$ and setting $\tilde{\lambda}=\lambda/n\geq h^{-\epsilon}\gg 1$ yields
 \[
(F_{\eta\lambda})^{*n}(z)=\frac{\eta\tilde{\lambda}}{2\pi}\int_{\tilde{w}}e^{i\eta\tilde{\lambda}(z\tilde{w}+ n^{2}f(\frac{\tilde{w}}{n}))}c^{n}(\frac{\tilde{w}}{\lambda},\eta n\tilde{\lambda})d\tilde{w},
\]
where $f(w):=2w+\frac{4}{3}((1-w)^{3/2}-1)$. Hence one $\eta$-derivative yields
\begin{multline}\label{derivfetlan}
\partial_{\eta}(F_{\eta\lambda})^{*n}(z)=\frac{1}{\eta}(F_{\eta\lambda})^{*n}(z)+\frac{\eta\tilde{\lambda}}{2\pi}\int_{\tilde{w}}e^{i\eta\tilde{\lambda}(z\tilde{w}+n^{2}f(\frac{\tilde{w}}{n}))}i\tilde{\lambda}(z\tilde{w}+ n^{2}f(\frac{\tilde{w}}{n}))c^{n}(\frac{\tilde{w}}{n},\eta n\tilde{\lambda})d\tilde{w}\\
+ \frac{\eta\tilde{\lambda}}{2\pi}\int_{\tilde{w}}e^{i\eta\tilde{\lambda}(z\tilde{w}+ n^{2}f(\frac{\tilde{w}}{n}))}n\partial_{\eta}c(\frac{\tilde{w}}{n},\eta n \tilde{\lambda})c^{n-1}(\frac{\tilde{w}}{n},\eta n \tilde{\lambda})d\tilde{w}.
\end{multline}
The symbol of the third term in the right hand side of \eqref{derivfetlan} is $n\partial_{\eta}c(w,\eta\lambda)c^{n-1}(w,\eta\lambda)$ and we have
\[
\partial_{\eta}c(w,\eta\lambda)=-\eta^{-2}\lambda^{-1}\sum_{j\geq 1}jc_{j}(1-w)^{-3j/2}(\eta\lambda)^{-(j-1)},
\]
and since $n\ll \lambda$, the contribution from this term is easily handled with.

The symbol in the second term in the right hand side of \eqref{derivfetlan} equals the symbol of $(F_{\eta\lambda})^{*n}$ multiplied by the factor $i\tilde{\lambda}(z\tilde{w}+\lambda n f(\frac{\tilde{w}}{n}))$. Recall that on the support of $c(w,\eta\lambda)$ we have $w=\tilde{w}/n\in \text{supp}(\kappa)$ and that the support of $\kappa$ can be taken as close as we want from $0$. Moreover, on the support of $\kappa(w)$ we have $f(w)=w^{2}/2+O(w^{3})$, hence
$n^{2}f(\frac{\tilde{w}}{n})=\tilde{w}^{2}/2+O(\tilde{w}^{3}/n)$.
On the other hand, when we take the convolution product of the second term in \eqref{derivfetlan} with $\varrho^{0}(.,\lambda)$ we find that the critical points of the phase in the oscillatory integral obtained in this way, given by
\[
\frac{\eta\tilde{\lambda}}{2\pi}\int_{\tilde{w},z'}e^{i\eta\tilde{\lambda}((z-z')\tilde{w}+n^{2}f(\frac{\tilde{w}}{n}))}i\tilde{\lambda}((z-z')\tilde{w}+n^{2}f(\frac{\tilde{w}}{n}))c^{n}(\frac{\tilde{w}}{n},\eta n\tilde{\lambda})\varrho^{0}(z',\lambda)d\tilde{w}dz',
\]
must satisfy $\tilde{w}=0$ and $z=z'$. The phase function which we denoted by $\phi_{n}(z,z',\tilde{w})$ satisfies $\phi_{n}(z,z,0)=0$. Applying the stationary phase theorem in $\tilde{w}$ and $z'$, the first term in the asymptotic expansion obtained in this way vanishes, and the next ones are multiplied by strictly negative, integer powers of $\tilde{\lambda}$, hence the contribution from this term will is also bounded. 

Notice that when we take higher order derivatives in $\eta$ of $\varrho^{n}$, we obtain symbols which are products of $\tilde{\lambda}^{j} (\phi_{n})^{j}\partial^{k-j}_{\eta}(c^{n}(\tilde{w}/n,\eta n \tilde{\lambda}))$ and can be dealt with in the same way, taking into account this time that the first $j$ terms in the asymptotic expansion obtained after applying the stationary phase vanish. As a consequence, after each integration by parts in $\eta$ using the operator $L_{2}$ we gain a factor $h$,  meaning that the contribution of $u^{n}_{h}$ is $O_{L^{2}}(h^{\infty})$.
\end{proof}

\subsubsection{Time $T$ and number of reflections $N$}\label{sectNt}
In what follows we choose the number $N$ of iterated cusp type solutions $u^n_h$.\begin{prop} \label{propnnn}
Let $1\leq N\leq C_0 Y a^{-1/2}$, for some fixed $C_0>0$. If $C_0$ is chosen sufficiently small, then the operator $J$ introduced in Section \ref{sectbdcond} is elliptic near the set 

$\Big(\cup_{0\leq n\leq N}I_n(c_0)\Big)\times \Big(\cup_{0\leq n\leq N}I_n(c_0)\Big)\times (1,-(1+a)^{1/2})$ and $\chi_{\partial}$ is a diffeomorphisme from a neighborhood of this set onto its image. We recall that $I_n(c_0)$ was introduced in \eqref{inco} and $0<c_0\leq \frac 38$ was fixed in Section \ref{secmodelconstruction}. 
\end{prop}
\begin{proof}
We use the properties of $\chi_{\partial}$ defined in Section \ref{secphases} together with the assumption on $b(y)$ in Theorem \ref{thms3}. Recall from \eqref{chipartmodel} that we have
\begin{equation}
\chi_{\partial}(d_{\eta}\theta_{0},d_{\tau}\theta_{0},\eta,\tau)=(y,t,d_{y}\theta_{0},d_{t}\theta_{0}).
\end{equation}
We can explicitely compute the gradient $\nabla_{y,t}(\chi^{-1}_{\partial})$ as follows
\begin{equation}\label{matrixderivchi}
\left(
\begin{array}{cc}
   \partial^{2}_{y,\eta}\theta_0(y,t,\eta,\tau)  & 0   \\
  \partial^{2}_{y,\tau}\theta_0(y,t,\eta,\tau)   &  1 
\end{array}
\right).
\end{equation}
Using \eqref{eqtheto}, \eqref{eqzet1tau} we have
\begin{equation}\label{eqderivthet}
(\partial_{y}\theta_0)^2(y,t,\eta,\tau) = \tau^2+\zeta_0(\eta,\tau)b^{2/3}(y)\tau^{4/3}+O(\zeta_0^2),
\end{equation}
therefore, taking the derivative with respect to $\eta$ in \eqref{eqderivthet} we find
\begin{equation}
2\partial_y\theta_0\partial^2_{\eta,y}\theta_0(y,t,\eta,\tau)=2\eta b^{2/3}(y)\Big(\frac{\tau}{\eta}\Big)^{4/3}+O(\zeta_0).
\end{equation}
This finally implies that 
\begin{equation}\label{thetety}
\partial^{2}_{y,\eta}\theta_0(y,t,\eta,\tau)=b^{2/3}(y)+O(\zeta_0(\eta,\tau)).
\end{equation}
Taking now the derivative with respect to $\tau$ in \eqref{eqderivthet} yields
\begin{equation}\label{thettay}
\partial^{2}_{y,\tau}\theta_0(y,t,\eta,\tau)=\Big(1-\Big(\frac{\tau}{\eta}\Big)^{4/3}b^{2/3}(y)\Big)+O(\zeta_0(\eta,\tau)).
\end{equation}
Notice also that the right hand side in both \eqref{thetety} and \eqref{thettay} is independent of $t$ (which follows from the linearity in time of $\theta$). Since $|b^{1/3}(y)-1|\leq\frac{1}{10}$ for $y\in [0,Y]$, and if $\eta/\tau$ is close to $1$ it follows that $\chi_{\partial}$ is a diffeomorphism from a small, fixed, conic neighborhood of $\pi(\bar{\varrho}_0,\bar{\vartheta}_0)$ into a small neighborhood of  $\pi(\varrho_0,\vartheta_0)$. 

\begin{rmq} 
In particular, if we consider the restriction of $\chi^{-1}_{\partial}$ to $(\eta,\tau)=(1,-(1+a)^{1/2})$, its Jacobian is given by $\text{Jac} (\chi^{-1}_{\partial})=\partial^{2}_{y,\eta}\theta_0(y,t,1,-(1+a)^{1/2})$ and is independent of $t$.
\end{rmq}

Recall that $\chi_{\partial}$ conjugates the billiard ball map $\delta^{\pm}$ to the normal form, $\chi_{\partial}\circ \delta^{\pm}_F=\delta^{\pm}\circ\chi_{\partial}$. Since we consider only positive time it is enough to work with $\delta^+$, $\delta^+_F$. For $n\geq 1$ and for $\pi(\bar{\varrho},\bar{\vartheta})$ in a small, conic neighborhood of $\pi(\bar{\varrho}_0,\bar{\vartheta}_0)$ we have, writting $\pi(\varrho,\vartheta)=\chi_{\partial}(\pi(\bar{\varrho},\bar{\vartheta}))$ (which belongs to a small neighborhood of $\pi(\varrho_0,\vartheta_0)$),
\begin{align}\label{dnchidnf}
\nonumber
(\delta^+)^n(\pi(\varrho,\vartheta))&= (\delta^+)^n(\chi_{\partial}(\pi(\bar{\varrho},\bar{\vartheta})))\\
&=\chi_{\partial}\Big((\delta^+_F)^n(\pi(\bar{\varrho},\bar{\vartheta}))\Big),
\end{align}
where $(\delta^+_F)^n$ is given by the formula \eqref{deltait} that we recall here,
\[
(\delta^+_F)^n(y,t,\eta,\tau)=\Big(y\pm4n(\frac{\tau^{2}}{\eta^{2}}-1)^{1/2}\pm\frac{8}{3}n(\frac{\tau^{2}}{\eta^{2}}-1)^{3/2},t\mp 4n(\frac{\tau^{2}}{\eta^{2}}-1)^{1/2}\frac{\tau}{\eta},\eta,\tau\Big).
\]
We have assumed (without loss of generality) that $\pi(\varrho_0,\vartheta_0)=(0,0,1,-1)$. Modulo a translation we can also assume that $\pi(\bar{\varrho}_0,\bar{\vartheta}_0)=(0,0,1,-1)$. We now take $\bar{\varrho}=\bar{\varrho_0}$ and $\bar{\vartheta}=(1,-(1+a)^{1/2})$ in \eqref{dnchidnf}. Therefore, for $a$ small, depending on $h$, $\pi(\bar{\varrho},\bar{\vartheta})$ belongs to a conic neighborhood of $\pi(\bar{\varrho}_0,\bar{\vartheta}_0)$, (letting $h$ go to $0$, this neighborhood can be made as small as we want). Rewriting \eqref{dnchidnf} at this point and using \eqref{deltait} yields
\begin{align}
\nonumber
(\delta^+)^n(\pi(\varrho,\vartheta))&=\chi_{\partial}\Big((\delta^+_F)^n(\pi(\bar{\varrho},\bar{\vartheta}))\Big),\\
\nonumber
&=\chi_{\partial}\Big(4na^{1/2}(1+a)-\frac 43 na^{3/2},4na^{1/2}(1+a)^{1/2},1,-(1+a)^{1/2}\Big)\\
&=:\chi_{\partial}(\bar{y}_n,\bar{t}_n,1,-(1+a)^{1/2}),
\end{align}
where we set
\[
\bar{y}_n=4na^{1/2}(1+a)-\frac 43 na^{3/2},
\]
\[
\bar{t}_n=4na^{1/2}(1+a)^{1/2}.
\]
If $na^{1/2}$ belongs to a small but fixed neighborhood of $0$ of size $Y$, we can write the right hand side term in the last equation as follows
\begin{equation}\label{nablachi}
\chi_{\partial}(\bar{y}_n,\bar{t}_n,1,-(1+a)^{1/2})=\nabla_{y,t}\chi_{\partial}(\bar{y},\bar{t},1,-(1+a)^{1/2})
\left(
\begin{array}{c}
  \bar{y}_n \\
  \bar{t}_n
\end{array}
\right),
\end{equation}
for some $(\bar{y},\bar{t})\in [0,\bar{y}_n]\times [0,\bar{t}_n]$. The matrix $\nabla_{y,t}\chi_{\partial}(\bar{y},\bar{t},1,-(1+a)^{1/2})$ is independent of $\bar{t}$ (since it is given by the inverse of $\nabla_{y,t}(\chi_{\partial}^{-1})$ computed in \eqref{matrixderivchi}) and at $\bar{y}=0$ is equal to
\[
\left(
\begin{array}{cc}
 1+O(a) &   0   \\
 O(a)  & 1  
\end{array}
\right)^{-1}
\]
(where we used \eqref{matrixderivchi}, \eqref{thetety} and \eqref{thettay}), therefore if $\bar{y}_n$ belongs to a small, fixed neighborhood of $0$ of size $y_0$, the matrix $\nabla_{y,t}\chi_{\partial}(\bar{y},\bar{t},1,-(1+a)^{1/2})$ remains close to the identity.

We can now estimate the number of iterations $N$ that we will use in our construction. 
Since $\bar{y}_n=4na^{1/2}+O(na^{3/2})$, choose $N\geq 1$ such that $Na^{1/2}$ to be sufficiently small so that for $0\leq \bar{y}\leq \bar{y}_N$ the gradient matrix $\nabla_{y,t}(\chi_{\partial})(\bar{y},.,1,-(1+a)^{1/2})$ to remain close to the identity. This is possible due to the independence on $\bar{t}_N, \bar{t}$ of the gradient matrix and its uniformly boundedness. Since from the initial assumption we have $|b^{2/3}(y)-1|< \frac 14$ for $y\in [0,Y]$ we can take
\begin{equation}\label{condnnn}
4N\simeq C_0 Y a^{-1/2},
\end{equation}
for some sufficiently small constant $0<C_0\leq 1$.
\begin{rmq}
The operator $J$ is elliptic in a small, fixed neighborhood of $0$. Taking $Y$  smaller if necessary, the symbol $d$ of $J$ will be elliptic for $y\in [0,Y]$, therefore the symbol of $J\circ Tr_{\pm}(u^n_{F,h})$ will remain elliptic for any $0\leq n\leq N$ with $N$ given by \eqref{condnnn}  if $C_0$ is small enough. This will be useful when computing the $L^r$ norms of the cusps $u^n_h$, whose the symbols will depend on $d$ and therefore will be elliptic uniformly in $0\leq n\leq N$.
\end{rmq}
\begin{rmq}
In the following we set 
\begin{equation}\label{defa}
a:=\frac{\sqrt{C_0}}{2}Y^{1/2}h^{\frac{1-\epsilon}{2}}.
\end{equation}
Indeed, from the condition $N\lesssim \lambda h^{\epsilon}=a^{3/2}h^{-(1-\epsilon)}$ which must hold so that the construction of the operators $(J_{+,\eta}\circ I_{-,\eta})^{\circ N}$ not to degenerate and from Proposition \ref{propnorm} in the Appendix which states that the lost of derivatives of $L^{q}L^r$ norms of $U_h$ (compared to the optimal Strichartz)  grows with $N$, the worst loss seems to appear for $N\simeq \lambda h^{\epsilon}$ (and $\epsilon>0$ small). Using now \eqref{condnnn} we get the estimate \eqref{defa} for the small parameter $a$.
\end{rmq}
\end{proof}

\paragraph{Time interval $[0,T]$:}
In the remaining part of this section we estimate the interval of time $[0,T]$ on which the norm of $U_h$ will be evaluated. Using Proposition \ref{propsuppunhx} below and $N$ like in \eqref{condnnn} we see that for $0\leq n\leq N$, the cusp $u^n_h(.,t)$ will be essentially supported for $(x,y,t)$ such that 
\begin{equation}
\partial_{\tau}\theta(x,y,t,1,-(1+a)^{1/2})\in  I_n(c_0),
\end{equation}
where
\begin{equation}
I_n(c_0)=2a^{1/2}(1+a)^{1/2}\times [2n-(1+c_{0}),2n+(1+c_{0})].
\end{equation}
We shall choose $T$ to belong to the essential support of $u^N_h$. 
From Proposition \ref{propsuppunhx} in Section \ref{sectlocasupp} below it follows that on the essential support of $u^N_h$ the following should hold:
\begin{multline}
\Big(\partial_{\eta}\theta-(1+a)^{1/2}\partial_{\tau}\theta\Big)(x,y,t,1,-(1+a)^{1/2})=-\frac 43 Na^{3/2}\\\pm \frac 23(-\zeta)^{3/2}(x,y,1,-(1+a)^{1/2}).
\end{multline}
We introduce the defining function for the caustic set, denoted $C(y,\eta,\tau)$, as follows
\[
-\zeta(x,y,\eta,\tau)=0\quad \text{if and only if}\quad x=C(y,\eta,\tau).
\]
Indeed, since $-\zeta(x,y,1,-(1+a)^{1/2})=a-x\zeta_1(y,1,-(1+a)^{1/2})+O(a^2)$ and using \eqref{eqzet1tau} together with $|b^{1/3}(y)-1|\leq1/10$ for $y\in [0,Y]$, we are allowed to  apply the implicit function theorem to determine the smooth curve $C(y,1,-(1+a)^{1/2})$ for $y\in [0,Y]$.  

Using the special form of the phase function $\theta(x,y,t,\eta,\tau)$ given in \eqref{specialtheta} we have $\partial_t\theta=\tau$, therefore its derivative with respect to the $\eta$ variable is independent of $t$. Hence the application
\begin{equation}\label{applicY}
y\rightarrow \partial_{\eta}\theta(C(y,1,-(1+a)^{1/2}),y,.,1,-(1+a)^{1/2})
\end{equation}
is independent of time. Moreover, \eqref{applicY} is a diffeomorphisme in a neighborhood of $y=0$ since its derivative doesn't vanish. Indeed
\[
\partial^{2}_{y,\eta}\theta_0(y,.,1,-(1+a)^{1/2})=b^{2/3}(y)+O(a)\in [(1-\frac{1}{10})^{2},(1+\frac{1}{10})^{2}]
\]
stays away from a neighborhood of $0$ and $C(y,1,-(1+a)^{1/2})=a b^{-1/3}(y)+O(a^2)$. 

Let $\bar{t}_N:=4Na^{1/2}(1+a)^{1/2}$ be the center of the interval $I_N(c_0)$. From the discussion above (and if $C_0$ in \eqref{condnnn} is small) it follows that there exists an unique $y_N\in [0,Y]$ such that 
\begin{equation}\label{choiceYNT}
\partial_{\eta}\theta(C(y_N,1,-(1+a)^{1/2}),y_N,.,1,-(1+a)^{1/2})=(1+a)^{1/2}\bar{t}_N-\frac 43Na^{3/2}=\bar{y}_N.
\end{equation}
We can now define $T$ as the unique time value which satisfies:
\begin{equation}\label{choiceT}
\partial_{\tau}\theta(C(y_N,1,-(1+a)^{1/2}),y_N,T,1,-(1+a)^{1/2})=\bar{t}_N.
\end{equation}
\begin{rmq}\label{rmqchoiceT}
The motivation of the choice of $T$ will be given in the end of this section and will be particularly useful to show that the restriction of the parametrix $U_h$ to $\partial\Omega\times [0,T]$ is $O(h^{\infty})$ (see Proposition \ref{propdiruh01}).
\end{rmq}

\subsubsection{Localization of the cusps}\label{sectlocasupp}
In this section we describe the essential supports of $u^{n}_{h}$. Recall from Proposition \ref{propsuppunfh} that in the Friedlander's case the solutions $u^{n}_{F,h}$ to \eqref{undef} were essentially supported in time and the tangential variable in the intervals, respectively
\begin{equation}\label{inttimecasmod}
2a^{1/2}\times [2n-(1+c_{0}),2n+(1+c_{0})],
\end{equation}
where $c_{0}$ is a constant sufficiently small and such that $\varrho\in\mathcal{S}_{K_{0}}(\lambda)$, with $K_{0}=[-c_{0},c_{0}]$.

We prove a similar property for $u^{n}_{h}$: precisely, their essential supports in the time variables will be contained in almost disjoint intervals obtained by taking the image of \eqref{inttimecasmod} by the symplectomorphisme $\chi_{\partial}$.
For $0\leq n\leq N$, we define
\begin{equation}\label{definc}
I_{n}(c): =[2n-(1+c), 2n+(1+c)]\times 2(1+a)^{1/2}a^{1/2}.
\end{equation}
We prove the following:
\begin{prop}\label{propsuppunhx}
We have
\begin{multline}\label{inclussuppunh}
\text{ess-supp}(u^{n}_{h})\subset \Big\{(x,y,t)| \partial_{\tau}\theta(x,y,t,1,-(1+a)^{1/2})\in I_{n}(c_{0})\
\text{and}\\
\Big(\partial_{\eta}\theta-(1+a)^{1/2}\partial_{\tau}\theta+\frac{4}{3}na^{3/2}\mp\frac{2}{3}(-\zeta)^{3/2}\Big)(x,y,t,1,-(1+a)^{1/2})=0\Big\}.
\end{multline}
By $\text{ess-supp}(u^{n}_{h})$ we denote the closure of the set of points outside of which $u^{n}_{h}$ is $O_{L^{2}}(h^{\infty})$.
In particular, $(u^{n}_{h})_{n\in\{0,..,N\}}$ have almost disjoint supports in time and the tangential variable.
\end{prop}
\begin{proof}
Let $\eta\in\text{supp}(\Psi)$. Since according to Lemma \ref{lemjfhn} we have $g^{n}_{h}(.,y,\eta)\in \mathcal{S}_{[-c_{0},c_{0}]}(\lambda/n)$, on the essential support (in the first variable) of the symbol $g^{n}_{h}$ the following holds 
\begin{equation}\label{ineg1}
\Big|\frac{(\partial_{\tau}\theta+\eta^{1/3}\xi\partial_{\tau}\zeta)(x,y,t,\eta,-\eta(1+a)^{1/2})}{2(1+a)^{1/2}a^{1/2}}-2n\Big|\\ \leq c_0.
\end{equation}
Let $c\in (0,1)$ be such that 
\begin{equation}\label{timea}
\Big|\frac{\partial_{\tau}\theta(x,y,t,\eta,-\eta(1+a)^{1/2})}{2(1+a)^{1/2}a^{1/2}}-2n\Big|\geq (1+c).
\end{equation}
We show that on the essential support of $u^{n}_{h}$ we must have $c\leq c_{0}$, therefore the contribution of $u^{n}_{h}$ will be $O_{L^{2}}(h^{\infty})$ for $(x,y,t)$ outside a set on which
\begin{equation}\label{unsupport}
\Big|\frac{\partial_{\tau}\theta(x,y,t,\eta,-\eta(1+a)^{1/2})}{2(1+a)^{1/2}a^{1/2}}-2n\Big|\leq (1+c_{0}).
\end{equation}
The inequalities \eqref{ineg1} and \eqref{timea} yield
\begin{align}\label{ineg2}
\nonumber
\Big|\frac{\eta^{1/3}\xi\partial_{\tau}\zeta(x,y,\eta,-\eta(1+a)^{1/2})}{2(1+a)^{1/2}a^{1/2}}\Big| & \geq \Big|\frac{\partial_{\tau}\theta(x,y,t,\eta,-\eta(1+a)^{1/2})}{2(1+a)^{1/2}a^{1/2}}-2n\Big|\\ 
\nonumber
& -\Big|\frac{(\partial_{\tau}\theta+\eta^{1/3}\xi\partial_{\tau}\zeta)(x,y,t,\eta,-\eta(1+a)^{1/2})}{2(1+a)^{1/2}a^{1/2}}-2n\Big|\\ & 
\geq (1+c-c_{0}).
\end{align}
On the other hand, on the Lagrangian $\Lambda_n$ introduced in \eqref{lagranh} we have
\begin{equation}\label{condlagra}
|\eta^{1/3}\xi|=(-\zeta)^{1/2}(x,y,\eta,-\eta(1+a)^{1/2}),\quad \zeta\leq 0.
\end{equation}
Since, according to Proposition \ref{lemlagrawfset}, outside any neighborhood of the Lagrangian set $\Lambda_n$ the contribution of $u^{n}_{h}$ equals $O_{L^{2}}(h^{\infty})$, it follows from \eqref{condlagra} and \eqref{ineg2} that if $(x,y,t)\in \text{ess-supp}(u^{n}_{h})$ and $c$ is such that \eqref{timea} holds, then
\begin{equation}\label{inegcco}
(1+c-c_{0})^{2}\leq 1\quad \text{which yields}  \quad c\leq c_{0}.
\end{equation}
Indeed,  to deduce \eqref{inegcco} we used $\partial_{\tau}\zeta=\partial_{\tau}\zeta_0+O(a)$ and
\[
0\leq -\zeta(x,y,\eta,-\eta(1+a)^{1/2})\leq-\zeta_0(x,y,\eta,-\eta(1+a)^{1/2})=a\eta^{2/3}.
\]
Therefore we proved that if $(x,y,t)\in \text{ess-supp}(u^{n}_{h})$, then \eqref{unsupport} holds, which yields, for $\eta\in \text{supp}(\Psi)$,
\begin{equation}\label{apartthett}
\partial_{\tau}\theta(x,y,t,\eta,-\eta(1+a)^{1/2}) \in I_{n}(c_{0}).
\end{equation}
Since $\partial_{\eta}\Phi_n$ is homogeneous of degree $0$ in $(\eta,\tau)$, we can take $(1,-(1+a)^{1/2})$ instead of $(\eta,-\eta(1+a)^{1/2})$ in \eqref{apartthett}. 
If $(y,t)\in \text{ess-supp}(u^n_h)\cap\text{ess-supp}(u^k_h)$ then we must have
\[
\partial_{\tau}\theta_0 (y,t,1,-(1+a)^{1/2})\in I_n(c_0)\cap I_k(c_0).
\]
The intersection is non-empty unless $|n-k|\leq1$.
\end{proof}

\begin{lemma}\label{lemlengsup}
Let
\begin{equation}\label{defjncontin}
J_{n}:=\Big\{t| \exists (x,y,t)\in \text{ess-supp}(u^n_h) \ \text{s.t.}\ \partial_{\tau}\theta(x,y,t,1,-(1+a)^{1/2})\in I_{n}(c_{0}/3-1)\Big\}
\end{equation}
and let $|J_{n}|$ denote its size. Then 
\begin{equation}\label{lenghtjn}
|J_{n}|\geq c_{0}a^{1/2}.
\end{equation}
Moreover, if $c_{0}$ is sufficiently small and if $(x,y,t)$ is such that $t\in J_{n}$, then 
\begin{equation}\label{xnotbond}
\frac 12 a \leq x \leq\frac 32 a.
\end{equation}
\end{lemma}
\begin{proof}
We first proceed with \eqref{xnotbond}. Let $(x,y,t)\in \text{ess-supp}(u^{n}_{h})$ so that
\begin{equation}
\Big|\frac{\partial_{\tau}\theta(x,y,t,1,-(1+a)^{1/2})}{2(1+a)^{1/2}a^{1/2}}-2n\Big|\leq \frac{c_{0}}{3}.
\end{equation}
Using the inequality \eqref{ineg1} which holds on the support of the symbol $g^{n}_{h}$ together with the condition \eqref{condlagra} which assures the localization on the Lagrangian $\Lambda_{\Phi_{n}}$, we obtain
\begin{equation}
\Big|\frac{\partial_{\tau}\theta (x,y,t,1,-(1+a)^{1/2})}{2(1+a)^{1/2}a^{1/2}}-2n+\frac{((-\zeta)^{1/2}\partial_{\tau}\zeta) (x,y,t,1,-(1+a)^{1/2})}{2(1+a)^{1/2}a^{1/2}} \Big| \leq c_{0}.
\end{equation}
From the last two inequalities we obtain
\[
\Big|\frac{((-\zeta)^{1/2}\partial_{\tau}\zeta) (x,y,t,1,-(1+a)^{1/2})}{2(1+a)^{1/2}a^{1/2}} \Big| \leq \frac 43 c_0.
\]
We can estimate the left term in the last inequality, since using $-\zeta_0(x,y,1,-(1+a)^{1/2})=a$ yields
\[
-\zeta(x,y,1,-(1+a)^{1/2})=a-x\zeta_1(y,1,-(1+a)^{1/2})+O(a^2).
\]
On the other hand we also have 
\[
\partial_{\tau}\zeta(x,y,1,-(1+a)^{1/2})=2(1+a)^{1/2}+O(a),
\]
therefore, using \eqref{eqzet1tau} in the form $\zeta_1=b(y)\tau^2+O(a)$, we obtain
\[
xb^{1/3}(y)\geq a(1-(\frac 43 c_0)^2).
\]
For $y\in [0,Y]$ we have $|b^{1/3}(y)-1|\leq1/10$ and, if $c_0$ is chosen small enough, $0<c_0\leq \frac 38$,  we obtain \eqref{xnotbond}. The second inequality follows also from the estimate of $b^{1/3}(y)$.

Now we proceed with \eqref{lenghtjn}. We shall use the defining function for the caustic set $C(y,\eta,\tau)$ introduced at the end of the last section by
\[
-\zeta(x,y,\eta,\tau)=0\quad \text{if and only if}\quad x=C(y,\eta,\tau).
\]
Since $\partial_t\theta=\tau$, the derivative $\partial_{\eta}\theta(x,y,t,\eta,\tau)$ is independent of $t$. The application
\[
y\rightarrow \partial_{\eta}\theta(C(y,1,-(1+a)^{1/2}),y,.,1,-(1+a)^{1/2})
\]
is a diffeomorphisme in a neighborhood of $y=0$ since its derivative doesn't vanish.  Indeed, recall that
\[
\partial^{2}_{y,\eta}\theta_0(y,.,1,-(1+a)^{1/2})=b^{2/3}(y)+O(a)\in [(1-\frac{1}{10})^{2},(1+\frac{1}{10})^{2}]
\]
stays away from a neighborhood of $0$ and $C(y,1,-(1+a)^{1/2})=a b^{-1/3}(y)+O(a^2)$. This implies the existence of uniques points $y_{\pm}$ such that
\[
\partial_{\eta}\theta(C(y_{\pm},1,-(1+a)^{1/2}),y_{\pm},.,1,-(1+a)^{1/2})=(1+a)^{1/2}\bar{t}_{\pm}+\frac{4}{3}na^{3/2},
\]
where set
$\bar{t}_{\pm}:=(2n\pm \frac{c_{0}}{3})\times 2(1+a)^{1/2}a^{1/2}$ such that  $I_{n}(c_{0}/3-1)=[\bar{t}_{-},\bar{t}_{+}]$. Since $\partial^{2}_{t,\tau}\theta=1$, we can also determine uniques points $t_{\pm}$ such that 
\[
\partial_{\tau}\theta(C(y_{\pm},1,-(1+a)^{1/2}),y_{\pm},t_{\pm},1,-(1+a)^{1/2})=\bar{t}_{\pm}.
\]
From the last equality we deduce that $t_{\pm}$ belong to $J_{n}$, since from the choice of $y_{\pm}$ we have $\Big(C(y_{\pm},1,-(1+a)^{1/2}),y_{\pm},t_{\pm}\Big)\in \text{ess-supp}(u^{n}_{h})$. We will estimate from below the size of $J_{n}$ by $|t_{+}-t_{-}|$ and show that $a^{1/2}\lesssim |J_n|$ which will conclude the proof. We have
\begin{align}
\nonumber
\bar{t}_+-\bar{t}_- & =t_+-t_-+\int_{y_-}^{y_+}\partial^{2}_{y,\tau}\theta_0(y,1,-(1+a)^{1/2})dy +O(a)\\
& =t_+-t_-+\int_{y_-}^{y_+}(b^{2/3}(y)-1)dy +O(a),
\end{align}
and also
\begin{align}
\nonumber
\bar{t}_+-\bar{t}_- & =\int_{y_-}^{y_+}\partial^{2}_{y,\eta}\theta_0(y,1,-(1+a)^{1/2})dy+O(a)\\
& =\int_{y_-}^{y_+}b^{2/3}(y)dy +O(a).
\end{align}
Using the last two equalities together with the fact that $ 0.8\leq b^{2/3}(y)\leq 1.2$ for $y\in [0,Y]$ we obtain
\[
 t_+-t_-=y_+-y_-+O(a) \in \Big[\frac{1}{1.2},\frac{1}{0.8}\Big] (\bar{t}_+-\bar{t}_- ),
\]
which yields \eqref{lenghtjn}.
\end{proof}

\begin{lemma}\label{lemdisjsup}
Let $k\geq 0$ and $t\in J_k$, then 
the following holds
\begin{equation}
U_{h}(x,y,t)=u^{k}_{h}(x,y,t)+O_{L^{2}}(h^{\infty}).
\end{equation}
\end{lemma}
\begin{proof}
Let $(x,y,t)\in \text{ess-supp}(u^{n}_{h})$ for some $0\leq n\leq N$. We show that $n=k$. First notice that we must have 
\begin{equation}\label{estcaustic}
x\leq C(y,1,-(1+a)^{1/2})= a b^{-1/3}(y)+O(a^2),
\end{equation}
otherwise being localized outside a neighborhood of $\Lambda_n$. Suppose that $n\neq k$. We have to show that the contribution from $u^{n}_{h}$ is $O_{L^{2}}(h^{\infty})$. It is enough to prove that if $0<c_{0}\leq \frac 38$, then the inequality \eqref{ineg1} cannot hold. On $\text{ess-supp}(u^n_h)$ we have
\begin{align*}
c_{0}& \geq \Big|\frac{\partial_{\tau}\Phi^{n}(x,y,t,1,-(1+a)^{1/2}))}{2(1+a)^{1/2}a^{1/2}}-2n\Big|\\
& =\Big|\frac{\partial_{\tau}\Phi^{k}(x,y,t,1,-(1+a)^{1/2})}{2(1+a)^{1/2}a^{1/2} }-2k-2(n-k)\Big|\\ 
& \geq 2|n-k|-\Big|\frac{\partial_{\tau}\theta(x,y,t,1,-(1+a)^{1/2})}{2(1+a)^{1/2}a^{1/2} }-2k\Big|-\frac{|\xi|\partial_{\tau}\zeta(x,y,1,-(1+a)^{1/2})}{2(1+a)^{1/2}a^{1/2} }\\ & \geq 2-\frac{c_{0}}{3}-\frac{((-\zeta)^{1/2}\partial_{\tau}\zeta)(x,y,1,-(1+a)^{1/2})}{2(1+a)^{1/2}a^{1/2}},
\end{align*}
where we used that $t\in J_k$ to estimate the second term in the third line and that on the Lagrangian $\Lambda_n$ we have $|\xi|=(-\zeta)^{1/2}$. Using the asymptotic expansion of $\zeta$ and \eqref{eqzet1tau}, we obtain
\[
-\zeta(x,y,1,-(1+a)^{1/2})=a-xb^{1/3}(y)+O(a^2)\geq a(2-\frac{c_0}{3})^2,
\]
which cannot hold if $c_0$ is sufficiently small, $0<c_0\leq \frac 38$.
We must then have $n=k$.
\end{proof}
\begin{rmq}\label{rmqnorml2uh}
Lemma \ref{lemdisjsup} shows that $u^{n}_{h}$ have almost disjoint essential supports in the variables $(y,t)$. Therefore, Proposition \ref{propnorml2square} applies and shows that  $U_{h}$ defined in \eqref{uhdef} is also an approximate solution to \eqref{wavvv} in the sense that for $t\in [0,T]$ we have
\begin{equation}\label{squareguh}
\|\square_{g} U_{h}(.,t)\|_{L^{2}(\Omega)}\leq O(h^{-1})\|U_{h}(.,t)\|_{L^{2}(\Omega)}.
\end{equation}
In the rest of this section we prove that $U_{h}$ satisfies the Dirichlet boundary condition 
\[
U_{h}|_{\partial\Omega\times [0,T]}=O(h^{\infty}).
\]
\end{rmq}

\begin{prop}\label{propdiruh01}(Dirichlet boundary condition for $U_{h}$)
The approximate solution $U_{h}$ to \eqref{wavvv} defined in \eqref{uhdef} satisfies the Dirichlet boundary condition
\begin{equation}\label{diruh01}
U_{h}|_{\partial\Omega\times [0,T]}=O(h^{\infty}).
\end{equation}
\end{prop}
\begin{proof}
Recall from Proposition \ref{propboundjtracepm} that we have
\[
J(Tr_{-}(u^{n}_{F,h}))(y,t)+J(Tr_{+}(u^{n+1}_{F,h}))(y,t)=O(h^{\infty}),
\]
from which we deduce
\begin{equation}\label{diruhohinf}
U_{h}(0,y,t) =J(Tr_{+}(u^{0}_{F,h}))(y,t)+J(Tr_{-}(u^{N}_{F,h}))(y,t).
\end{equation}
We have to prove that the term in the last line is $O(h^{\infty})$ for $t\in [0,T]$. Notice that the first term in the right hand side of \eqref{diruhohinf} is essentially supported for $t$ in a small interval that doesn't meet $[0,T]$, hence its contribution is clearly trivial for $t\in [0,T]$. We now deal with the second term in \eqref{diruhohinf}: precisely, we shall prove that its essential support in time doesn't meet the interval $[0,T]$ with $T$ given by \eqref{choiceT}, which will imply that its contribution on $\partial\Omega\times [0,T]$ will be $O(h^{\infty})$ (see Remark \ref{rmqchoiceT}). Let us proceed with details: it will be enough to show that 
\begin{equation}\label{contraT}
T\notin \text{ess-supp}(J(Tr_{-}(u^{N}_{F,h}))).
\end{equation}
We will argue by contradiction and assume that \eqref{contraT} holds true. Then,  according to \eqref{ineg1}, on the support of $g^N_h$ we must have for some $y_T\in [0,Y]$ 
\begin{equation}\label{inegT1}
\Big|\frac{(\partial_{\tau}\theta_0+\xi\partial_{\tau}\zeta_0)(y_T,T,1,-(1+a)^{1/2})}{2(1+a)^{1/2}a^{1/2}}-2n\Big|\\ \leq c_0.
\end{equation}
Using that \eqref{condlagra} must hold on $\Lambda_N$ and since we consider here the contribution on the boundary corresponding to the "minus" sign (i.e corresponding to $\xi<0$), we deduce that in \eqref{inegT1} we must have
\[
\xi=-(-\zeta_0)^{1/2}(1,-(1+a)^{1/2})=-a.
\]
Replacing this in \eqref{inegT1} and using again $\partial_{\tau}\zeta_0=2(1+a)^{1/2}+O(a)$ we find that 
\begin{equation}\label{tauT1}
\partial_{\tau}\theta_0(y_T,T,1,-(1+a)^{1/2})\in 2a^{1/2}(1+a)^{1/2}[2N+1-c_0,2N+1+c_0].
\end{equation}
From Proposition \ref{propsuppunhx} we also must have
\begin{multline}\label{etaT1}
\partial_{\eta}\theta_0(y_T,T,1,-(1+a)^{1/2})=(1+a)^{1/2}\partial_{\tau}\theta_0(y_T,T,1,-(1+a)^{1/2})\\-\frac 43Na^{3/2}-\frac 23(-\zeta_0)^{3/2}(1,-(1+a)^{1/2}).
\end{multline}
In order to show that this is not possible we come back to the choice of $T$ in the end of Section \ref{sectNt}. Precisely, using \eqref{choiceT} and \eqref{choiceYNT} together with the estimate \eqref{estcaustic} on the defining function $C$ of the caustic set which holds uniformly for $y\in [0,Y]$, yields
\begin{equation}\label{tauN1}
\partial_{\tau}\theta_0(y_N,T,1,-(1+a)^{1/2})=\bar{t}_N+O(a),
\end{equation}
\begin{equation}\label{etaN1}
\partial_{\eta}\theta_0(y_N,T,1,-(1+a)^{1/2})=\bar{y}_N+O(a),
\end{equation}
where we recall that $\bar{T}_N=4Na^{1/2}(1+a)^{1/2}$ and $\bar{y}_N=(1+a)^{1/2}\bar{t}_N-\frac 43 Na^{3/2}$. Notice also that for $N$ defined by \eqref{condnnn} the contribution of $Na^{3/2}$ is $O(a)$. 

Using \eqref{tauT1}, \eqref{tauN1}, \eqref{etaT1} and \eqref{etaN1} gives
\begin{multline}\label{contraegalTN}
|\partial_{\eta}\theta_0(y_T,T,1,-(1+a)^{1/2})-\partial_{\eta}\theta_0(y_N,T,1,-(1+a)^{1/2})|=\\(1+a)^{1/2}
|\partial_{\tau}\theta_0(y_T,T,1,-(1+a)^{1/2})-\partial_{\tau}\theta_0(y_N,T,1,-(1+a)^{1/2})|+O(a)\\
\in [2a^{1/2}(1+a)^{1/2}(1-c_0),2a^{1/2}(1+a)^{1/2}(1+c_0)],
\end{multline}
(where the inclusion in the last line follows from \eqref{tauT1} and \eqref{tauN1}).
Evaluating the term in the first line of \eqref{contraegalTN} yields
\begin{multline}
|\partial_{\eta}\theta_0(y_T,T,1,-(1+a)^{1/2})-\partial_{\eta}\theta_0(y_N,T,1,-(1+a)^{1/2})|\\
=|y_T-y_N|\times \int_0^1 \partial^2_{y,\eta}\theta_0(\alpha y_T+(1-\alpha)y_N,T,1,-(1+a)^{1/2})d\alpha\\
=|y_T-y_N|\times\Big( \int_0^1b^{2/3}(\alpha y_T+(1-\alpha)y_N)d\alpha+O(a)\Big)\\
\geq \frac 34 |y_T-y_N|,
\end{multline}
where in third line we used the approximation \eqref{thetety} and in the last line we used that $y_T,y_N\in [0,Y]$ where $|b^{2/3}-1|<\frac 14$ and therefore $b^{2/3}(y)\geq \frac 34$ for $y\in [0,Y]$. The term in the second line of \eqref{contraegalTN} is estimated similarly by
\begin{multline}
|\partial_{\tau}\theta_0(y_T,T,1,-(1+a)^{1/2})-\partial_{\tau}\theta_0(y_N,T,1,-(1+a)^{1/2})|\\
=|y_T-y_N|\times \int_0^1 \partial^2_{y,\tau}\theta_0(\alpha y_T+(1-\alpha)y_N,T,1,-(1+a)^{1/2})d\alpha\\
=|y_T-y_N|\times \Big(\int_0^1|b^{2/3}(\alpha y_T+(1-\alpha)y_N)-1|d\alpha+O(a)\Big)\\
\leq \frac 14 |y_T-y_N|,
\end{multline}
where we took advantage this time of \eqref{thettay} and that  $|b^{2/3}(y)-1|<\frac 14$ for $y\in [0,Y]$.
Using \eqref{contraegalTN} and the last two inequalities yields:
\begin{equation}
\frac 34|y_T-y_N|\leq 2a^{1/2}(1+a)^{1/2}(1+c_0)
\end{equation}
and
\begin{equation}
\frac 14|y_T-y_N|\geq 2a^{1/2}(1+a)^{1/2}(1-c_0),
\end{equation}
therefore we find 
\begin{equation}
3(1-c_0)\leq 1+c_0,
\end{equation}
hence we must have $c_0\geq \frac 12$ which is a contradiction since in Section \ref{secmodelconstruction} we have chosen $0<c_0\leq \frac 38$. Therefore \eqref{contraT} can't hold and in the same way we can see that $\text{ess-supp}(J(Tr_{-}(u^N_{F,h}))$ doesn't meet the interval $[0,T]$. 
\end{proof}

\subsubsection{Strichartz estimates for the approximate solution $U_{h}$}\label{sectstri}
\begin{prop}
Let $r>4$ and $\epsilon>0$ be the one fixed in Section \ref{secdoide}. We define
\begin{equation}
\beta(r)=\frac{3}{2}(\frac{1}{2}-\frac{1}{r})+\frac{1}{6}(\frac{1}{4}-\frac{1}{r})
\end{equation}
and let $\beta\leq\beta(r)-\epsilon$. Then the approximate solution $U_{h}$ of the wave equation \eqref{wavvv} satisfies 
\begin{equation}\label{stricon}
h^{\beta}\|U_{h}\|_{L^{q}([0,T],L^{r}(\Omega))}\geq h^{-7\epsilon/8}\|U_{h}|_{t=0}\|_{L^{2}(\Omega)}\gg
\|U_{h}|_{t=0}\|_{L^{2}(\Omega)}.
\end{equation}
\end{prop}
\begin{rmq}
Notice that the condition $\beta<\beta(r)$ shows that $U_{h}$ can't satisfy the Strichartz inequalities of the free case, a loss of at least $\frac{1}{6}(\frac{1}{4}-\frac{1}{r})$ derivatives being unavoidable.
\end{rmq}
\begin{proof}
We estimate from below the $L^{q}([0,T],L^{r}(\Omega))$ norm of $U_{h}$ using Proposition \ref{propnorm} from the Appendix. The key point here is that $u^{n}_{h}$ have almost disjoint supports in time and in the tangential variable, hence we can bound from below the $L^{q}([0,T])$ norm by a sum of integrals over small intervals of time $J_{k}$ on which there will be only one cusp, $u^{k}_{h}$ to consider, the contribution from all the others being trivial. The intervals $J_{k}$ will be the ones defined in \eqref{defjncontin} for which Lemma \ref{lemlengsup} applies.
\begin{align}\label{bonor}
\nonumber
\|U_{h}\|^{q}_{L^{q}([0,T],L^{r}(\Omega))}=\int_{0}^{T}\|U_{h}(.,t)\|^{q}_{L^{r}(\Omega)}dt =
\int_{0}^{T}\|\sum_{n=0}^{N}u^{n}_{h}(.,t)\|^{q}_{L^{r}(\Omega)}dt \\
\nonumber
\geq\sum_{k\leq N}\int_{t\in J_{k}}\|\sum_{n=0}^{N}u^{n}_{h}(.,t)\|^{q}_{L^{r}(\Omega)}dt +O(h^{\infty})\\ 
\nonumber
\simeq \sum_{k\leq N}|J_{k}|\|u^{0}_{h}(.,0)\|^{q}_{L^{r}(\Omega)} +O(h^{\infty})\\
\simeq \frac{c_0}{4}Y \|u^{0}_{h}(.,0)\|^{q}_{L^{r}(\Omega)}+O(h^{\infty}).
\end{align}

Indeed, we have shown in Lemma \ref{lemdisjsup} that for $t$ belonging to small intervals of time $J_{k}$ there is only $u^{k}_{h}$ to be considered in the sum since the contribution from  each $u^{n}_{h}$ with $n\neq k$ is $O_{L^{2}}(h^{\infty})$. In the last line of \eqref{bonor} we have used Lemma \ref{lemlengsup} to estimate from below $|J_{k}|$, uniformly in $k$, by $c_{0}a^{1/2}$, where $c_{0}\in (0,3/8]$ is fixed, and the fact that $N\simeq \frac{Y}{4} a^{-1/2}$.

On the other hand, for $t\in J_{k}$, the "piece" of cusp $u^{k}_{h}(.,t)$ does not "live" enough to reach the boundary, as it is shown in the last part of Lemma \ref{lemlengsup}. Precisely, from \eqref{xnotbond} it follows that if $t\in J_{k}$ then $x\geq a/2$, therefore on the essential support of $u^{k}_{h}(.,t)$ the normal variable doesn't approach the boundary. This means that the restrictions of $u^{k}_{h}$ to $J_{k}$ have disjoint supports.

Moreover, we see from Proposition \ref{propnorm} that for $t\in J_{k}$ the $L^{r}(\Omega)$ norms of $u^{k}_{h}(.,t)$ are all equivalent to the $L^{r}(\Omega)$ norm of $u^{0}_{h}(.,0)$. Using Proposition \ref{propnorm} and \eqref{defa} we deduce that there are constants $C=C(Y)$ independent of $h$ such that for $r=2$ 
\begin{align}\label{estnorm21}
\|U_{h}|_{t=0}\|_{L^{2}(\Omega)}& \simeq h\|\partial_{t}U_{h}|_{t=0}\|_{L^{2}(\Omega)}\\
\nonumber
&\simeq \|u^{0}_{h}(.,0)\|_{L^{2}(\Omega)}\\
\nonumber
&\simeq
ha^{1/4}\simeq Y^{1/8}h^{1+\frac{1-\epsilon}{8}}.
\end{align}
while for $r>4$ we get, using also \eqref{bonor},
\begin{equation}\label{estnorm3}
\|U_{h}\|_{L^{q}([0,T],L^{r}(\Omega))}\geq C(Y)
h^{\frac{1}{3}+\frac{5}{3r}}, \ \text{where}\ C(Y)=(\frac{c_0}{4}Y)^{1/q}.
\end{equation}
We deduce that for $\beta\leq\beta(r)-\epsilon$ the following holds
\begin{align}
\nonumber
h^{\beta}\|U_{h}\|_{L^{q}([0,T],L^{r}(\Omega))} 
& \geq C(Y) h^{\beta(r)-\epsilon}h^{\frac{1}{3}+\frac{5}{3r}}\\
\nonumber
& =C(Y) h^{-7\epsilon/8+1+(1-\epsilon)/8} \\
\nonumber
& \simeq (C(Y)Y^{-1/8}) h^{-7\epsilon/8}\|u_{h}(.,0)\|_{L^{2}}\\ 
 & \gg (\|U_{h}|_{t=0}\|_{L^{2}(\Omega)}+h\|\partial_{t}U_{h}|_{t=0}\|_{L^{2}}),
\end{align}
where we recall that $Y$ was fixed, depending on $b$ and, hence, on $\Omega$ only.
\begin{rmq}
Using Proposition \ref{propnorm} we can also estimate the $L^{r}$ norms for $2\leq r<4$,\begin{equation}\label{estimnorm11}
\|U_{h}\|_{L^{q}([0,T],L^{r}(\Omega))}\geq C
h^{\frac{1}{r}+\frac{1}{2}+\frac{1-\epsilon}{2}(\frac{1}{r}-\frac{1}{4})},
\end{equation}
for some constant $C$ depending on $\Omega$ only. 
Notice however that in this case there is no contradiction when comparing \eqref{estimnorm11} to the 
usual Strichartz inequalities of the free case \eqref{striw} recalled in Proposition \ref{proprd}.
\end{rmq}
\end{proof}

\subsubsection{End of the proof of Theorem \ref{thms2}}

We can now  achieve the proof of Theorem \ref{thms2}. Let $\epsilon>0$ be the one fixed in Section \ref{secdoide} above and $N$ be given by \eqref{condaN}. Consider the $L^{2}$ normalized approximate solution to \eqref{wavvv} 
\[
v^{n}_{h,\epsilon}(x,y,t):=\frac{1}{\|U_{h}(.,0)\|_{L^{2}(\Omega)}}u^{n}_{h}(x,y,t),
\]
and set
\[
\tilde{V}_{h,\epsilon}(x,y,t):=\sum_{n=0}^{N}v^{n}_{h}(x,y,t)=\frac{1}{\|U_{h}(.,0)\|_{L^{2}(\Omega)}}U_{h}(x,y,t).
\]
We claim that $\tilde{V}_{h,\epsilon}$ $v^{n}_{h}$ satisfy the conditions of Theorem \ref{thms3}. Notice that this would achieve the proof of Theorem \ref{thmstrichartz}, since in Section \ref{secred} we showed that matters can be reduced to proving Theorem \ref{thms3}.
Indeed, it follows from Proposition \ref{propnorm} that for $4<r<\infty$, $v^{n}_{h,\epsilon}$ satisfy
\begin{equation}  
\left\{
\begin{array}{ll}
\|v^{n}_{h,\epsilon}(.,t)\|_{L^{r}(\Omega)}\geq C h^{-\frac{3}{2}(\frac{1}{2}-\frac{1}{r})-\frac{1}{6}(\frac{1}{4}-\frac{1}{r})+2\epsilon},\quad\text{for} \quad t\in J_{n}\\
\sup_{\epsilon>0}\|v^{n}_{h,\epsilon}(.,t)\|_{L^{2}(\Omega)}\leq 1,
\end{array}
\right.
\end{equation}
where in order to bound uniformly the $L^{2}$ norms we use the fact that for $t\in J_{n}$ and $0\leq n\leq N$
\begin{equation}
\|u^{n}_{h}(.,t)\|_{L^{2}(\Omega)}\simeq \|u^{0}_{h}(.,t)\|_{L^{2}(\Omega)}\simeq\|u^{0}_{h}(.,0)\|_{L^{2}(\Omega)}=\|U_{h}|_{t=0}\|_{L^{2}(\Omega)}.
\end{equation}
From Proposition \ref{propsuppunhx}, the cusps $v^{n}_{h,\epsilon}$ have almost disjoint essential supports in the time and tangential variable and for the normal variable in an interval of size $a\simeq \frac 12 Y^{1/2}h^{\frac{1-\epsilon}{2}}$. Moreover, the approximate solution $\tilde{V}_{h,\epsilon}$ is localized at spatial frequency $1/h$ and satisfies 
\begin{equation}
\|\tilde{V}_{h,\epsilon}\|_{L^{2}(\Omega)}\lesssim 1,\quad \|\partial_{y}\tilde{V}_{h,\epsilon}\|_{L^{2}(\Omega)}\lesssim\frac{1}{h},\quad \|\partial^{2}_{y}\tilde{V}_{h,\epsilon}\|_{L^{2}(\Omega)}\lesssim\frac{1}{h^{2}},
\end{equation}
with constants independent of $\epsilon$,
which follows from the spectral localization together with the uniform bounds of the derivatives of $g^{n}_{h}$ with respect to $y$. 
From Proposition \ref{propnorml2square} and the almost orthogonality property of the supports in $y$ we also obtain
\begin{equation}
\square_{g} \tilde{V}_{h,\epsilon}=O_{L^{2}(\Omega)}(1/h).
\end{equation}
Finally, Proposition \ref{propdiruh01} assures that the Dirichlet boundary condition is satisfied by the restriction of $\tilde{V}_{h,\epsilon}$ to the interval of time $[0,T]$:
\begin{equation}
\tilde{V}_{h,\epsilon}|_{\partial\Omega\times [0,T]}=O(h^{\infty}).
\end{equation}

\section{Appendix}
In this section we compute the $L^{r}$ norms of the phase integrals associated to a cusp type
Lagrangian. We prove the following:
\begin{prop}\label{propnorm}
For $t\in J_{n}$ defined in \eqref{defjncontin} the $L^{r}(\Omega)$ norms of a cusp $u^{n}_{h}(.,t)$ of the form \eqref{uhndef} satisfy, uniformly for $n\in\{0,..,N\}$,
\begin{itemize}
\item 
for $2\leq r<4$ 
\begin{equation}\label{estimnorm111}
\|u^{n}_{h}(.,t)\|_{L^{r}(\Omega)}\simeq
h^{\frac{1}{r}+\frac{1}{2}}a^{\frac{1}{r}-\frac{1}{4}},
\end{equation}
\begin{equation}\label{estnorm211}
\|u^{n}_{h}(.,0)\|_{L^{2}(\Omega)}\simeq
ha^{1/4};
\end{equation}
\item for $r>4$ 
\begin{equation}\label{estnorm3}
\|u^{n}_{h}(.,t)\|_{L^{r}(\Omega)}\simeq
h^{\frac{1}{3}+\frac{5}{3r}}.
\end{equation}
\end{itemize}
\end{prop}

\begin{proof}
To estimate the $L^{r}(\Omega)$ norms of $u^{n}_{h}$,
we must distinguish three cases:
\begin{itemize}
\item If $-\zeta(x,y,1,-(1+a)^{1/2})\leq Mh^{2/3}$
for some constant $M\geq 1$, we make the changes of variables 
\[
h^{2/3}X=-\zeta(x,y,1,-(1+a)^{1/2}),
\]
which, by the implicit function theorem yields the smooth function $x=x(X,y)$ and
\[
Y(X,y,t)= \partial_{\eta}\theta(x,y,t,1,-(1+a)^{1/2})\\ -(1+a)^{1/2}\partial_{\tau}\theta(x,y,t,1,-(1+a)^{1/2})+\frac{4}{3}na^{3/2},
\]
for $x=x(X,y)$. The Jacobian of the last application can be easily estimated
\begin{align}
\nonumber
\frac{dY}{dy}(X,y,t) & =\Big(\partial^{2}_{y,\eta}\theta_0 -(1+a)^{1/2}\partial^{2}_{y,\tau}\theta_0\Big)(x(X,y),y,t,1,-(1+a)^{1/2})+O(x(X,y))\\
\nonumber
& =b^{2/3}(y)-(1+a)^{1/2}(b^{2/3}(y)-1)+O(a)+O(x(X,y))\\
& =1+O(a)+O(x(X,y)).
\end{align}
Write $y=y(X,Y,t)$ and consider the change of variables $\xi\rightarrow h^{1/3}\xi$. We set
\[
Q(X,u):=\frac{\xi^{3}}{3}-X\xi.
\]
For $\beta:\mathbb{R}\rightarrow [0,1]$ we also define
\begin{multline}
f^{n}_{\beta}(X,Y,t,\eta,h):=\int e^{i\eta
Q(X,\xi)}\beta(\xi)\\ \times g^{n}_{h}(\frac{\partial_{\tau}\theta(x(X,y),y(X,Y,t),t,1,-(1+a)^{1/2})}{2a^{1/2}(1+a)^{1/2}}+h^{1/3-\alpha/2}\xi-2n,y(X,Y,t),\eta,\lambda)d\xi.
\end{multline}
We introduce
\begin{equation}
F^{n}_{\beta}(X,Y,t,h):=\int e^{i\eta Y/h}f^{n}_{\beta}(X,Y,t,\eta,h)d\eta,
\end{equation}
and we make integrations by parts with respect to $\eta$ in order to compute
\begin{equation}
Y^{p}F^{n}_{\beta}(X,Y,t,h)=(ih)^{p}\int e^{i\eta Y/h}\partial^{p}_{\eta}f^{n}_{\beta}(X,Y,t,\eta,h)d\eta.
\end{equation}
The derivatives of $f_{\beta}^{n}$ are estimated using the precise form of the symbol $g^{n}_{h}$. Recall that 
\begin{equation}\label{sigmndefapp}
g^{n}_{h}(z,y,\eta)
\simeq
\Psi(\eta)\eta^{1/3}\Big(\sum_{k\geq 0} h^{k/2}a^{-k/2}\mu_{k}(y,\eta,h)\partial^{k}_{z}\varrho^{n}(z,\eta,\lambda)\Big),
\end{equation}
where 
\begin{equation}\label{convprodnvr}
\partial^{k}_{z}\varrho^{n}(z,\eta,\lambda)=(F_{\eta\lambda})^{*n}*\partial^{k}_{z}\varrho^{0}(.,\lambda)(z),\quad \forall k\geq 0,
\end{equation}
with $\varrho^{0}(.,\lambda)\in\mathcal{S}_{K_{0}}(\lambda)$ independent of $\eta$ and where $(F_{\eta\lambda})^{*n}$ is defined in \eqref{Fn}. The derivatives of $(F_{\eta\lambda})^{*n}$ with respect to $\eta$ were estimated in the proof of Lemma \ref{lemlagrawfset} and are bounded uniformly in $n$. On the other hand, $\mu_{k}$ are symbols of order $0$ and type $(1,0)$ hence all the derivatives $\partial^{p-l}_{\eta}(\eta^{1/3}\Psi(\eta)\mu_{k})(y,\eta,h))$ are bounded (on the support of $\Psi(\eta)$) by constants $C_{p-l}$. 
We estimate $\partial^{p}_{\eta}f^{n}_{\beta}$ as follows
\begin{multline}
\partial^{p}_{\eta}f^{n}_{\beta}(X,Y,t,\eta,h)
=\int e^{i\eta Q(X,\xi)}\beta(\xi)\sum_{l=0}^{p}C_{p}^{l}(iQ(X,\xi))^{p-l}\\
\times \partial^{l}_{\eta}g^{n}_{h}(\frac{\partial_{\tau}\theta(x(X,y),y(X,Y,t),t,1,-(1+a)^{1/2})}{2a^{1/2}(1+a)^{1/2}}+h^{1/3-\alpha /2}\xi-2n,\eta,\lambda) d\xi.
\end{multline}
To estimate $\partial^{l}_{\eta}g^{n}_{h}$ we use \eqref{sigmndefapp}. We first take $\beta(\xi):=1_{|\xi|\leq \sqrt{1+M}}$ and estimate the $L^{\infty}$ norms of $w^{p}F^{n}_{\beta}$ in this case, in order to use the multiplier's theorem to bound the $L^{r}$ norms of $u^{n}_{h}$. We have
\begin{multline}\label{bau}
\|(\frac{Y}{h})^{p}F^{n}_{\beta}(X,Y,t,h)\|_{L^{\infty}_{Y}}
\leq \sum_{k\geq 0}h^{k/2}a^{-k/2} \sum_{l=0}^{p}C(p,l)\sup_{|u|\leq\sqrt{1+M}}|Q(X,u)|^{l-j}\\
\times\sum_{j=0}^{l}\int_{\eta} \Big|\partial^{l-j}_{\eta}\Big(\Psi(\eta)\eta^{1/3}\mu_{k}(y,\eta,h)\Big)\Big|\\\times \Big|\partial^{j}_{\eta}\partial^{k}_{z}\varrho^{n}(\frac{\partial_{\tau}\theta(x(X,y),y(X,Y,t),t,1,-(1+a)^{1/2})}{2a^{1/2}(1+a)^{1/2}}+h^{1/3-\delta/2}u-2n,\eta,\lambda)\Big|d\eta ,
\end{multline}
therefore $\|(\frac{Y}{h})^{p}F^{n}_{\beta}(X,Y,t,h)\|_{L^{\infty}_{Y}}\leq C_{p,M}$
for some constants $C_{p,M}$ depending only of $p$ and $M$, where we estimated the derivatives of $\partial^{k}_{z}\varrho^{n}$ with respect of $\eta$ using the formula \eqref{convprodnvr} and the proof of Lemma \ref{lemlagrawfset}.

Since we need an uniform bound of $\|W^{p}F_{1}^{n}(X,hW,t,h)\|_{L^{\infty}}$ we have to consider also the case $1-\beta$. On the support of $(1-\beta)$ we have $\sqrt{1+M}\leq |\xi|\lesssim a^{1/2}h^{-1/3}$, and the contribution of the integral defining $u^n_h(.,t)$ is $O_{L^{2}}(h^{\infty})$. Indeed, in this case $|\xi^{2}-X|\geq 1$ and we conclude by repeated integrations by parts like in the proof of Lemma \ref{lemlagrawfset}.

As a result we obtain
\begin{multline}\label{estnorm}
\|u^{n}_{h}(.,t)\|^{r}_{L^{r}(-\zeta(x,y,1,-(1+a)^{1/2})\leq
Mh^{2/3},y)}\\
= h^{r/3}\int_{|X|\leq h^{2/3}}|F^{n}_{1}(X,Y,t,h)|^{r}\frac{dx(X,y)}{dX}\frac{dy(X,Y,t)}{dY}dYdX
\\
\simeq h^{5/3+r/3} \|F^{n}_{1}(X,hW,t,h)\|^{r}_{L^{r}(|X|\leq M,W)}\\\simeq
h^{5/3+r/3},
\end{multline}
where in the second line we used the estimates \eqref{bau}, while in the last line in \eqref{estnorm} we made the change of variables $Y=hW$ and we used the fact that $\mu_{0}$ is elliptic on the essential support of $u^n_h$ (this last fact follows from Proposition \ref{propnnn} and the choice of $N$ in \eqref{condnnn}).

\item If $-\zeta(x,y,1,-(1+a)^{1/2})\in(Mh^{2/3},\frac 32 a]$, where $M\gg1$ is large enough, we apply the stationary phase lemma.
\begin{prop}\label{thmphasestat}(\cite[Thm.7.7.5]{hormand})
Let $K\subset\mathbb{R}$ be a compact set, $f\in C^{\infty}_{0}(K)$,
$\phi\in C^{\infty}(\mathring{K})$ such that $\phi(0)=\phi'(0)=0$,
$\phi''(0)\neq 0$, $\phi'\neq 0$ in
$\mathring{K}\setminus 0\}$. Let $\omega\gg 1$, then for every
$k\geq 1$ we have
\begin{equation}
\int e^{i\omega\phi(u)}f(u)du\simeq \frac{(2\pi i
)^{1/2}e^{i\omega\phi(0)}}{(\omega\phi''(0))^{1/2}}\sum_{j\geq 0}\omega^{-j}L_{j}f.
\end{equation}
Here $C$ is bounded when $\phi$ stays in a bounded set in $C^{\infty}(\mathring{K})$, $|u|/|\phi'(u)|$ has a uniform bound and
\begin{equation}\label{hormand}
L_{j}f=\sum_{\nu-\mu=j}\sum_{2\nu\geq3\mu}\frac{i^{-j}2^{-\nu}}{\mu!\nu!}(\phi''(0))^{-\nu}\partial^{2\nu}
(\kappa^{\mu}f)(0).
\end{equation}
where
$\kappa(u)=\phi(u)-\phi(0)-\frac{\phi''(0)}{2}u^{2}$ vanishes of third order at $0$. 
\end{prop}
We make the change of variable $\xi=(-\zeta)^{1/2}(x,y,1,-(1+a)^{1/2})(\pm 1+u)$ to compute
the integral in $\xi$ in the formula defining $U^n_h(.,t)$. Recall that it writes
\begin{multline}\label{ispm}
\sum_{k\geq 0}h^{k/2}a^{-k/2}\Psi(\eta)\eta^{1/3}\mu_{k}(y,\eta,h)\\
\times \int e^{\frac{i}{h}\eta(\frac{\xi^{3}}{3}+\xi\zeta(x,y,1,-(1+a)^{1/2})}\partial^{k}_{z}\varrho^{n}(z+\xi/a^{1/2}-2n,\eta,\lambda)d\xi,
\end{multline}
where we set $z=\frac{\partial_{\tau}\theta(x,y,t,1,-(1+a)^{1/2})}{2a^{1/2}(1+a)^{1/2}}$ and where $\partial_z$ denotes the derivative in the first variable.
Applying Proposition \ref{thmphasestat} with 
\[
\omega:=(-\zeta)^{3/2}(x,y,1,-(1+a)^{1/2})/h\gg 1,
\]
\[
\phi_{\pm}(u):=\frac{u^{3}}{3}\pm u^{2},\quad 
\kappa_{\pm}(u):=u^{3}/3,
\]
the integrals in $\xi$ in \eqref{ispm} write, for each $k\geq 0$,
 \begin{multline}
(h\pi)^{1/2}\eta^{-2/3}(-\zeta)^{-1/4}e^{\mp\frac{2}{3}i\eta(-\zeta)^{3/2}/h
\pm\frac{i\pi}{4}}
\sum_{j\geq 0}h^{j}(-\zeta)^{-3j/2}\eta^{-j}\\ \times L_{j}\Big(\partial^{k}_{z}\varrho^{n}(z+(-\zeta)^{1/2}(\pm1+u)/a^{1/2}-2n,\eta,\lambda)\Big)_{\Big|_{u=0}},
\end{multline}
where $\zeta=\zeta(x,y,1,-(1+a)^{1/2})$.
Since $\partial^{k}_{z}\varrho^{n}$ writes as the convolution product \eqref{convprodnvr}, we have, for $z$ as above,
\begin{multline}\label{formunhalmfin}
u^{n}_{h}(x,y,t)=(\pi h)^{1/2}\sum_{k\geq 0}h^{k/2}a^{-k/2}\sum_{j\geq 0}h^{j}(-\zeta)^{-3j/2-1/4}L_{j}(\partial^{k}_{z}\varrho^{0})(.)*\\
*\Big(\int_{\eta}e^{\frac{i}{h}\eta(\theta(x,y,t,1,-(1+a)^{1/2})+\frac{4}{3}na^{3/2}\mp\frac{2}{3}(-\zeta)^{3/2}(x,y,1,-(1+a)^{1/2}))}\\\times \frac{\Psi(\eta)}{\eta^{j+1}}\mu_{k}(y,\eta,h)(F_{\eta\lambda})^{*n}d\eta\Big)_{\Big|_{(z\pm(-\zeta)^{1/2}(x,y,1,-(1+a)^{1/2})/a^{1/2}-2n)}}.
\end{multline}
We set
\[
F^{n,k,j}(.,y,\eta,h):=\frac{\Psi(\eta)}{\eta^{j+1}}\mu_{k}(y,\eta,h)(F_{\eta\lambda})^{*n}(.).
\]
Since $\Psi(\eta)$ is compactly supported for $\eta$ in a neighborhood of $1$, the Fourier transform $\widehat{F^{n,k,j}}$ with respect to $\eta$ of each $F^{n,k,j}$ is rapidly decreasing and the integral in $\eta$ in \eqref{formunhalmfin} becomes
\[
\widehat{F^{n,k,j}}\Big(.,y,\frac{(\theta(x,y,t,1,-(1+a)^{1/2})+\frac{4}{3}na^{3/2}\mp\frac{2}{3}(-\zeta)^{3/2}(x,y,t,1,-(1+a)^{1/2}))}{h},h\Big)
\]
We perform again the changes of variables $x\rightarrow x(X,y)$, $y\rightarrow y(X,Y,t)$ defined in the first part of the proof and take also $Y=hW$. Then \eqref{formunhalmfin} becomes
\begin{equation}
u^{n}_{h}(x,y,t):=\pi^{1/2}\sum_{k,j\geq 0}h^{k/2}a^{-k/2}u^{n,k,j}_{h}(x,y,t),
\end{equation}
where we set, for $z=\frac{\partial_{\tau}\theta(x(X,y),y(X,hW,t),t,1,-(1+a)^{1/2})}{2a^{1/2}(1+a)^{1/2}}$,
\begin{multline}
u^{n,k,j}_{h}(x,y,t):=h^{1/3}X^{-1/4-3j/2}L_{j}(\partial^{k}_{z}\varrho^{0}(.,\lambda))*\\
*\widehat{F^{n,k,j}}(.,y(X,hW,t),W\mp\frac{2}{3}X^{3/2},h)|_{(z+h^{1/3-\alpha/2}X^{1/2}-2n)}.
\end{multline}
We have to distinguish another two situations.
\begin{enumerate}
\item If $r>4$ then a simple computation shows that for $k\geq 0$ the $L^{r}$ norms of each $u^{n,k,j}$ can be estimated from above as follows
\begin{multline}
\|u^{n,k,j}_{h}(.,t)\|^{r}_{L^{r}((-\zeta)(x,y,t,1,-(1+a)^{1/2})\in(Mh^{2/3},\frac 32 a],y)} \\ \lesssim h^{r(1/2+j+5/3r-1/6-j)}\int_{M}^{\frac 32 ah^{-2/3}}X^{-r(1/4+3j/2)}dX\\
 \lesssim h^{r/3+5/3}\frac{M^{1-r(1/4+3j/2)}}{(r(1/4+3j/2)-1)},
\end{multline}
and since the operators $L_{j}$ are of order $2j$, for each $j$ there will be $2j$ terms in the sum  $\sum_{j}u^{n,k,j}_{h}$: summing up over $j\geq 0$ (taking $M\geq 2$ for example) and using that  $\varrho^{n}\in\mathcal{S}_{K_{0}}(\lambda/n)$ for $n\geq 1$, $\varrho^{0}\in\mathcal{S}_{K_{0}}(\lambda)$  (fact that assures uniform bounds for the derivatives $\partial^{k}_{z}\varrho^{n}$ for each $n\geq 0,k\geq 0$), yields
\[
\|\sum_{j\geq 0}u^{n,k,j}_{h}(.,t)\|_{L^{r}((-\zeta)(x,y,1,-(1+a)^{1/2})\in(Mh^{2/3},\frac 32 a],y)}\lesssim C(r)h^{r/3+5/3}, \quad C(r)=\frac{1}{r/4-1}.
\]
On the other hand
\begin{multline}
\|u^{n}_{h}(.,t)\|_{L^{r}((-\zeta)(x,y,1,-(1+a)^{1/2})\in(Mh^{2/3},\frac 32 a],y)}\\
\lesssim \sum_{k\geq 0} h^{k(1-\delta/2)}\|\sum_{j\geq 0}u^{n,k,j}_{h}(.,t)\|_{L^{r}((-\zeta)(x,y,1,-(1+a)^{2/3})\in(Mh^{2/3},\frac 32 a],y)}.
\end{multline}
For $k=0$, due to the ellipticity of the symbol $\mu_{0}(y,\eta,h)=a_{h}(y,\eta,-\eta(1+a)^{1/2})$ we can estimate also from below the $L^{r}$ norm of $u^{n,0,j}_{h}(.,t)$ by $C(r)h^{1/3+5/3r}$ and consequently we deduce
\[
\|u^{n,0}_{h}(.,t)\|_{L^{r}((-\zeta)(x,y,1,-(1+a)^{1/2})\in(Mh^{2/3},\frac 32 a],y)}\simeq C(r)h^{r/3+5/3}.
\]
Hence \eqref{estnorm211} follows. 

\item We deal now with the case $r\in [2,4)$. In this case we do not expect any contradiction to the usual Strichartz estimates, therefore we shall compute the $L^{2}$ norms of $\sum_{j\geq 0}u^{n,k,j}_{h}(.,t)$ only to prove \eqref{estnorm211}. We let \eqref{estimnorm111} as an exercise, since in the proof of Theorem \ref{thmstrichartz} we do not use this estimate. 

For $j=0$ we compute, as before
\[
\int_{M}^{\frac 32 ah^{-2/3}}X^{-1/2}dX\simeq 2(\frac 32 ah^{-2/3})^{1/2},
\]
while for $j\geq 1$ we have $2(1/4+3j/2)-1>0$ and
\[
\int_{M}^{\frac 32 ah^{-2/3}}X^{-2(1/4+3j/2)}dX=-\frac{X^{1-2(1/4+3j/2)}\Big|^{\frac 32 ah^{-2/3}}_{M}}{2(1/4+3j/2)-1}\simeq \frac{M^{1/2-3j}}{3j-1/2}.
\]
For $M\geq 2$ the sum of $\|u^{n,k,j}_{h}(.,y,t)\|^{2}_{L^{2}((-\zeta)(x,y,1,-(1+a)^{1/2})\in(Mh^{2/3},\frac 32 a],y)}$ over $j\geq 1$ (where for each $j$ we count $2j$ terms which appear in the expression of $L_{j}(\partial^{k}_{z}\varrho^{n}))$ is small enough compared to $\|u^{n,k,0}_{h}(.,y,t)\|^{2}_{L^{2}((-\zeta)(x,y,1,-(1+a)^{1/2})\in(Mh^{2/3},\frac 32 a],y)}$, while for $k=0$ we can estimate also from below, as before
\begin{align}
\nonumber
\|\sum_{j\geq 0}u^{n,0,j}_{h}(.,t)\|_{L^{2}((-\zeta)(x,y,1,-(1+a)^{1/2})\in(Mh^{2/3},\frac 32 a],y)} & \simeq h^{1/3+5/6}(\frac 32 ah^{-2/3})^{1/4}\\ & \simeq ha^{1/8}.
\end{align}
We have proved \eqref{estimnorm111} for $r=2$; for $r\in (2,4)$, as already mentioned, we do not give the proof since it is not used in the proof of Theorem \ref{thms3} and since it follows exactly in the same way as for $r=2$.
\end{enumerate}
\item In the last case corresponding to $(-\zeta)(x,y,1,-(1+a)^{1/2})\geq \frac 32 a$ we use Lemma \ref{lemlagrawfset} we obtain that the contribution in each $u^{n}_{h}(.,t)$ is $O_{L^{2}}(h^{\infty})$, since in this case we are localized away from a neighborhood of the Lagrangian $\Lambda_n$.
\end{itemize}
\end{proof}

\bibliography{ondes-bib2}

\begin{thebibliography}{10}

\bibitem{blsmso08}
Matthew~D. Blair, Hart~F. Smith, and Christopher~D. Sogge.
\newblock Strichartz estimates for the wave equation on manifolds with
  boundary.
\newblock {\em Ann. Inst. H. Poincar\'e Anal. Non Lin\'eaire},
  26(5):1817--1829, 2009.

\bibitem{botz}
Jean-Marc Bouclet and Nikolay Tzvetkov.
\newblock On global {S}trichartz estimates for non-trapping metrics.
\newblock {\em J. Funct. Anal.}, 254(6):1661--1682, 2008.

\bibitem{bulepl07}
Nicolas Burq, Gilles Lebeau, and Fabrice Planchon.
\newblock Global existence for energy critical waves in 3-{D} domains.
\newblock {\em J. Amer. Math. Soc.}, 21(3):831--845, 2008.

\bibitem{bupl07}
Nicolas Burq and Fabrice Planchon.
\newblock Global existence for energy critical waves in 3-d domains : Neumann
  boundary conditions.
\newblock {\em American J.of Math.}, 131(6):1715--1742, 2009.

\bibitem{cawe90}
Thierry Cazenave and Fred~B. Weissler.
\newblock The {C}auchy problem for the critical nonlinear {S}chr\"odinger
  equation in {$H\sp s$}.
\newblock {\em Nonlinear Anal.}, 14(10):807--836, 1990.

\bibitem{esk77}
Gregory Eskin.
\newblock Parametrix and propagation of singularities for the interior mixed
  hyperbolic problem.
\newblock {\em J. Analyse Math.}, 32:17--62, 1977.

\bibitem{give85}
J.~Ginibre and G.~Velo.
\newblock The global {C}auchy problem for the nonlinear {S}chr\"odinger
  equation revisited.
\newblock {\em Ann. Inst. H. Poincar\'e Anal. Non Lin\'eaire}, 2(4):309--327,
  1985.

\bibitem{give95}
J.~Ginibre and G.~Velo.
\newblock Generalized {S}trichartz inequalities for the wave equation.
\newblock In {\em Partial differential operators and mathematical physics
  ({H}olzhau, 1994)}, volume~78 of {\em Oper. Theory Adv. Appl.}, pages
  153--160. Birkh\"auser, Basel, 1995.

\bibitem{hormand}
Lars H{\"o}rmander.
\newblock {\em The analysis of linear partial differential operators. {III}},
  volume 274 of {\em Grundlehren der Mathematischen Wissenschaften [Fundamental
  Principles of Mathematical Sciences]}.
\newblock Springer-Verlag, Berlin, 1985.
\newblock Pseudodifferential operators.

\bibitem{doi}
Oana Ivanovici.
\newblock Counterexamples to {S}trichartz estimates for the wave equation in
  domains.
\newblock {\em Mathematische Annalen, to appear}, 2009.
\newblock see {\tt http://www.springerlink.com/content/f8242848k8360107}.

\bibitem{kap91}
L.~V. Kapitanski{\u\i}.
\newblock Estimates for norms in {B}esov and {L}izorkin-{T}riebel spaces for
  solutions of second-order linear hyperbolic equations.
\newblock {\em Zap. Nauchn. Sem. Leningrad. Otdel. Mat. Inst. Steklov. (LOMI)},
  171(Kraev. Zadachi Mat. Fiz. i Smezh. Voprosy Teor. Funktsii. 20):106--162,
  185--186, 1989.

\bibitem{lev90}
L.~V. Kapitanski{\u\i}.
\newblock Some generalizations of the {S}trichartz-{B}renner inequality.
\newblock {\em Algebra i Analiz}, 1(3):127--159, 1989.

\bibitem{ka87}
Tosio Kato.
\newblock On nonlinear {S}chr\"odinger equations.
\newblock {\em Ann. Inst. H. Poincar\'e Phys. Th\'eor.}, 46(1):113--129, 1987.

\bibitem{keta98}
Markus Keel and Terence Tao.
\newblock Endpoint {S}trichartz estimates.
\newblock {\em Amer. J. Math.}, 120(5):955--980, 1998.

\bibitem{gle06}
Gilles Lebeau.
\newblock Estimation de dispersion pour les ondes dans un convexe.
\newblock In {\em Journ\'ees ``\'Equations aux D\'eriv\'ees Partielles''
  (Evian, 2006)}. 2006.
\newblock see {\tt
  http://www.numdam.org/numdam-bin/fitem?id=JEDP\_2006\_\_\_\_A7\_0}.

\bibitem{ls95}
Hans Lindblad and Christopher~D. Sogge.
\newblock On existence and scattering with minimal regularity for semilinear
  wave equations.
\newblock {\em J. Funct. Anal.}, 130(2):357--426, 1995.

\bibitem{mel76}
R.~B. Melrose.
\newblock Equivalence of glancing hypersurfaces.
\newblock {\em Invent. Math.}, 37(3):165--191, 1976.

\bibitem{meta87}
Richard~B. Melrose and Michael~E. Taylor.
\newblock Boundary problems for the wave equations with grazing and gliding
  rays, 1987.

\bibitem{moseso}
Gerd Mockenhaupt, Andreas Seeger, and Christopher~D. Sogge.
\newblock Local smoothing of {F}ourier integral operators and
  {C}arleson-{S}j\"olin estimates.
\newblock {\em J. Amer. Math. Soc.}, 6(1):65--130, 1993.

\bibitem{sm98}
Hart~F. Smith.
\newblock A parametrix construction for wave equations with {$C\sp {1,1}$}
  coefficients.
\newblock {\em Ann. Inst. Fourier (Grenoble)}, 48(3):797--835, 1998.

\bibitem{smso95}
Hart~F. Smith and Christopher~D. Sogge.
\newblock On the critical semilinear wave equation outside convex obstacles.
\newblock {\em J. Amer. Math. Soc.}, 8(4):879--916, 1995.

\bibitem{smso06}
Hart~F. Smith and Christopher~D. Sogge.
\newblock On the {$L\sp p$} norm of spectral clusters for compact manifolds
  with boundary.
\newblock {\em Acta Math.}, 198(1):107--153, 2007.

\bibitem{stta02}
Gigliola Staffilani and Daniel Tataru.
\newblock Strichartz estimates for a {S}chr\"odinger operator with nonsmooth
  coefficients.
\newblock {\em Comm. Partial Differential Equations}, 27(7-8):1337--1372, 2002.

\bibitem{stri77}
Robert~S. Strichartz.
\newblock Restrictions of {F}ourier transforms to quadratic surfaces and decay
  of solutions of wave equations.
\newblock {\em Duke Math. J.}, 44(3):705--714, 1977.

\bibitem{tat02}
Daniel Tataru.
\newblock Strichartz estimates for second order hyperbolic operators with
  nonsmooth coefficients. {III}.
\newblock {\em J. Amer. Math. Soc.}, 15(2):419--442 (electronic), 2002.

\bibitem{vas06}
Andr{\'a}s Vasy.
\newblock Propagation of singularities for the wave equation on manifolds with
  corners.
\newblock {\em Ann. of Math. (2)}, 168(3):749--812, 2008.

\end{thebibliography}

\end{document}